\documentclass[12pt,a4paper,reqno]{amsart}

\usepackage[utf8]{inputenc}
\usepackage[english]{babel}

\usepackage{amssymb}
\usepackage{amsthm}
\usepackage{amsmath}
\usepackage{mathrsfs}
\usepackage{mathtools}
\usepackage{xcolor}
\usepackage{graphicx}
\usepackage{tikz}
\usetikzlibrary{arrows.meta, positioning}
\usepackage{caption}
\usepackage{esint}
\usepackage{stix}
\usepackage{indentfirst}
\usepackage{enumitem}
\usepackage{hyperref}
\hypersetup{
	colorlinks=true,
	linkcolor=blue,
	urlcolor=blue,
	filecolor=blue,
	citecolor=blue,
}
\theoremstyle{plain}
\newtheorem{theorem}{Theorem}
\newtheorem{proposition}{Proposition}[section]
\newtheorem{lemma}[proposition]{Lemma}
\newtheorem{corollary}[proposition]{Corollary}

\newtheorem{remark}{Remark}[section]

\numberwithin{equation}{section}

\begin{document}
	
\title{Complete Rigidity at infinity and Existence of the Levinson Cavity}

\author{Yan Li}
\address{School of Mathematics and Statistics, Ningbo University, Ningbo, Zhejiang, 315000, China.}
\email{liyan5@nbu.edu.cn}

\author{Sitan Lin}
\address{School of Mathematical Sciences, Institute of Natural Sciences, Shanghai Jiao Tong University, Shanghai 200240, China.}
\email{lin\_sitan@sjtu.edu.cn;linst3@alumni.sysu.edu.cn}

\author{Georg S. Weiss}
\address{University of Duisburg-Essen}
\email{georg.weiss@uni-due.de}

\author{Chunjing Xie}
\address{School of Mathematical Sciences,  Ministry of Education Key Laboratory of Scientific and Engineering Computing, CMA-Shanghai, Shanghai Jiao Tong University, Shanghai 200240, China.}
\email{cjxie@sjtu.edu.cn}

	
	\begin{abstract}
		We present a {\em potential theoretic approach reducing the analysis of the asymptotic shape of free surfaces to the 
	        analysis of a precise ordinary differential equation resulting from the reduction process}. Although the approach relies mainly on the principal part of the PDE operator to allow for a representation formula and is thus not restricted to problems of elliptic type, we present it at the clean-cut example of {\em three-dimensional axially symmetric steady incompressible cavity flows}, which are Neumann-type Bernoulli free boundary problems and for which frequency formulas are unknown and, if they do exist, insufficient to yield the very precise asymptotic behavior we prove here.
	        
	        In 1946 Norman Levinson derived by a power-law ansatz with a slowly varying correction a precise formula for the asymptotic shape of such cavities. However his result requires very strong assumptions such that it has remained an open problem for 80 years whether the cavity solutions we know to exist by a result by Garabedian-Lewy-Schiffer \cite{GLS_1952} actually share this asymptotic behavior, or whether at least one solution possessing the Levinson asymptotics exists.
	        
	        Here we answer both questions affirmatively, and we obtain {\em complete rigidity at infinity} of the Levinson solution in the class of axially symmetric solutions, that is, any solution satisfying mild and natural assumptions at the fixed boundary and infinity converges asymptotically to the Levinson profile $(\log r)^{-1/4}\sqrt{r}$. 	\end{abstract}

        \maketitle
    
	\tableofcontents
	
	\section{Introduction}

	\subsection{The Potential Theoretic Reduction to ODE approach}
	
	The analysis of the asymptotic behavior at singularities and at infinity ---the two of which are actually closely related--- in problems with {\em free surface} is important both from the perspective
of natural sciences and from a mathematical perspective.
Examples are the Euler equations with free surface
(including {\em jets} and {\em cavities} in potential flow
but also {\em standing water waves}), the {\em two-phase Stefan problem},
the {\em Muskat problem}, the {\em Mumford-Shah problem}, but also 
singularities less known in the analysis of partial differential equations  like the {\em Taylor cone} singularity in 
the {\em Electro-Hydrodynamic Equations}.

Let us focus here on results analyzing unknown singularities/unknown blow-down profiles rather than results proving the existence of a solution with a certain asymptotic behavior, although both types of methods may be related and complement each other. While drawing on geometric analysis, tremendous progress has been made
analyzing unknown singularities/unknown blow-downs in problems where a {\em monotonicity and/or frequency formula} is available (see for example
\cite{bonnet,efw,figalli2,figalli1,garofalo_lin,PSU,vw1} and references therein),
much less is known in problems where such  
monotonicity or frequency formulas seem unavailable like
the two-phase Stefan problem, standing water waves, the Muskat problem,
the Electro-Hydrodynamic Equations equations etc. One may refer to \cite{SJWu, cordoba,stefan, ehd} and references therein for the progress on these problems.

Moreover, problems in which the free surface is not a {\em level set},
seem to be notoriously hard.

However, {\em potential theoretic methods} can actually help
in these problems, provided that the principle part of the operator
(or at least some part of it) 
has a Green's function of which sufficient information is available.
Our approach draws on Section~4 of the result by Eberle-Figalli-Weiss \cite{efw},
in which the {\em generalized Newtonian potential}
has been applied to a point on the ``asymptotic axis of symmetry''
to obtain for the section of coincidence set and its volume
 the integral inequality, 
from which the authors deduce a sharp growth bound from above of the coincidence set.

\begin{center}
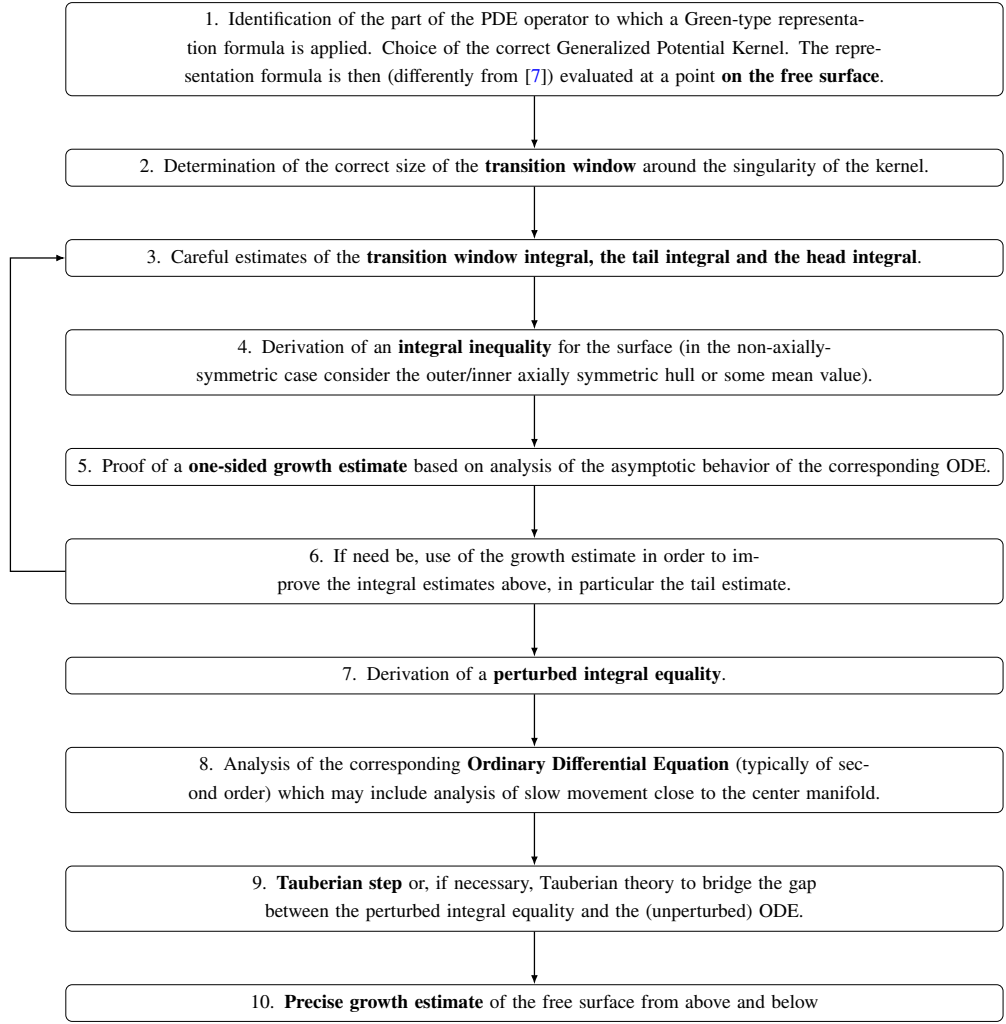
\begin{figure}
\resizebox{0.8\textwidth}{!}{%
\begin{tikzpicture}[
  node distance=1.15cm,
  box/.style={
    draw,
    rounded corners,
    align=center,
    text width=20cm,
    inner sep=6pt
  },
  arrow/.style={-{Latex[length=2mm]}, thick}
]

\node[box] (p1) {1. Identification of the part of the PDE operator to which a Green-type representation formula is applied. Choice of the correct Generalized Potential Kernel. The representation formula is then (differently from \cite{efw}) evaluated at a point \textbf{on the free surface}.};

\node[box, below=of p1] (p2) {2. Determination of the correct size of the \textbf{transition window} around the singularity of the kernel.};

\node[box, below=of p2] (p3) {3. Careful estimates of the \textbf{transition window integral, the tail integral and the head integral}.};

\node[box, below=of p3] (p4) {4. Derivation of an \textbf{integral inequality} for the surface (in the non-axially-symmetric case consider the outer/inner axially symmetric hull or some mean value).};

\node[box, below=of p4] (p5) {5. Proof of a \textbf{one-sided growth estimate} based on analysis of the asymptotic behavior of the corresponding ODE.};

\node[box, below=of p5] (p6) {6. If need be, use of the growth estimate in order to improve the integral estimates above, in particular the tail estimate.};

\node[box, below=of p6] (p7) {7. Derivation of a \textbf{perturbed integral equality}.};

\node[box, below=of p7] (p8) {8. Analysis of the corresponding \textbf{Ordinary Differential Equation} (typically of second order) which may include analysis of slow movement close to the center manifold.};

\node[box, below=of p8] (p9) {9. \textbf{Tauberian step} or, if necessary, Tauberian theory to bridge the gap between the perturbed integral equality and the (unperturbed) ODE.};

\node[box, below=of p9] (p10) {10. \textbf{Precise growth estimate} of the free surface from above and below};

\draw[arrow] (p1) -- (p2);
\draw[arrow] (p2) -- (p3);
\draw[arrow] (p3) -- (p4);
\draw[arrow] (p4) -- (p5);
\draw[arrow] (p5) -- (p6);
\draw[arrow] (p6) -- (p7);
\draw[arrow] (p7) -- (p8);
\draw[arrow] (p8) -- (p9);
\draw[arrow] (p9) -- (p10);

\draw[arrow]
  (p6.west) -- ++(-1.2,0)
  |- (p3.west);

\end{tikzpicture}
}
\caption{Potential Theoretic Reduction to ODE}\label{method}
\end{figure}
\end{center}

Although the approach introduced in the present paper, sketched in Figure \ref{method},
shares with that of \cite{efw} the idea of using the singularity in the potential expansion in order to derive a differential inequality,
there are major differences which make our approach
versatile and robust. First, the two-stage/iterative character
is vital in obtaining the precise ordinary differential equation
in a rigorous way. Second, in Section~4 of \cite{efw} the expansion point in the representation formula is chosen on the ``asymptotic axis of symmetry'' which simplifies calculations significantly and yields only a differential inequality and a growth estimate from above, while the growth estimate from below turns out to be much more subtle in all problems we have so far applied the approach to. Moreover, the different potential in the current paper
(while \cite{efw} uses a generalized Newtonian potential, the potential we use in the present paper is neither a generalized Newtonian potential nor a single-layer or double-layer potential but of mixed/combined type)
and its nonlocal character lead ---when carrying out consistently the approach in Figure \ref{method}---
quickly to an analysis in the present paper which is quite different from
that in Section~4 of \cite{efw}. 

Apart from the successful application in the present paper,
Dennis Kriventsov and the third author have applied the method successfully
in \cite{preprint_stef} (work in preparation)
to the 
{\em two-phase Stefan problem}, a parabolic free boundary problem, in which the
free boundary condition is considered to be of hyperbolic type.
It may be worth mentioning that in case of the Stefan problem, 
the principal part of the PDE operator is {\em truly quasilinear}
and does as a whole not admit a representation formula, and 
the desired
growth estimate follows only after a careful {\em center manifold analysis}.

We conjecture that the method ---apart from being more precise than
for example an analysis by frequency formulas--- is so general and robust
that it is applicable to the analysis of various (so far inaccessible) singularities and blow-downs in PDE of different or even mixed type and in 
integro-differential equations,
that it can be extended to non-axially-symmetric settings, and
that it may even be used to show flatness-implies-regularity 
and certain De Giorgi conjectures. Another significant point of our
approach is that while an underlying ODE may also be present in matched-asymptotic existence proofs or behind power-law ansatz results, its role is often implicit, and the corresponding dynamical mechanism is not made explicit. By contrast the ODE \eqref{ODE} in our approach is explicit and tangible.

\subsection{Cavity flows}
    
Steady cavity flows have long been a classical topic in fluid mechanics (\cite{BZ_1957,Gilbarg_1960}). Physically, a cavity appears in the low-pressure region behind a rigid body when the relative velocity of the fluid is sufficiently high. The boundary separating cavity and fluid is a free surface, the shape of which is not known a priori and must be determined as part of the solution. For inviscid incompressible flows without vorticity, the cavity problem reduces to a nonlinear free boundary problem for a harmonic function, namely the velocity potential, subject to a Neumann condition on the fixed boundary and overdetermined boundary conditions on the unknown free surface. The problem is of fundamental importance not only for its engineering applications, e.g. hydrofoil design, propeller performance, and high-speed underwater vehicles, but also for its role as a prototype free boundary problem in continuum mechanics. Understanding the asymptotic shape of the cavity is particularly important as it gives information on the global structure of the flow and is closely linked to the computation of hydrodynamic drag exerted on the body (\cite{Wu_1972}).

Early mathematical results relied on the hodograph method and conformal mappings. Using such complex analysis  techniques, the existence of two-dimensional infinite cavity flows past a curved body was established, see \cite{BZ_1957}. Moreover, it is known that the free boundary grows at most like the square root of the downstream distance, see \cite[Section 25]{Gilbarg_1960}. 

The axially symmetric cavity problem is considerably more challenging, as conformal mapping techniques are no longer available, 
and (see below) the cavity is not a paraboloid as in the two-dimensional cavity setting or the obstacle problem.
For incompressible cavity flows without vorticity, the velocity potential $\phi$ solves the following Neumann-type Bernoulli problem
    \begin{equation}\label{eq:cavity_pde}
    \begin{cases}\begin{split}
        &\Delta\phi = 0 && \text{in } D,\\
		&\partial_{\nu}\phi = 0 && \text{on } N \cup \Gamma,\\
		&\lvert\nabla \phi\rvert = 1 && \text{on } \Gamma,
    \end{split}\end{cases}
    \end{equation}
    where $D\subset \mathbb R^3$ is the flow region, $N$ is the fixed  boundary, $\Gamma$ is the free boundary, $\nu$ and $s$ are the unit outward normal and unit tangent on $\partial D:=N\cup\Gamma$, respectively (see Figure \ref{Fig:1}). 
	\begin{center}
		\includegraphics[height=4.5cm]{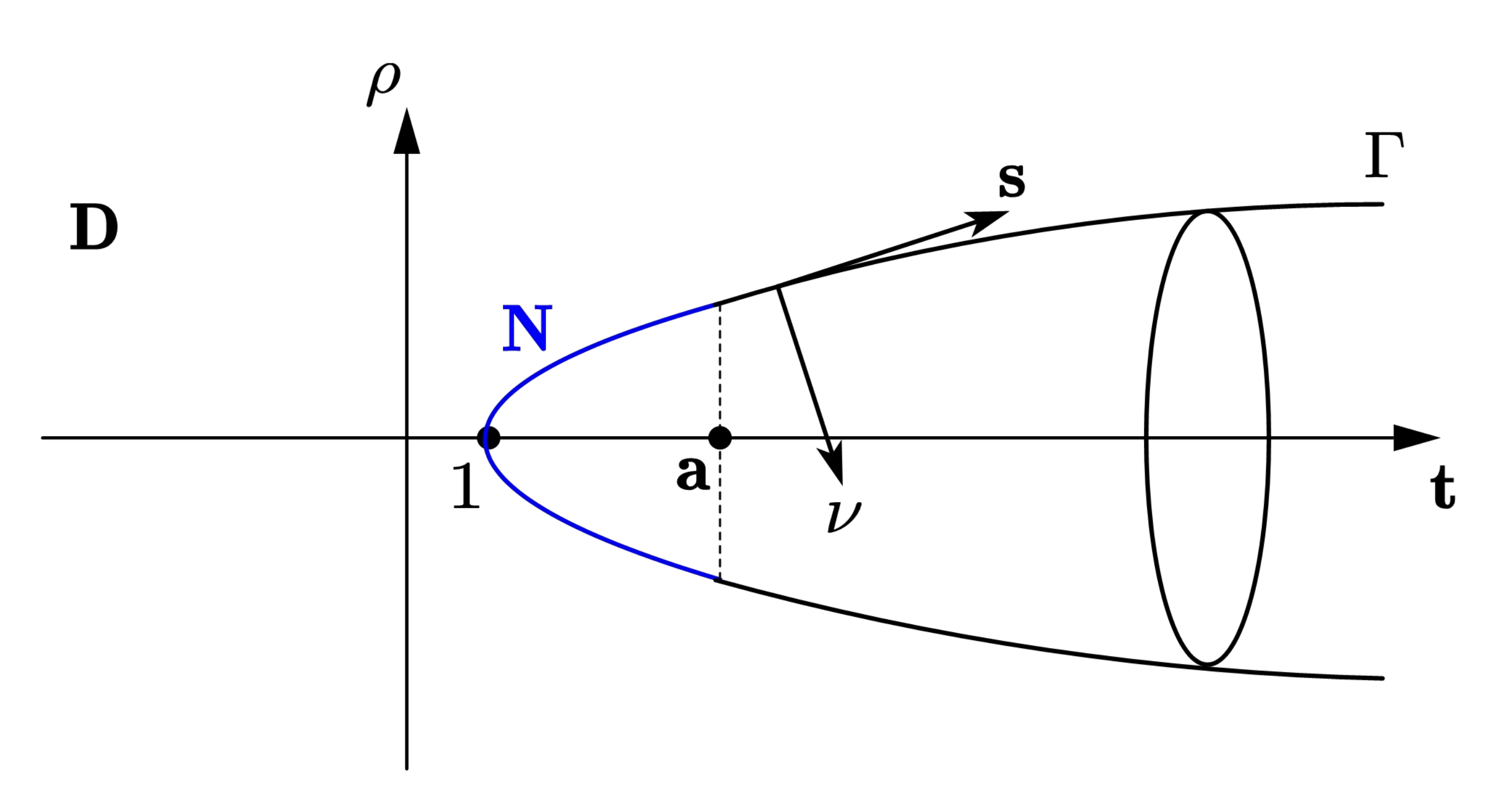}
		\captionof{figure}{Three-dimensional axially symmetric cavity problem}  
		\label{Fig:1}
	\end{center}
    We adopt cylindrical coordinates $(t,\rho,\lambda)$, where $t$ denotes the axial coordinate, $\rho$ the radial distance from the axis, and $\lambda$ the azimuthal angle. Throughout the paper, we always assume that the boundaries $N$ and $\Gamma$ are represented by a continuous function $g : [1,\infty) \to [0,\infty)$ as
		\begin{equation}\label{eq:condition_boundaryN}
			N:=\{(t,g(t),\lambda): t\in [1,a],\, \lambda \in [0,2\pi)\} 
		\end{equation}
		and
		\begin{equation}\label{eq:condition_boundaryG}
			\Gamma:=\{(t,g(t),\lambda):t\in (a,\infty),\, \lambda \in [0,2\pi)\},
		\end{equation}
		respectively, where $a>1$ is a constant, $g$ is smooth on $[1,a]$ and $[a,\infty)$ with $g^{\prime}(a-)=g^{\prime}(a+)$.
    
    In a landmark paper, Levinson (\cite{Levinson_1946}) used an integral representation by Green's theorem and derived the asymptotic behavior of the cavity by a power-law ansatz with a slowly varying correction, followed by an asymptotic balance analysis. Under the assumptions that $g(r)$ is concave at infinity and takes the form $g(r)=r^kh(r)$ with $0<k<1$, and that $\lim_{r\to\infty} rh'(r)/h(r)=0$ (note that this assumption implies
	\begin{align*}
\frac{rg'(r)}{g(r)} \to k \in (0,1),
\end{align*}
which is stronger than e.g. a convergence of the mean frequency and excludes oscillations by assumption), Levinson obtained 
	\begin{equation}\label{eq:intro_Levinson}
		\frac{\sqrt{r}}{(\log r)^{1/4+\varepsilon}}<g(r)<\frac{\sqrt{r}}{(\log r)^{1/4-\varepsilon}}
	\end{equation}
	for each $\varepsilon>0$ and for large $r$. 
	A further assumption that $r(\log r) h^{\prime}(r)/h(r)=O(1)$ for large $r$ leads to the formula
	\begin{equation*}
		g(r)=\frac{C r^{1/2}}{(\log r)^{1/4}}\left[1-\frac{1}{8}\frac{\log\log r}{\log r}+O\left(\frac{1}{\log r}\right)\right]\qquad \text{as } r\to\infty.
	\end{equation*}
    Note that some integrals in the derivation of Levinson's representation formula and thereafter are not convergent in the range $k\in (3/4,1)$ (for example the first integral in \cite[(2.12)]{Levinson_1946}), so that the proofs would have to be modified in order to make the result rigorous.
    
	The Levinson asymptotic behavior was derived independently by Gurevich (\cite{Gurevich_1947}), not using the Bernoul\-li free boundary condition at all but assuming that the Stokes
	stream function is an explicit Legendre-mode perturbation of a uniform stream. Subsequently, Scheid (\cite{Scheid_1950}) extended the expansion to higher order terms under additional assumptions on the asymptotic behavior of $g^{\prime}(r)$.
	
On the other hand, the {\em existence} of axially symmetric infinite cavities was first proved by Garabedian, Lewy, and Schiffer (\cite{GLS_1952}) using a variational approach together with a limit process from finite cavities. The uniqueness for axially symmetric infinite cavities was investigated in \cite{Gilbarg_1952}. Garabedian (\cite{Garabedian_1956}) later developed a perturbation method that treats axially symmetric flow as a deformation of plane flow within a family parametrized by a dimension parameter. Furthermore, the numerical result in \cite{Garabedian_1956} for the circular disk is consistent with Levinson's asymptotic formula for the infinite cavity shape. 
In the 1980s, Alt, Caffarelli, and Friedman (\cite{ACF_1985,Friedman_1982}) introduced a systematic variational framework for free boundary problems, proving existence and regularity for both plane and axially symmetric cavities. In particular, Caffarelli and Friedman (\cite{CF_1982}) established some basic physical properties for the axially symmetric infinite cavities obtained in \cite{GLS_1952}. Most importantly, they proved that the cavity has a sublinear growth rate at infinity. However, the sublinear growth rate is far away from the assumptions needed in Levinson's analysis (\cite{Levinson_1946}). Consequently, while the existence of a solution with sublinear growth is known, whether such a solution exhibits the precise asymptotic behavior predicted by Levinson has remained a challenging open question for 80 years. 

In fact, while Dirichlet-type Bernoulli problems, in which the free surface is a level set of the solution, have been extensively researched in the free boundary community, Neumann-type Bernoulli problems are notoriously hard, and not even well-posedness in the sense of a flatness-implies-regularity result is known for this important problem in mathematical physics. Even in an axially-symmetric setting, where the problem can be transformed into a Dirichlet-type Bernoulli problem with singular coefficients, known PDE techniques seem not to be applicable. For example, while it is possible (see \cite{efw}) to obtain frequency estimates for the difference of the solution and its blow-down limit in the obstacle problem, we have tried in vain to do something similar
in the cavity problem; on the contrary we have found evidence pointing towards a frequency formula corresponding to the one in \cite{aw} to be non-existent in the cavity problem.

In this paper, we obtain by the Potential Theoretic Reduction to ODE approach described above {\em complete rigidity at infinity} of the Levinson solution in the class of axially symmetric solutions, that is, any solution satisfying mild and natural assumptions at the fixed boundary and infinity converges asymptotically to the {\em Levinson profile}
	in the sense of \eqref{eq:intro_Levinson}. Combining our result with the existence result by Garabedian-Lewy-Schiffer \cite{GLS_1952} and that by Caffarelli-Friedman \cite{CF_1982}, we also show existence of a solution with the Levinson asymptotics \eqref{eq:intro_Levinson}, which has been an open problem ever since Levinson's result in 1946. The rigidity part of our result does not use energy-minimizing or finite Morse index assumptions at all, and we contend that extension to convex non-axially-symmetric cavities as well as flows with vorticity (where the vorticity satisfies certain structural assumptions) is possible, albeit even more technical than the present result.
	
    \subsection{Main results}
    We now turn to the precise formulation of our results. Consider the axially symmetric cavity problem \eqref{eq:cavity_pde}, with boundaries $N$ and $\Gamma$ parametrized by $g$ as in \eqref{eq:condition_boundaryN} and \eqref{eq:condition_boundaryG}, respectively. The trivial case $g\equiv0$ which corresponds to the absence of a cavity is excluded, meaning that we assume that
    \begin{equation}\label{assumption_g_not_zero}
     g\not\equiv0  \qquad\text{on }  [1,\infty). 
    \end{equation}
	We work under the following conditions.
	\begin{enumerate}[label=(\arabic*)] 
		\item The function $g$ has sublinear growth at infinity, i.e. 
		\begin{equation}\label{eq:condition_g_sublinear}
		\lim_{t \to \infty} \frac{g(t)}{t} = 0.
		\end{equation} 
		
		\item The flow speed on the fixed boundary $N$ does not exceed that on the free boundary $\Gamma$, i.e., 
		\begin{equation}\label{eq:condition_speed}
		\lvert \nabla \phi\rvert \le 1 \qquad\text{on } N.
		\end{equation}
		
		\item The flow velocity is asymptotically horizontal, i.e.,
		\begin{equation}\label{eq:condition_phid}
		\lim_{\lvert X\rvert \to \infty}\nabla\phi(X)=(1,0,0) \qquad \text{for } X=(t,\rho,\lambda)\in D.
		\end{equation}
	\end{enumerate}

    The main results of this paper are stated as follows.
	
	\begin{theorem}\label{thm:results_1}
		Assume that \eqref{eq:condition_g_sublinear}-\eqref{eq:condition_phid} hold. Then
		\begin{equation*}
		\limsup_{r\to \infty} \frac{g(r)}{r^{1/2} (\log r)^{-1/4+\varepsilon}} < \infty \qquad \text{for every }  \varepsilon > 0.
		\end{equation*}
	If in addition, the fixed boundary $N$ is concave (toward the cavity region), i.e.,  
		\begin{equation}\label{eq:concave_N}
		g^{\prime\prime}\le 0 \qquad\text{on } (1,a),
		\end{equation}
		then 
		\begin{equation*}
		\liminf_{r \to \infty} \frac{g(r)}{r^{1/2}(\log r)^{-1/4-\varepsilon}} > 0 \qquad \text{for every } \varepsilon>0.
		\end{equation*}
	\end{theorem}

    There are a few remarks in order.

        \begin{remark}
        As we mentioned before, the existence of the axially symmetric cavities with properties  \eqref{eq:condition_g_sublinear} and \eqref{eq:condition_phid} was established in \cite{CF_1982,GLS_1952}. Moreover, for fixed boundaries with convex profile $g$, a standard maximum principle argument implies that the fluid attains its maximum speed on the free boundary (see, e.g., the argument in \cite[Theorem 7.1]{ACF_1985}), thereby verifying condition \eqref{eq:condition_speed}. 
           Since for a conical fixed boundary $g$ is linear, it satisfies both the concavity hypothesis required for the lower-bound estimate in Theorem \ref{thm:results_1} and the convexity condition needed for the maximum principle argument.
    \end{remark}
    
        \begin{remark}
		The condition \eqref{eq:condition_speed} is equivalent to the pressure on the fixed boundary not to exceed the cavity pressure. If the liquid pressure on a portion of the fixed boundary were below the cavity pressure, that portion would itself be susceptible to cavitation. So the condition assures that $N$ remains wetted and a clear distinction of $N$ and the cavity is possible.   In the proof, the condition  \eqref{eq:condition_speed} will only be applied to establish the concavity of the free boundary, see Lemmas \ref{lem:Bernstein_property} and \ref{lem:concavity}.
	\end{remark}

           The results of \cite{CF_1982,GLS_1952} and Theorem \ref{thm:results_1} then yield the following existence theorem.
           
        \begin{theorem}\label{thm:existence}
        There exist three-dimensional axially symmetric cavity flows such that the free boundary has the asymptotic behavior
        \begin{equation}\label{eq:thm2}
        r^{\frac12}(\log r)^{-\frac14-\varepsilon}<g(r)<r^{\frac12}{(\log r)^{-\frac 14+\varepsilon}} 
        \end{equation}
        for each $\varepsilon>0$ and for large $r$.
        \end{theorem}
          
    \subsection{Key ideas of the proof and the structure of the paper}\label{subsec:structure of the paper}
In this subsection, we give a detailed outline for the proof of Theorem \ref{thm:results_1} and present the structure of the paper simultaneously.

Let us stress that using the assumption of axial symmetry, it would be possible to
transform the Neumann Bernoulli problem, using the Stokes stream function,
to a Dirichlet Bernoulli problem with
singular coefficients. However, aiming at future results on the Neumann Bernoulli
problem such as an extension of the present result to a non-axially-symmetric
configuration as well as a flatness-implies-regularity result,
all our proofs will use exclusively the original Neumann Bernoulli problem. 

In Section \ref{sec:pre}, we derive several geometric properties of the free boundary,
especially its concavity, from the velocity potential conditions \eqref{eq:condition_speed} and \eqref{eq:condition_phid}. We also identify the blow-down limit of the velocity potential. 

Let $V$ be the difference between the velocity potential and its blow-down limit defined in \eqref{eq:condition_blowdown}. In Section \ref{sec:potential_expansion} we use Green's theorem to represent $V(Y)$ (in contrast to \cite{efw}) {\em at a free boundary point} $Y$ (see Proposition \ref{prop:pe}). The classical Green's function fails to ensure convergence at infinity whereas a generalized Newtonian potential kernel (see \eqref{eq:Gdef}) does, thus yielding an exact integral identity connecting the velocity potential to the free boundary profile $g$ (see \eqref{eq:pe_F1}). We stress that this  integral identity differs from Levinson's (only formally derived) integral equation (1.5) in \cite{Levinson_1946}. As already mentioned, the resulting potential we use is neither a generalized Newtonian potential nor a single-layer potential or double-layer potential but of mixed/combined type.

The idea is now that the part of the integral close to the singularity of the kernel
 at point 
$Y=(r,g(r),0)$ should be close to some function of $r, g(r), g'(r)$ and $V(r)$.
Although finding the correct transition window around the singularity of the kernel
is a key step in our method, for the sake of shortness we give at the end of Section
\ref{sec:potential_expansion} just the correct transition window $(r/2,2r)$
for the division of the integral into integrals over subintervals; in general the size of the transition window
does depend on $g$, and actually on our way to the final transition window we started out with a preliminary ``short transition window'' $(r-g(r),r+g(r))$. According to the definition of the transition window, we define the transition window integral $J_g$,
the tail integral $T_g$ and the head integral $H_g$ as well as the fixed boundary integral
$N_g$.

In Section \ref{sec:initial_growth}, we establish lower bounds for each of the integrals arising from the decomposition.
In particular, we show (see Lemma \ref{lem:Jg_lower_bound}) that the transition window integral $J_g$
is bounded from below by the function 
$$\frac{1}{9}g(r)g^{\prime}(r)\log\frac{r}{2g(r)}$$
minus integral terms.
However the coefficient $1/9$ is not sharp and the function
is {\em not} the correct function to derive the ODE for our problem. 
Still these crude estimates yield an integro-differential inequality for $g$, from which we obtain an initial upper growth bound $g(r) \le r^{1/2+\varepsilon}$ at infinity
(see Proposition \ref{prop:r_root_growth}).

The initial upper bound on $g$, though not yet optimal, plays a crucial role in the subsequent refinement. It allows us to refine in Section \ref{sec:refined_growth} the integral estimates with additional control on $g$. Obviously the estimate on the tail
integral $T_g$ improves. What is not so obvious, is that we divide the transition window integral $J_g$ into two parts. The first part we estimate in Lemma \ref{lem:Jg_refine_gvF} as follows:
\begin{equation*}
	\frac{1}{2\pi} \int_{0}^{2\pi} \int_{\frac r2}^{2r} g(t)V(t)F_{g}(t,\lambda;r) \mathrm{d}t\mathrm{d}\lambda \ge (1-\varepsilon)V(r) - \varepsilon\qquad  \text{for } r>R_{\varepsilon}.
\end{equation*}
Note that unfortunately, $V(r)$ on the right-hand side 
depends on $r$ and $g'$ in a nonlocal way. On the other hand, $V(r)$ appears already on the left-hand side
of our representation formula (see \eqref{eq:pe_main}), so 
while introducing a nonlocal quantity or second variable
into our integro-differential inequality, this term is unavoidable anyway.

The second part of $J_g$ we estimate in Lemma \ref{lem:Jg_refine_gdg}
as
\begin{equation*}
	\begin{split}
		& \frac{1}{2\pi} \int_{0}^{2\pi} \int_{\frac r2}^{2r} g(t)g^{\prime}(t)\left(\frac{1}{\lvert X-Y\rvert} - \frac{1}{\lvert X\rvert}\right) \mathrm{d}t\mathrm{d}\lambda\\
		\ge &\ 2(1-2\eta)g(r-\eta r)g^{\prime}(r+\eta r)\log\frac{\eta r}{g(r)} - \int_{\frac r2}^{\frac{r+\eta r}2} \frac{g(t)g^{\prime}(t)}{t} \mathrm{d}t
		\qquad \text{for } r > R_{\eta}.
	\end{split}
\end{equation*}
While this gives us the sharp coefficient, it introduces the unwanted slight $\eta$-shift,
which has to be taken care of in the proof of the final integro-differential inequality 
in Lemma \ref{lem:refined_differential_ineq}. Corollary \ref{coro:refined_ineq_sharp} uses Hardy's inequality to obtain an integro-differential inequality without the nonlocal term $V$, which in turn leads to the sharp
upper bound $g(r)\le r^{1/2}(\log r)^{-1/4+\varepsilon}$ at infinity (see Proposition \ref{prop:g_refined_growth}). 

In Subsection \ref{subsec:sharp_estimate} we derive the corresponding estimates from above
and obtain in Proposition \ref{prop:lower_bound_log1} the {\em perturbed integral equality}
\begin{equation}\label{pert_int}
		V(r) + (1+o(1))g(r)g^{\prime}(r)\log\frac{r}{g(r)} - \frac{1}{2} \int_{a}^{r} \frac{g(t)g^{\prime}(t)}{t}\mathrm{d}t = C+o(1) \qquad\text{as } r \to \infty.
	\end{equation}
Formally, this corresponds to the precise second order ODE
\begin{equation}\label{ODE}
\left(g(r)g^{\prime\prime}(r)+g^{\prime}(r)^2\right)
\log\frac{r}{g(r)}
+\frac{g(r)g^{\prime}(r)}{2r}
+\sqrt{1+g^{\prime}(r)^2}
-1-g^{\prime}(r)^2
=0,
\end{equation}
which can be analyzed and selects the desired logarithmic exponent $-1/4$.
Unfortunately there is another critical logarithmic exponent $-1/2$,
which is not relevant for the ODE, but is highly relevant for the
perturbed integral equality. It is the threshold at which the asymptotic
behavior of $g(r)g^{\prime}(r)\log(r/g(r))$ changes from unbounded to
bounded behavior as the exponent decreases, and the asymptotic behavior
$$\sqrt{\frac{r}{\log r}}$$
cannot be easily discarded using only the perturbed integral equality.
It may be worth mentioning that it is possible to construct examples of concave increasing polygonal functions $g$ which 
oscillate between $\sqrt{r/\log r}$ and $r^\alpha$ for some $\alpha\in(0,1/2)$ and satisfy the cruder perturbed integral equality
\begin{equation*}
		V(r) + g(r)g^{\prime}(r)\log\frac{r}{g(r)} - \frac{1}{2} \int_{a}^{r} \frac{g(t)g^{\prime}(t)}{t}\mathrm{d}t = o(V(r))\qquad\text{as } r \to \infty.
	\end{equation*}    

A crucial step in overcoming the asymptotic behavior $\sqrt{r/\log r}$ as well as oscillations is the bound from below,
\begin{equation}\label{eq:logbound}
\liminf_{r \to \infty} g(r)g^{\prime}(r)\log\frac{r}{g(r)} > 0,
\end{equation}
which is proven in Proposition \ref{prop:concave_log} in Subsection \ref{subsec:logbound},
assuming that the fixed boundary is concave. 
This bound from below means that the $o(1)$-perturbation in \eqref{pert_int}
is for large $r$ much smaller than $g(r)g^{\prime}(r)\log(r/g(r))$. 
The last missing piece is then a {\em Tauberian step} in Subsection \ref{subsec:taub}.
Proposition \ref{prop:Tauberian} transforms the perturbed integral identity into a {\em first-order dynamical system}, thereby avoiding both the use of more elaborate Tauberian machinery and estimating differences of the representation formula, and shows that the improved perturbed integral equality
\begin{equation*}
V(r) + g(r)g^{\prime}(r)\log\frac{r}{g(r)} - \frac{1}{2}\int_{a}^{r} \frac{g(t)g^{\prime}(t)}{t}\mathrm{d}t 
= \tilde C+o\left(g(r)g^{\prime}(r)\log\frac{r}{g(r)}\right) \qquad\text{as } r \to \infty
\end{equation*}
(which is a consequence of \eqref{eq:logbound}),
implies
the sharp
lower bound $g(r)\ge r^{1/2}(\log r)^{-1/4-\varepsilon}$ at infinity.

Together with the upper bound obtained earlier, this yields the precise growth rate of the free boundary.

As a consequence, we also obtain the existence of axially symmetric cavity flows with this exact asymptotic behavior.

\subsection{Notations}     Throughout the paper, the space $\mathbb{R}^3$ is equipped with the Euclidean inner product $X\cdot Y$ and the associated norm $\lvert X\rvert$. We use cylindrical coordinates $(t,\rho,\lambda)$, where $t$ denotes the axial coordinate, $\rho$ the radial distance from the axis, and $\lambda$ the azimuthal angle. The function $g$ is a parametrization of the fixed boundary $N$ and the free boundary $\Gamma$, with $a>0$ denoting the axial coordinate of the junction between $N$ and $\Gamma$.

 We denote by $B_R$ the open ball in $\mathbb{R}^3$ centered at the origin with radius $R>0$ and $\partial B_R$ its boundary.

    In the estimates, $C$ denotes a generic positive constant, whose value may vary from line to line.

    \section{Preliminaries}\label{sec:pre}
    This section collects some preliminary results that will be used in the subsequent analysis.

    \subsection{Geometric properties of the free boundary}\label{subsec:g_properties}
    
	We first show that the flow speed $\lvert\nabla\phi\rvert$ attains its maximum on the free boundary $\Gamma$ throughout the flow region under the conditions \eqref{eq:condition_speed} and \eqref{eq:condition_phid}. Note that $\lvert\nabla\phi\rvert=1$ on $\Gamma$ by \eqref{eq:cavity_pde}. 
	
	\begin{lemma}\label{lem:Bernstein_property}
		Assume that the velocity potential $\phi$ satisfies \eqref{eq:condition_speed} and \eqref{eq:condition_phid}. 
		Then 
		\begin{equation}\label{eq:Bernstein_property}
			\lvert \nabla \phi\rvert \le 1 \qquad\text{in } D,
		\end{equation}
		where $D$ is the flow region.
	\end{lemma}
	
	\begin{proof}
		We argue by contradiction. Suppose that there is a point $X_{0}\in D$ such that $\lvert \nabla \phi(X_{0})\rvert=1+\varepsilon$ with some $\varepsilon>0$. Since $\lvert\nabla\phi(X)\rvert\to1$ as $\lvert X\rvert\to\infty$ by \eqref{eq:condition_phid}, there exists $R>\lvert X_{0}\rvert$ such that
		\begin{equation*}
			\lvert \nabla\phi\rvert \le 1+\frac{\varepsilon}2  \qquad\text{on } D\cap\partial B_{R}.
		\end{equation*}
		For a harmonic function $\phi$ in $D$, $\lvert\nabla\phi\rvert^2$ is subharmonic in $D$. 
Using the maximum principle yields
\begin{equation*}
	(1+\varepsilon)^2=\lvert \nabla \phi(X_{0})\rvert^{2}\le \sup_{N\cup\Gamma\cup (D\cap\partial B_{R})} \lvert\nabla\phi\rvert^{2}\le  \left(1+\frac{\varepsilon}2\right)^{2}.
\end{equation*}
This leads to a contradiction and thus finishes the proof of the lemma.
\end{proof}

On each streamline that bounds fluid from one side, straightforward calculations (see, e.g.,  \cite[Lemma 7.2]{ACF_1985}) give 
\begin{equation}\label{eq:speed_curvature_relation}
\frac{\partial q}{\partial\nu}=kq, 
\end{equation}
where $q:=\lvert\nabla\phi\rvert$ is the flow speed, $\nu$ is the unit outward normal, and $k$ is the curvature of the streamline. On the free boundary $\Gamma$, 
\begin{equation*}
k:=-\frac{g^{\prime\prime}}{\left(1+\lvert g^{\prime}\rvert^2\right)^{3/2}}.
\end{equation*}
If \eqref{eq:Bernstein_property} holds, then \eqref{eq:speed_curvature_relation} gives $k>0$ and hence $g^{\prime\prime}<0$ on $\Gamma$.  
Therefore, the concavity of the free boundary follows immediately from Lemma \ref{lem:Bernstein_property}.

\begin{lemma}\label{lem:concavity}
Assume that the velocity potential $\phi$ satisfies \eqref{eq:condition_speed} and \eqref{eq:condition_phid},  then the free boundary $\Gamma$ is concave (toward the cavity region), i.e.
\begin{equation}\label{eq:concave_fb}
g^{\prime\prime} \le 0 \qquad \text{on } (a,\infty).
\end{equation}
\end{lemma}

By the concavity of the free boundary \eqref{eq:concave_fb}, it holds that
\begin{equation}\label{eq:condition_dg_sign}
g^{\prime}\ge 0 \quad\text{on } (a,\infty).
\end{equation}
Otherwise, $g$ would eventually become negative, contradict its nonnegativity. For nontrivial $g$ (see \eqref{assumption_g_not_zero}), this implies that $g(t)>0$ for large $t$. Moreover, \eqref{eq:concave_fb} together with the sublinear growth condition \eqref{eq:condition_g_sublinear} gives
\begin{equation}\label{eq:condition_dg_to0}
g^{\prime}(t)\le \frac{g(t)-g(a)}{t-a}\to0 \qquad\text{as }t\to\infty.
\end{equation}
Consequently, $g^{\prime}$ is bounded on $(2a,\infty)$ even though $g^{\prime}(a)$ may be infinite.

    \subsection{Identification of the blow-down limit}
    
    We next determine the blow-down limit of the velocity potential.
    
    \begin{lemma}
       Assume that the velocity potential $\phi$ satisfies \eqref{eq:condition_phid}. Then the blow-down limit of $\phi$ exists and is given by
		\begin{equation}\label{eq:condition_blowdown}
			\phi_{\infty}(X):=\lim_{R \to \infty} \frac{\phi (RX)}{R} = t \qquad\text{for }X=(t,\rho,\lambda)\in \mathbb{R}^3.
		\end{equation} 
    \end{lemma}
    \begin{proof}
     For $X\in \mathbb{R}^3$ and sufficiently large $R>0$ with $RX\in \overline D$,  one has
		\begin{equation*}
		\phi(RX)=\phi(0)+\int_0^R \nabla \phi(\tau X) \cdot X \mathrm d\tau.
		\end{equation*}
		Hence it holds that
		\begin{equation*}
		\frac{\phi(RX)}{R}=\frac{\phi(0)}R+\frac1R\int_0^R \nabla \phi(\tau X) \cdot X \mathrm d\tau=\frac{\phi(0)}R+\int_0^1\nabla \phi(\tau RX) \cdot X \mathrm d\tau.
		\end{equation*}
		Letting $R \to \infty$, we obtain \eqref{eq:condition_blowdown} from \eqref{eq:condition_phid}.   
    \end{proof}

    \subsection{Auxiliary inequalities}
    
We shall need the following Hardy's inequality on a finite interval, which can be obtained from the classical one (see \cite[Theorem 327]{HLP_1952_Inequalities}) via zero extension of the associated functions.

\begin{lemma}[Hardy's inequality]\label{lemma:preliminaries_1}
Let $p\in(1,\infty),\, r>0$, and let $f\in L^{p}(0,r)$ be nonnegative. Then 
\begin{equation*}
\int_{0}^{r} \left(\frac{1}{t}\int_{0}^{t} f(\tau)\mathrm{d}\tau \right)^{p}\mathrm{d}t \le \left( \frac{p}{p-1} \right)^{p} \int_{0}^{r} f(t)^{p}\mathrm{d}t.
\end{equation*}
\end{lemma}

We will also use the following algebraic inequality.

\begin{lemma}\label{lem:Nconcave_liminf2}
Let $p,\, x,\, y\geq 0$, and $0\le w\le 2$.
Then 
\begin{equation*}
\frac{x^{2}-xp -x(x+y)w}{1+\sqrt{1+p^{2}}} + 1+x^{2}+2y(x+y)w > 0.
\end{equation*}
\end{lemma}

\begin{proof}
Rearranging the terms yields
\begin{equation*}
\begin{split}
&\ \frac{x^{2}-xp -x(x+y)w}{1+\sqrt{1+p^{2}}} + 1+x^{2}+2y(x+y)w\\
=&\  \frac{2-w+\sqrt{1+p^{2}}}{1+\sqrt{1+p^{2}}} x^{2} - \frac{p}{1+\sqrt{1+p^{2}}}x + 1 + \left( \frac{-1}{1+\sqrt{1+p^{2}}} + 2 \right) xyw + 2y^{2}w.
\end{split}
\end{equation*}
Since $2-w\ge 0$ and the last two terms in the above equality are both nonnegative, calculating directly we get 
\begin{equation*}
	\begin{split}
		&\ \frac{x^{2}-xp -x(x+y)w}{1+\sqrt{1+p^{2}}} + 1+x^{2}+2y(x+y)w\\
		\ge &\ \frac{\sqrt{1+p^{2}}}{1+\sqrt{1+p^{2}}} x^{2} - \frac{p}{1+\sqrt{1+p^{2}}} x + 1 = \frac{\sqrt{1+p^{2}}}{1+\sqrt{1+p^{2}}} \left( x^{2} - \frac{p}{\sqrt{1+p^{2}}} x \right) + 1\\
		=&\ \frac{\sqrt{1+p^{2}}}{1+\sqrt{1+p^{2}}} \left( \left( x-\frac{p}{2\sqrt{1+p^{2}}}\right)^{2} - \frac{p^{2}}{4(1+p^{2})} \right) + 1\\
		\ge&\ - \frac{\sqrt{1+p^{2}}}{1+\sqrt{1+p^{2}}}\cdot \frac{p^{2}}{4(1+p^{2})} + 1 \geq -\frac{1}{4}+1>0.
	\end{split}
\end{equation*}
This finishes the proof.
\end{proof}

\section{A potential expansion identity for the cavity problem}\label{sec:potential_expansion}

In this section, we derive an integral identity by potential expansion for the three-dimensional axially symmetric cavity problem \eqref{eq:cavity_pde}. This identity provides the basis for our subsequent analysis on the asymptotic shape of the free boundary. The derivation uses a generalized Newtonian potential kernel. This choice of kernel is a key  point of departure from Levinson's integral equation in \cite{Levinson_1946}, as it yields an exact identity in which the convergence of all integrals follows automatically from the natural assumption \eqref{eq:condition_phid}.

Define 
\begin{equation}\label{eq:Gdef}
G(X,Y) := \frac{1}{\lvert X-Y\rvert} - \frac{1}{\lvert X\rvert}
\end{equation}
and
\begin{equation}\label{eq:Vdef}
V(X):= \phi(X) - \phi_{\infty}(X),
\end{equation}
where $\phi_{\infty}$ is the blow-down limit of $\phi$ as in \eqref{eq:condition_blowdown}. We have the following potential expansion of $V$.

\begin{proposition}\label{prop:pe}
Let $G$ and $V$ be defined in \eqref{eq:Gdef} and \eqref{eq:Vdef}, respectively. For each $Y \in \Gamma$, 
\begin{equation}\label{eq:pe_1}
	-V(Y) + 2V(0) = \frac{1}{2\pi} \int_{N \cup \Gamma} V(X)\partial_{\nu}G(X,Y) - G(X,Y)\partial_{\nu}V(X) \mathrm{d}\mathcal{H}^{2}(X).
\end{equation}
\end{proposition}

\begin{proof}
For each $Y \in \Gamma$ and each $R > \lvert Y\rvert$, since  $\Delta V=0$ in $D$, Green's identity gives
\begin{equation}\label{eq:pe_2}
	\begin{split}
		-\frac 12V(Y)+V(0)
		&=\frac1{4\pi}\int_{\partial (B_{R}\cap D)} V(X)\partial_{\nu}G(X,Y) - G(X,Y)\partial_{\nu}V(X) \mathrm{d}\mathcal{H}^{2}(X).
	\end{split}
\end{equation}
For $X \in \partial B_{R} \cap D$, by \eqref{eq:condition_blowdown} and \eqref{eq:condition_phid}, one has 
\begin{equation*}
	V(X) = (\phi - \phi_{\infty})(X) = o(R) 
	\quad \text{and}\quad 
	\partial_{\nu}V(X) =\nabla(\phi - \phi_{\infty})(X)\cdot \frac{X}{\lvert X\rvert} =o(1) \qquad \text{as } R\to \infty.
\end{equation*}
Moreover, direct calculations yield that for $X \in \partial B_{R} \cap D$, as $R\to \infty$, one has
\begin{equation*}
	G(X,Y) = \frac{\lvert X\rvert - \lvert X-Y\rvert}{\lvert X-Y\rvert \lvert X\rvert} = O(R^{-2})
\end{equation*}
and 
\begin{equation*}
	\begin{split}
		\partial_{\nu}G(X,Y)
		&=\left(-\frac{X-Y}{\lvert X-Y\rvert^{3}} + \frac{X}{\lvert X\rvert^{3}} \right)\cdot\frac{X}{\lvert X\rvert} \\
		&=\frac{Y\cdot X}{\lvert X-Y\rvert^{3}\lvert X\rvert} +\left(\frac{1}{\lvert X\rvert^{3}}-\frac{1}{\lvert X-Y\rvert^{3}}\right)\lvert X\rvert
		= O(R^{-3}).
	\end{split}
\end{equation*}
Hence 
\begin{equation}\label{eq:pe_3}
	\lim_{R \to \infty} \int_{\partial B_{R} \cap D} V(X)\partial_{\nu}G(X,Y) - G(X,Y)\partial_{\nu}V(X) \mathrm{d}\mathcal{H}^{2}(X) = 0.
\end{equation}
Then it follows from \eqref{eq:pe_2} and \eqref{eq:pe_3} that
\begin{equation*}
	\begin{split}
		-V(Y) + 2V(0)
		&= \lim_{R \to \infty}\frac1{2\pi} \int_{B_{R} \cap \partial D} V(X)\partial_{\nu}G(X,Y) - G(X,Y)\partial_{\nu}V(X) \mathrm{d}\mathcal{H}^{2}(X)\\
		&= \frac1{2\pi}\int_{N\cup \Gamma} V(X)\partial_{\nu}G(X,Y) - G(X,Y)\partial_{\nu}V(X) \mathrm{d}\mathcal{H}^{2}(X).
	\end{split}
\end{equation*}
This finishes the proof of the proposition.
\end{proof}

\begin{center}
\includegraphics[height=4cm]{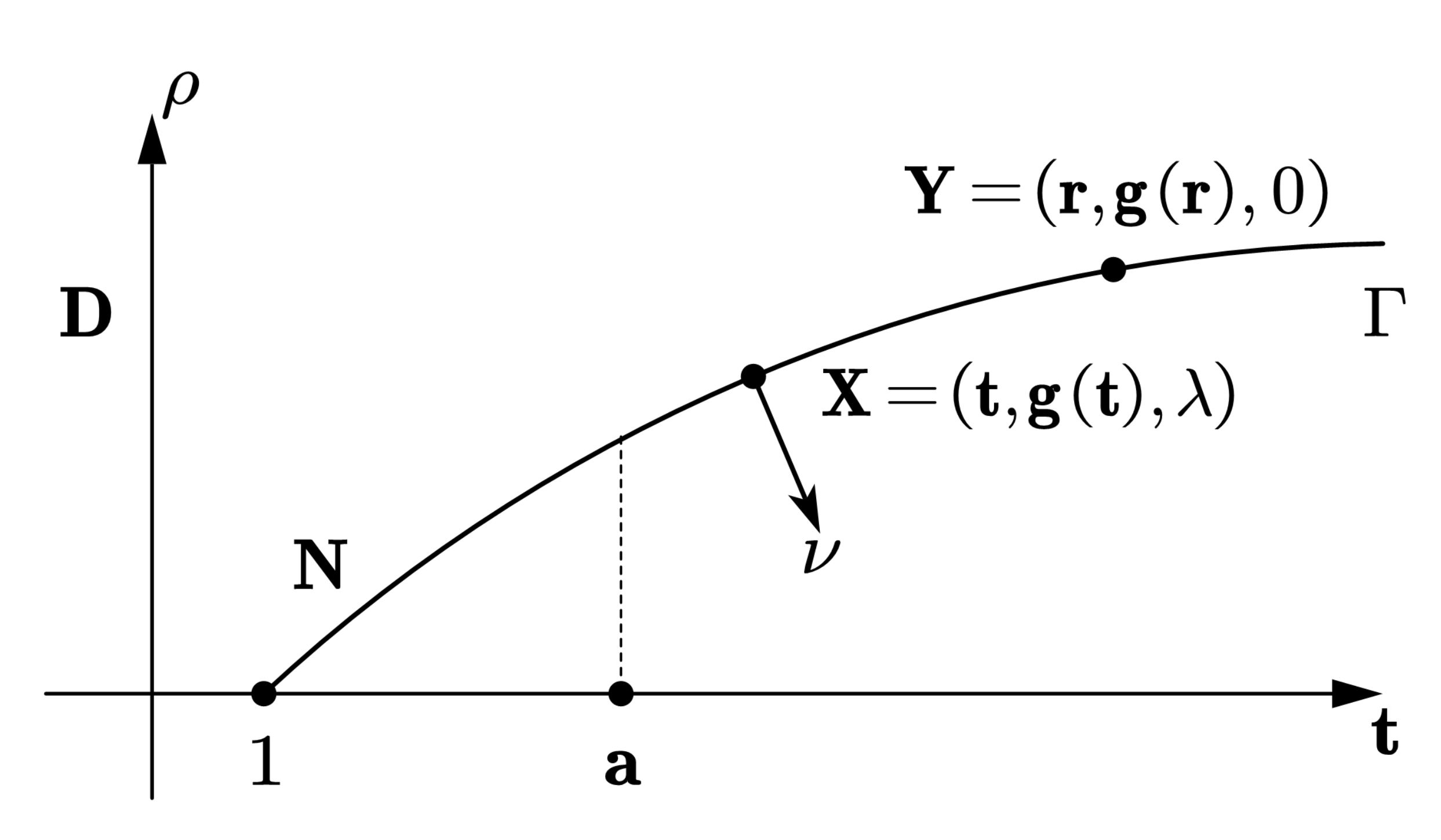}
\captionof{figure}{Potential expansion at $Y=(r,g(r),0)$}
\label{Fig:2}
\end{center}

From now on, we let $Y=(r,g(r),0)\in \Gamma$ and $X=(t,g(t),\lambda)\in N \cup \Gamma$. The unit outward normal at each point $X$ is given by
\begin{equation*}
\nu= \frac{(g^{\prime}(t),1,\lambda + \pi)}{\sqrt{1 + \lvert g^{\prime}(t)\rvert^{2}}},
\end{equation*}
which in Cartesian coordinates becomes 
\begin{equation*}
\nu= \frac{(g^{\prime}(t),-\cos\lambda,-\sin\lambda)}{\sqrt{1 + \lvert g^{\prime}(t)\rvert^{2}}}.
\end{equation*}
Then 
\begin{equation*}
\begin{split}
	\partial_{\nu}G(X,Y) 
	&=\left(-\frac{X-Y}{\lvert X-Y\rvert^{3}} + \frac{X}{\lvert X\rvert^{3}}\right)\cdot\nu=\frac{Y\cdot\nu}{\lvert X-Y\rvert^{3}}+(X\cdot\nu)\left(\frac{1}{\lvert X\rvert^{3}} - \frac{1}{\lvert X-Y\rvert^{3}}\right)\\
	&=\frac{rg^{\prime}(t)- g(r)\cos\lambda}{\sqrt{1 + \lvert g^{\prime}(t)\rvert^{2}}}\cdot\frac{1}{\lvert X-Y\rvert^{3}} + \frac{tg^{\prime}(t)-g(t)}{\sqrt{1 + \lvert g^{\prime}(t)\rvert^{2}}} \left(\frac{1}{\lvert X\rvert^{3}} - \frac{1}{\lvert X-Y\rvert^{3}}\right),
\end{split}
\end{equation*}
and since $\partial_{\nu}\phi=0$ on $N\cup\Gamma$,  
\begin{equation*}
\partial_{\nu}V(X) 
=\nabla(\phi - \phi_{\infty})(X)\cdot\nu
=-\nabla \phi_{\infty}(X)\cdot\nu
=\frac{-g^{\prime}(t)}{\sqrt{1 + \lvert g^{\prime}(t)\rvert^{2}}}.
\end{equation*}
Therefore, \eqref{eq:pe_1} can be rewritten as
\begin{equation}\label{eq:pe_F1}
\begin{split}
&\ -V(Y) + 2V(0)\\
=&\ \frac{1}{2\pi} \int_{0}^{2\pi} \int_{1}^{\infty} g(t)\sqrt{1 + \lvert g^{\prime}(t)\rvert^{2}} \big( V(X)\partial_{\nu}G(X,Y) - G(X,Y)\partial_{\nu}V(X) \big) \mathrm{d}t \mathrm{d}\lambda\\
=&\ \frac{1}{2\pi} \int_{0}^{2\pi} \int_{1}^{\infty} g(t)V(X)F_{g}(t,\lambda;r) + g(t)g^{\prime}(t)\left( \frac{1}{\lvert X-Y\rvert} - \frac{1}{\lvert X\rvert} \right) \mathrm{d}t \mathrm{d}\lambda,
\end{split}
\end{equation}
where 
\begin{equation}\label{eq:Fg_def}
F_{g}(t,\lambda;r) := \frac{rg^{\prime}(t) - g(r)\cos\lambda}{\lvert X-Y\rvert^{3}} + (tg^{\prime}(t) - g(t) ) \left( \frac{1}{\lvert X\rvert^{3}} - \frac{1}{\lvert X-Y\rvert^{3}} \right).
\end{equation}

Write $V(t) := V(X)$ and $\partial_{s}\phi(t) := \partial_{s}\phi(X)$ for each $X = (t,g(t),\lambda) \in N\cup \Gamma$. By the definition of $V$ in \eqref{eq:Vdef}, we have 
\begin{equation}\label{eq:Vt_expression}
V(t)= V(a)+\int_{a}^{t} \left(\partial_{s}\phi(\tau) \sqrt{1 + \lvert g^{\prime}(\tau)\rvert^{2}} - 1 \right) \mathrm{d}\tau.
\end{equation}
Since $\phi$ is defined up to an additive constant, we may assume $V(a) = 0$. Then
\begin{equation}\label{eq:Vt_def}
V(t)= 
\begin{cases}
\displaystyle \int_{t}^{a} \left( 1 - \partial_{s}\phi(\tau)\sqrt{1 + \lvert g^{\prime}(\tau)\rvert^{2}} \right) \mathrm{d}\tau \quad&\text{if } t\in (1,a];\\[12pt]
\displaystyle \int_{a}^{t}\left(\sqrt{1 + \lvert g^{\prime}(\tau)\rvert^{2}} - 1\right)\mathrm{d}\tau \quad &\text{if } t \in (a,\infty),
\end{cases}
\end{equation}
where $\partial_s\phi=1$ on $\Gamma$ has been used for $t\in(a,\infty)$.  
From \eqref{eq:pe_F1}, for $Y=(r,g(r),0)$ with $r>a$, we have the following identity
\begin{equation}\label{eq:pe_main}
-V(r) + 2V(0) = \frac{1}{2\pi} \int_{0}^{2\pi} \int_{1}^{\infty} g(t)V(t)F_{g}(t,\lambda;r) + g(t)g^{\prime}(t)\left( \frac{1}{\lvert X-Y\rvert} - \frac{1}{\lvert X\rvert} \right) \mathrm{d}t\mathrm{d}\lambda.
\end{equation}

In the following, we use the decomposition
\begin{equation}\label{eq:pe_main_NHJT}
-V(r) + 2V(0)=N_g(r)+H_g(r)+J_g(r)+T_g(r),
\end{equation}
where 
\begin{align}
&N_g(r) := \frac{1}{2\pi} \int_{0}^{2\pi} \int_{1}^{a} g(t)V(t)F_{g}(t,\lambda;r) + g(t)g^{\prime}(t)\left(\frac{1}{\lvert X-Y\rvert} - \frac{1}{\lvert X\rvert}\right) \mathrm{d}t\mathrm{d}\lambda,\label{eq:Ng_def}\\
&H_{g}(r) := \frac{1}{2\pi} \int_{0}^{2\pi} \int_{a}^{\frac r2} g(t)V(t)F_{g}(t,\lambda;r) + g(t)g^{\prime}(t)\left( \frac{1}{\lvert X-Y\rvert} - \frac{1}{\lvert X\rvert} \right) \mathrm{d}t\mathrm{d}\lambda,\label{eq:Hg_def}\\
&J_{g}(r) := \frac{1}{2\pi} \int_{0}^{2\pi} \int_{\frac r2}^{2r} g(t)V(t)F_{g}(t,\lambda;r) + g(t)g^{\prime}(t)\left( \frac{1}{\lvert X-Y\rvert} - \frac{1}{\lvert X\rvert} \right) \mathrm{d}t\mathrm{d}\lambda,\label{eq:Jg_def}\\
&T_{g}(r):= \frac{1}{2\pi} \int_{0}^{2\pi} \int_{2r}^{\infty} g(t)V(t)F_{g}(t,\lambda;r) + g(t)g^{\prime}(t)\left( \frac{1}{\lvert X-Y\rvert} - \frac{1}{\lvert X\rvert} \right) \mathrm{d}t\mathrm{d}\lambda. \label{eq:Tg_def}
\end{align}

\section{An initial growth estimate for the free boundary}\label{sec:initial_growth}

In this section, we estimate the four terms in the identity  \eqref{eq:pe_main_NHJT} one by one, and combine the resulting lower bounds to obtain an integro-differential inequality. This leads to an algebraic upper bound for the growth of the free boundary with exponent $1/2+\varepsilon$.

We shall temporarily work under the additional assumption that $g$ is unbounded, i.e.,
\begin{equation}\label{eq:condition_g_unbounded}
\limsup_{t \to \infty} g(t) = \infty.
\end{equation} 
This assumption is for technical convenience only. As will be shown below (see Propositions \ref{prop:r_root_growth} and \ref{prop:g_refined_growth}), the bounded case follows by a trivial argument, so it suffices to treat the unbounded case.  
Note that the unboundedness of $g$ together with the concavity of the free boundary \eqref{eq:concave_fb} yields 
\begin{equation}\label{eq:condition_g_infty}
\lim_{t\to\infty} g(t)=\infty.
\end{equation}

\subsection{Estimates on the decomposed integrals}\label{subsec:lower_bounds}

First, we prove the following two preliminary lemmas.

\begin{lemma}\label{lem:Fg_lower_bound}
Assume that $g$ satisfies \eqref{eq:concave_fb}. Let $F_g$ be defined in \eqref{eq:Fg_def}. Then it holds that
\begin{equation*}
	F_{g}(t,\lambda;r) \ge \frac{g(r)(1-\cos\lambda)}{\lvert X-Y\rvert^{3}} + \frac{tg^{\prime}(t) - g(t)}{\lvert X\rvert^{3}}
	\qquad \text{for } t,\, r>a.
\end{equation*}
\end{lemma}

\begin{proof}
By concavity of $g$ on $(a,\infty)$ (see \eqref{eq:concave_fb}), one has 
\begin{equation*}
g(r)\le  
g(t) +(r-t)g^{\prime}(t) \qquad \text{for } t,\, r>a.
\end{equation*}
Hence, it follows from the definition of $F_g$ in \eqref{eq:Fg_def} that 
\begin{equation*}
F_{g}(t,\lambda;r)= \frac{g(t) +(r-t)g^{\prime}(t) - g(r)\cos\lambda}{\lvert X-Y\rvert^{3}} + \frac{tg^{\prime}(t) - g(t)}{\lvert X\rvert^{3}}
\ge \frac{g(r)(1-\cos\lambda)}{\lvert X-Y\rvert^{3}} + \frac{tg^{\prime}(t) - g(t)}{\lvert X\rvert^{3}}.
\end{equation*}
This finishes the proof of the lemma. 
\end{proof}

\begin{lemma}\label{lem:Vg_relation}
Assume that $g$ satisfies \eqref{eq:condition_g_sublinear} and \eqref{eq:condition_g_infty}. Let $V$ be as in \eqref{eq:Vt_def}. There exists a constant $t_{0} > a$ such that
\begin{equation*}
\frac{g(t)^{2}}{t} \le 4V(t)\qquad \text{for } t> t_{0}.
\end{equation*}
\end{lemma}

\begin{proof}
In view of \eqref{eq:Vt_def}, together with the fact that arc length dominates chord length, we have
\begin{equation*}
	\begin{split}
		V(t) =  \int_{a}^{t} \left(\sqrt{1 + \lvert g^{\prime}(\tau)\rvert^{2}} - 1\right) \mathrm{d}\tau
		&\ge \sqrt{(t - a)^{2} + (g(t) - g(a))^{2}} - (t - a)\\
		&=  \frac{(g(t) - g(a))^{2}}{\sqrt{(t - a)^{2} + (g(t) - g(a))^{2}} + (t - a)} 
        \qquad\text{for } t>a.
	\end{split}
\end{equation*}
Since $g(t)=o(t)$ and $g(t)\to\infty$ as $t\to\infty$ (see \eqref{eq:condition_g_sublinear} and \eqref{eq:condition_g_infty}), it holds that
\begin{equation*}
	\lim_{t\to\infty}\frac{g(t)-g(a)}{t-a}=0 
	\quad\text{and} \quad 
	\lim_{t \to \infty} \frac{(g(t) - g(a))^{2}}{t-a}\cdot\frac{t}{g(t)^{2}} = 1. 
\end{equation*}
Therefore, there exists a $t_0>a$ such that 
\begin{equation*}
	V(t) \ge \frac{(g(t) - g(a))^{2}}{3(t - a)}\ge \frac{g(t)^2}{4t} \qquad \text{for } t> t_{0}.
\end{equation*}
Hence the proof of the lemma is completed.
\end{proof}

Now we estimate $N_g(r)$, $H_g(r)$, $J_g(r)$, and $T_g(r)$ in turn.

\begin{lemma}\label{lemma_1_potential_expansion}
Let $N_g$ be defined in \eqref{eq:Ng_def}. Then the limit
\begin{equation}\label{eq:Ng_estimate}
\lim_{r\to \infty}N_g(r) =: C_{N}\in \mathbb{R} \qquad \text{exists.}
\end{equation}
\end{lemma}

\begin{proof}
From the derivation of \eqref{eq:pe_F1}, we have
\begin{equation*}
N_g(r) = \frac{1}{2\pi} \int_{N} V(X)\partial_{\nu}G(X,Y) - G(X,Y)\partial_{\nu}V(X) \mathrm{d}\mathcal{H}^{2}(X),
\end{equation*}
where $Y=(r,g(r),0)\in \Gamma$. Note that 
$V$ and $\partial_{\nu}V$ are bounded on the fixed boundary $N$. Moreover, as $r\to\infty$, one has
\begin{equation*}
G(X,Y) = \frac{1}{\lvert X-Y\rvert} - \frac{1}{\lvert X\rvert} = O\left(r^{-1}\right) - \frac{1}{\lvert X\rvert}
\end{equation*}
and
\begin{equation*}
\partial_{\nu}G(X,Y) = \left(-\frac{X-Y}{\lvert X-Y\rvert^{3}} + \frac{X}{\lvert X\rvert^{3}}\right)\cdot \nu = O\left(r^{-2}\right) +\frac{X\cdot\nu}{\lvert X\rvert^{3}}.
\end{equation*}
Therefore, it holds that 
\begin{equation*}
\begin{split}
N_{g}(r) &= \frac{1}{2\pi} \int_{N} V(X)\left(O\left(r^{-2}\right) + \frac{X\cdot \nu}{\lvert X\rvert^{3}}\right) -  \left(O\left(r^{-1}\right) - \frac{1}{\lvert X\rvert}\right)\partial_{\nu}V(X)\mathrm{d}\mathcal{H}^{2}(X)\\
&= C_N + O(r^{-2})+O(r^{-1}) \qquad\text{as } r \to \infty,
\end{split}
\end{equation*}
where 
\begin{equation*}
C_N:= \frac{1}{2\pi} \int_{N} V(X)\frac{X\cdot \nu}{\lvert X\rvert^{3}}+ \frac{1}{\lvert X\rvert}\partial_{\nu}V(X)\mathrm{d}\mathcal{H}^{2}(X).
\end{equation*}
This proves the lemma.
\end{proof}

We have the following lower bound for the head term $H_g(r)$. 

\begin{lemma}\label{lem:Hg_lower_bound} 
Assume that $g$ satisfies \eqref{eq:condition_g_sublinear},  \eqref{eq:concave_fb}, and \eqref{eq:condition_g_unbounded}. Let $H_g$ be defined in \eqref{eq:Hg_def} and $V$ be as in \eqref{eq:Vt_def}. There exists a constant $C>0$ independent of $r$ such that
\begin{equation*}
	H_{g}(r) \ge -C - 4 \int_{a}^{\frac r2} \frac{V(t)^{2}}{t^{2}} \mathrm{d}t - \int_{a}^{\frac r2} \frac{g(t)g^{\prime}(t)}{t} \mathrm{d}t\qquad \text{for } r> 2t_0,
\end{equation*}
where $t_0$ is given by Lemma \ref{lem:Vg_relation}.
\end{lemma}

\begin{proof} 
Using Lemma \ref{lem:Fg_lower_bound} and $g^{\prime}\ge0$ on $(a,\infty)$ in \eqref{eq:condition_dg_sign} yields
\begin{equation}\label{eq:Fg_lower_bound1}
	F_{g}(t,\lambda;r) \ge \frac{g(r)(1-\cos\lambda)}{\lvert X-Y\rvert^{3}} + \frac{tg^{\prime}(t) - g(t)}{\lvert X\rvert^{3}}\ge -\frac{g(t)}{\lvert X\rvert^{3}} \qquad\text{for } t,\, r>a.
\end{equation}
This together with $\lvert X\rvert = \sqrt{t^{2} + g(t)^{2}} \ge t$ gives
\begin{equation}\label{eq:Fg_lower_bound2}
	\begin{split}
		H_{g}(r) &= \frac{1}{2\pi} \int_{0}^{2\pi}\int_{a}^{\frac r2} g(t)V(t)F_{g}(t,\lambda;r) + g(t)g^{\prime}(t) \left(\frac{1}{\lvert X-Y\rvert} - \frac{1}{\lvert X\rvert}\right) \mathrm{d}t\mathrm{d}\lambda\\
		&\ge - \int_{a}^{\frac r2} \frac{g(t)^{2}V(t)}{\lvert X\rvert^{3}}\mathrm{d}t - \int_{a}^{\frac r2} \frac{g(t)g^{\prime}(t)}{\lvert X\rvert} \mathrm{d}t\\
        &\ge - \int_{a}^{\frac r2} \frac{g(t)^{2}V(t)}{t^{3}}\mathrm{d}t - \int_{a}^{\frac r2} \frac{g(t)g^{\prime}(t)}{t} \mathrm{d}t
        \qquad\text{for } r>2a.
	\end{split}
\end{equation}
Note that $g(t)^2\le 4tV(t)$ for $t>t_0$ by Lemma \ref{lem:Vg_relation}. Hence one has
\begin{equation}\label{eq:Fg_lower_bound3}
\int_{a}^{\frac r2} \frac{g(t)^{2}V(t)}{t^{3}} \mathrm{d}t 
    \le C + 4\int_{t_{0}}^{\frac r2} \frac{V(t)^{2}}{t^{2}} \mathrm{d}t
      \le C + 4\int_{a}^{\frac r2} \frac{V(t)^{2}}{t^{2}} \mathrm{d}t
      \qquad \text{for } r >2t_{0},
\end{equation} 
where
\begin{equation*}
C=\int_{a}^{t_{0}} \frac{g(t)^{2}V(t)}{t^{3}} \mathrm{d}t>0.
\end{equation*}
The desired estimate follows by combining \eqref{eq:Fg_lower_bound2} and \eqref{eq:Fg_lower_bound3}. 
\end{proof}

Next we estimate the transition term $J_{g}(r)$ from below. 

\begin{lemma}\label{lem:Jg_lower_bound} 
Assume that $g$ satisfies \eqref{eq:condition_g_sublinear},  \eqref{eq:concave_fb}, and \eqref{eq:condition_g_unbounded}.  Let $J_g$ be defined in \eqref{eq:Jg_def} and $V$ be as in \eqref{eq:Vt_def}. There exists $R_0>\max\{2t_0,4a\}$ with $t_0$ as in Lemma \ref{lem:Vg_relation} such that
\begin{equation}\label{eq:Jg_lower_bound} 
J_{g}(r) \ge \frac{1}{9}g(r)g^{\prime}(r)\log\frac{r}{2g(r)} - 4\int_{\frac r2}^{2r} \frac{V(t)^{2}}{t^{2}} \mathrm{d}t - \int_{\frac r2}^{r} \frac{g(t)g^{\prime}(t)}{t} \mathrm{d}t\qquad \text{for } r> R_0.
\end{equation}
\end{lemma}

\begin{proof}
Since $g(t)=o(t)$ and $g^{\prime}(t)=o(1)$ as $t \to \infty$ (see \eqref{eq:condition_g_sublinear} and \eqref{eq:condition_dg_to0}), there exists $R_0>\max\{2t_0,4a\}$ such that 
\begin{equation}\label{eq:r0_1}
0<g(t)<\frac t4 \quad\text{and}\quad 0\le g^{\prime}(t)<1 \qquad\text{for } t>\frac {R_0}2.
\end{equation}
Throughout the remainder of this proof, let $r>R_0$. 
As in the proof of Lemma \ref{lem:Hg_lower_bound}, applying \eqref{eq:Fg_lower_bound1}, $\lvert X\rvert\ge t$, and Lemma \ref{lem:Vg_relation} yields
\begin{equation}\label{eq:Jg_lower_gVF}
\frac1{2\pi}\int_0^{2\pi}\int_{\frac r2}^{2r} g(t)V(t)F_{g}(t,\lambda;r) \mathrm{d}t\mathrm{d}\lambda \ge -\int_{\frac r2}^{2r}\frac{g(t)^{2}V(t)}{t^{3}} \mathrm{d}t \ge -4\int_{\frac r2}^{2r} \frac{V(t)^{2}}{t^{2}} \mathrm{d}t.
\end{equation}
Moreover, 
\begin{equation}\label{eq:X-Y_X_sign}
	\begin{split}
		\lvert X-Y\rvert^{2}-\lvert X\rvert^{2} 
		&=\left((t-r)^{2} + (g(t) - g(r))^{2} + 2g(t)g(r)(1-\cos\lambda)\right)-\left(t^{2} + g(t)^{2}\right)\\
		&=-2tr+r^2+g(r)^2-2g(t)g(r)\cos\lambda
		<0 \qquad\text{for } t>r,
	\end{split}
\end{equation} 
where we used $g(r)^2<r^2/16$ and $g(t)g(r)<tr/16$. 
This implies
\begin{equation}\label{eq:X-Y_X_compare}
	\frac1{\lvert X-Y\rvert}>\frac1{\lvert X\rvert} \qquad\text{for } t>r,
\end{equation}
and leads to
\begin{equation}\label{eq:Jg_lower_gdg}
\frac{1}{2\pi} \int_{0}^{2\pi} \int_{r}^{2r} g(t)g^{\prime}(t) \left(\frac{1}{\lvert X-Y\rvert} - \frac{1}{\lvert X\rvert}\right) \mathrm{d}t \mathrm{d}\lambda\ge 0.
\end{equation}
Then it follows from the definition of $J_g$ in \eqref{eq:Jg_def}, \eqref{eq:Jg_lower_gVF}, and \eqref{eq:Jg_lower_gdg} that
\begin{equation}\label{eq:Jg_lower_1}\begin{split}
		J_{g}(r)& =\frac{1}{2\pi} \int_{0}^{2\pi} \int_{\frac r2}^{2r} g(t)V(t)F_{g}(t,\lambda;r) + g(t)g^{\prime}(t)\left( \frac{1}{\lvert X-Y\rvert} - \frac{1}{\lvert X\rvert} \right) \mathrm{d}t\mathrm{d}\lambda\\
		&\ge -4 \int_{\frac r2}^{2r} \frac{V(t)^{2}}{t^{2}} \mathrm{d}t + \frac{1}{2\pi} \int_{0}^{2\pi} \int_{\frac r2}^{r} g(t)g^{\prime}(t) \left(\frac{1}{\lvert X-Y\rvert} - \frac{1}{\lvert X\rvert}\right) \mathrm{d}t \mathrm{d}\lambda.
\end{split}\end{equation}

To estimate the last integral in \eqref{eq:Jg_lower_1}, we note that by the concavity of $g$ on $(a,\infty)$ (see \eqref{eq:concave_fb}), 
\begin{equation*}
	\begin{split}
		g(t)g^{\prime}(t)\ge
		\left(\frac{t-a}{r-a}g(r) + \frac{r-t}{r-a} g(a)\right)g^{\prime}(r) \ge \frac{t-a}{r-a}g(r)g^{\prime}(r)\ge \frac13g(r)g^{\prime}(r) \qquad\text{for } t\in\left(\frac r2,r\right).
	\end{split}
\end{equation*}
In addition, using \eqref{eq:condition_dg_sign}, $\lvert g^{\prime}(r/2)\lvert<1$, and $g(r)\le \lvert t-r\rvert$ for $t\in(r/2,r-g(r))$, we get for $t\in(r/2,r-g(r))$
\begin{equation*}
	\begin{split}
		\lvert X-Y\rvert^{2} &= (t - r)^{2} + (g(t) - g(r))^{2} + 2g(t)g(r)(1- \cos\lambda)\\
		&\le (t-r)^{2} \left(1 + \lvert g^{\prime}(r/2)\rvert^{2}\right) + 4g(r)^{2}
		\le  6(t-r)^2.
	\end{split}
\end{equation*}
Thus 
\begin{equation}\label{eq:Jg_lower_2}
	\frac{1}{2\pi} \int_{0}^{2\pi} \int_{\frac r2}^{r}\frac{ g(t)g^{\prime}(t)}{\lvert X-Y\rvert}\mathrm{d}t\mathrm{d}\lambda\ge \int_{\frac r2}^{r-g(r)} \frac{g(r)g^{\prime}(r)}{3\sqrt{6}\lvert t-r\rvert} \mathrm{d}t \ge \frac{1}{9}g(r)g^{\prime}(r) \log\frac{r}{2g(r)}.
\end{equation}
On the other hand, by $\lvert X\rvert\ge t$,
\begin{equation}\label{eq:Jg_lower_3}
	-\frac{1}{2\pi}\int_{0}^{2\pi}\int_{\frac r2}^{r}\frac{g(t)g^{\prime}(t)}{\lvert X\rvert}\mathrm{d}t \mathrm{d}\lambda\ge
	-\int_{\frac r2}^{r}\frac{g(t)g^{\prime}(t)}{t}\mathrm{d}t.
\end{equation}
Combining \eqref{eq:Jg_lower_1}, \eqref{eq:Jg_lower_2}, and \eqref{eq:Jg_lower_3} gives \eqref{eq:Jg_lower_bound}. Hence the proof of the lemma is completed.
\end{proof}

It remains to estimate the tail term $T_{g}(r)$. 

\begin{lemma}\label{lem:Tg_lower_bound}
Assume that $g$ satisfies \eqref{eq:condition_g_sublinear},  \eqref{eq:concave_fb}, and \eqref{eq:condition_g_unbounded}.  Let $T_g$ be defined in \eqref{eq:Tg_def}. There exist constants $C>0$ independent of $r$ and $R_1>R_0$ with $R_0$ as in Lemma \ref{lem:Jg_lower_bound} such that
\begin{equation*}
T_{g}(r) \ge -C\frac{g(r)^3}{r^{2}}\qquad \text{for } r>R_1.
\end{equation*}
\end{lemma}

\begin{proof} 
We divide the proof into three steps.

\textit{Step 1. Lower bound for $\int_{2r}^{\infty} g(t)V(t)F_{g}(t,\lambda;r) \mathrm{d}t$.} 
Let $r>R_0$. By the concavity of $g$ in \eqref{eq:concave_fb}, one has
\begin{equation*}
g^{\prime}(t) \le \frac{g(t) - g(a)}{t - a}\le \frac{g(t)}{t - a} \qquad\text{for } t>a, 
\end{equation*}
and hence
\begin{equation*}
tg^{\prime}(t) - g(t) \le ag^{\prime}(t) \qquad\text{for } t>a. 
\end{equation*}
This, together with \eqref{eq:X-Y_X_compare}, gives
\begin{equation*}
	(tg^{\prime}(t) - g(t)) \left(\frac{1}{\lvert X\rvert^{3}} - \frac{1}{\lvert X-Y\rvert^{3}}\right)\ge ag^{\prime}(t)\left(\frac{1}{\lvert X\rvert^{3}} - \frac{1}{\lvert X-Y\rvert^{3}}\right) \ge -\frac{ag^{\prime}(t)}{\lvert X-Y\rvert^{3}} \qquad\text{for } t>r.
\end{equation*}
Note that
\begin{equation*}
\lvert X-Y\rvert^{3} \ge (t-r)^{3} \ge \frac{1}{8}t^3  \qquad\text{for } t>2r.
\end{equation*}
Then it follows from the definition of $F_g$ in \eqref{eq:Fg_def} and $g^{\prime}\ge 0$ on $(a,\infty)$ (see \eqref{eq:condition_dg_sign}) that
\begin{equation}\label{eq:Tg_lower_Fg}
	\begin{split}
		F_{g}(t,\lambda;r) &= \frac{rg^{\prime}(t) - g(r)\cos\lambda}{\lvert X-Y\rvert^{3}} + (tg^{\prime}(t) - g(t)) \left(\frac{1}{\lvert X\rvert^{3}} - \frac{1}{\lvert X-Y\rvert^{3}}\right)\\
		&\ge \frac{rg^{\prime}(t) - g(r)}{\lvert X-Y\rvert^{3}} - \frac{ag^{\prime}(t)}{\lvert X-Y\rvert^{3}}
		\ge -\frac{g(r)}{\lvert X-Y\rvert^{3}}
		\ge  -\frac{8g(r)}{t^3} \qquad\text{for } t>2r.
	\end{split}
\end{equation}

On the other hand, since $g^{\prime}$ is bounded on $(2a,\infty)$ (see Section \ref{subsec:g_properties}),   
using the expression of $V(t)$ in \eqref{eq:Vt_def} and the fact that $V(2a)$ is finite, we get \begin{equation}\label{eq:Tg_lower_Vt}\begin{split}
		V(t)&=V(2a)+\int_{2a}^{t} \left(\sqrt{1 + \lvert g^{\prime}(\tau)\rvert^{2}} - 1\right) \mathrm{d}\tau
		\le V(2a)+\frac12\int_{2a}^{t}\lvert g^{\prime}(\tau)\rvert^{2} \mathrm{d}\tau\\
		&\le C\left(1+\int_{2a}^{t}g^{\prime}(\tau) \mathrm{d}\tau\right)
		\le Cg(t) \qquad\text{for } t>2a.
\end{split}\end{equation}
Applying \eqref{eq:Tg_lower_Fg} together with \eqref{eq:Tg_lower_Vt} and then integrating by parts yields
\begin{equation*}
	\begin{split}
		\int_{2r}^{\infty} g(t)V(t)F_{g}(t,\lambda;r) \mathrm{d}t 
		&\ge-Cg(r)\int_{2r}^{\infty} \frac{g(t)^{2}}{t^{3}} \mathrm{d}t\\
		&=Cg(r) \left(\lim_{t \to \infty}\frac{g(t)^{2}}{t^{2}}-\frac{g(2r)^{2}}{(2r)^{2}}-2\int_{2r}^{\infty}\frac{g(t)g^{\prime}(t)}{t^{2}}\mathrm{d}t \right).
	\end{split}
\end{equation*}
Note that $g(t)^2/t^2\to0$ as $t\to\infty$ by  \eqref{eq:condition_g_sublinear}. In addition, by the concavity of $g$ on $(a,\infty)$,
\begin{equation}\label{eq:g2r_r}
g(2r)\le g(a)+(2r-a)\cdot\frac{g(r)-g(a)}{r-a}\le 3g(r) \qquad\text{for } r>4a.
\end{equation}
Therefore, it holds that
\begin{equation}\label{eq:Tg_lower_bound_1}
	\begin{split}
		\int_{2r}^{\infty} g(t)V(t)F_{g}(t,\lambda;r) \mathrm{d}t 
		\ge -C\frac{g(r)^3}{r^{2}} - Cg(r) \int_{2r}^{\infty} \frac{g(t)g^{\prime}(t)}{t^{2}} \mathrm{d}t.
	\end{split}
\end{equation}

\textit{Step 2. Lower bound for $\int_{2r}^{\infty} g(t)g^{\prime}(t)\left(\frac{1}{\lvert X-Y\rvert} - \frac{1}{\lvert X\rvert}\right) \mathrm{d}t$.} 
For $t>2r$ with $r>R_0$, since $g(r)^2\le g(t)g(r)\le tr/16$ by \eqref{eq:r0_1} and $r/(2t)<1/4$,
\begin{equation*}
	\begin{split}
		2X\cdot Y-\lvert Y\rvert^{2} &= 2(tr + g(t)g(r)\cos\lambda)- (r^{2} + g(r)^{2})
		\ge 2tr\left(1 -\frac{r}{2t}-\frac1{8}\right)\ge\frac 54tr,
	\end{split}
\end{equation*}
and since $2X\cdot Y<\lvert X\rvert^2+\lvert Y\rvert^{2}$ for $X\neq Y$, it follows that
\begin{equation*}
0<\frac{2X\cdot Y-\lvert Y\rvert^{2}}{\lvert X\rvert^2}<1.
\end{equation*}
Moreover, in view of \eqref{eq:condition_g_sublinear}, one has
\begin{equation*}
\lvert X\rvert = \sqrt{t^{2} + g(t)^{2}}=(1+o(1))t \qquad\text{as } t\to\infty.
\end{equation*}
Hence, by the elementary inequality $(1-z)^{-1/2}-1\ge z/2$ for $z\in(0,1)$, there exists $R_1>R_0$ such that for $t>2r$ with $r>R_1$,
\begin{equation*}
\frac{1}{\lvert X-Y\rvert} - \frac{1}{\lvert X\rvert}
= \frac{1}{\lvert X\rvert} \left(\left(1-\frac{2X\cdot Y- \lvert Y\rvert^{2}}{\lvert X\rvert^{2}}\right)^{-1/2} - 1\right)
\ge \frac{2X\cdot Y - \lvert Y\rvert^{2}}{2\lvert X\rvert^{3}}
\ge \frac r{2t^2}.
\end{equation*}
As a consequence, 
\begin{equation}\label{eq:Tg_lower_bound_2}
\int_{2r}^{\infty} g(t)g^{\prime}(t) \left(\frac{1}{\lvert X-Y\rvert} - \frac{1}{\lvert X\rvert}\right) \mathrm{d}t \ge \frac{r}{2} \int_{2r}^{\infty} \frac{g(t)g^{\prime}(t)}{t^{2}} \mathrm{d}t \qquad\text{for } r>R_1.
\end{equation}

\textit{Step 3. Lower bound for $T_g(r)$.} Combining \eqref{eq:Tg_lower_bound_1} with \eqref{eq:Tg_lower_bound_2} and enlarging $R_1$ if necessary yields         
\begin{equation*}\begin{split}
		T_{g}(r)&= \frac{1}{2\pi} \int_{0}^{2\pi} \int_{2r}^{\infty} g(t)V(t)F_{g}(t,\lambda;r) + g(t)g^{\prime}(t)\left( \frac{1}{\lvert X-Y\rvert} - \frac{1}{\lvert X\rvert} \right) \mathrm{d}t\mathrm{d}\lambda \\
		&\ge -C\frac{g(r)^3}{r^{2}} + \left(\frac{r}{2}-Cg(r)\right)\int_{2r}^{\infty} \frac{g(t)g^{\prime}(t)}{t^{2}} \mathrm{d}t \ge -C\frac{g(r)^3}{r^{2}} \qquad\text{for } r>R_1, 
\end{split}\end{equation*}
where $g(r)=o(r)$ as $r\to\infty$ is used in the last inequality. This finishes the proof of the lemma. 
\end{proof}

\subsection{Proof of the initial growth estimate}\label{subsec:initial_growth}

With Lemmas \ref{lemma_1_potential_expansion}-\ref{lem:Tg_lower_bound} at hand, we derive  the following integro-differential inequality, which is crucial to obtain the $1/2+\varepsilon$- power growth of $g$.

\begin{lemma}\label{lem:r_root_differential_ineq}
Assume that $g$ satisfies \eqref{eq:condition_g_sublinear},  \eqref{eq:concave_fb}, and \eqref{eq:condition_g_unbounded}.  Let $V$ be as in \eqref{eq:Vt_def}. There exist constants $C>0$ independent of $r$ and $R_2>R_1$ with $R_1$ as in Lemma \ref{lem:Tg_lower_bound} such that
\begin{equation}\label{eq:r_root_differential_ineq}
	\frac{1}{9}g(r)g^{\prime}(r)\log\frac{r}{2g(r)} \le C + 4\int_{a}^{2r} \frac{V(t)^{2}}{t^{2}} \mathrm{d}t + \int_{a}^{r} \frac{g(t)g^{\prime}(t)}{t} \mathrm{d}t\qquad \text{for } r>R_2.
\end{equation}
\end{lemma}

\begin{proof}
It follows from the identity \eqref{eq:pe_main_NHJT} and Lemmas \ref{lemma_1_potential_expansion}-\ref{lem:Tg_lower_bound} that there exist positive constants $C$ and $R_1$ such that
\begin{equation*}\begin{split}
& -V(r) + 2V(0) \\
\ge&  - C - 4\int_{a}^{2r} \frac{V(t)^{2}}{t^{2}} \mathrm{d}t - \int_{a}^{r} \frac{g(t)g^{\prime}(t)}{t} \mathrm{d}t
+ \frac{1}{9}g(r)g^{\prime}(r)\log\frac{r}{2g(r)} - C\frac{g(r)^3}{r^{2}} 
\qquad \text{for } r > R_1.
\end{split}\end{equation*}
This implies 
\begin{equation*}
	\frac{1}{9}g(r)g^{\prime}(r)\log\frac{r}{2g(r)} + V(r) - C\frac{g(r)^3}{r^{2}} \le C+ 4\int_{a}^{2r} \frac{V(t)^{2}}{t^{2}} \mathrm{d}t + \int_{a}^{r} \frac{g(t)g^{\prime}(t)}{t} \mathrm{d}t\qquad \text{for } r > R_1.
\end{equation*}
By virtue of Lemma \ref{lem:Vg_relation} and $g(r)=o(r)$ as $r\to\infty$ from \eqref{eq:condition_g_sublinear}, there exists $R_2>R_1$ such that 
\begin{equation*}
	V(r) - C\frac{g(r)^3}{r^{2}}\ge \frac{g(r)^2}{4r}-C\frac{g(r)^3}{r^{2}}=
	\frac{g(r)^2}{r}\left(\frac14-C\frac{g(r)}{r}\right)\ge 0 \qquad \text{for } r > R_2.
\end{equation*}
Hence \eqref{eq:r_root_differential_ineq} holds and the proof of the lemma is completed.
\end{proof}

We are now in a position to prove the $1/2+\varepsilon$-power upper bound for $g$.

\begin{proposition}\label{prop:r_root_growth}
Assume that $g$ satisfies \eqref{eq:condition_g_sublinear} and \eqref{eq:concave_fb}. Then
\begin{equation*}
	\limsup_{r\to \infty} \frac{g(r)}{r^{\alpha}} < \infty\qquad \text{for every } \alpha \in \left(\frac12,1\right).
\end{equation*}
\end{proposition}

\begin{proof}
If $g$ is bounded, the conclusion follows trivially. Hence, without loss of generality, we assume that $g$ satisfies \eqref{eq:condition_g_unbounded}. 
We argue by contradiction. Suppose that there exists an $\alpha \in(1/2,1)$ such that
\begin{equation}\label{eq:r_root_1}
	\limsup_{r \to \infty} \frac{g(r)}{r^{\alpha}} = \infty.
\end{equation}
We derive a contradiction in four steps.

\textit{Step 1. Set up.} For each $i\in \mathbb{N}$, let $r_{i}\in [a,i]$ be such that
\begin{equation}\label{eq:Ki_def}
\frac{g(r_{i})^2}{r_{i}^{2\alpha}}= \max_{r\in[a,i]} \frac{g(r)^{2}}{r^{2\alpha}}=:K_{i}.
\end{equation}
Then $K_{i},\, r_i\to\infty$ as $i \to \infty$ by \eqref{eq:r_root_1}. 
Consider 
\begin{equation*}
	\Phi_{i}(r) := g(r)^{2} - K_{i}r^{2\alpha}.
\end{equation*}
It follows from \eqref{eq:Ki_def} that $\Phi_{i} \le 0$ on $[a,i]$ and $\Phi_{i}(r_i)=0$. Hence 
\begin{equation}\label{equation_1_g(r)<r^{1/2}}
	\Phi_{i}^{\prime}(r_{i}) = 2g(r_{i})g^{\prime}(r_{i}) - 2\alpha K_{i}r_{i}^{2\alpha-1} \ge 0,
\end{equation}
where $r_i>a$ for large $i$.

\textit{Step 2. Upper bound for  $\int_{a}^{2r_i} \frac{V(t)^{2}}{t^{2}} \mathrm{d}t$.} 
Recall that the function $V$ is given by \eqref{eq:Vt_def}. On $(a,\infty)$, it follows from the concavity of $g$ in \eqref{eq:concave_fb}  that
\begin{equation*}
	V^{\prime}(t)=\sqrt{1 + \lvert g^{\prime}(t)\rvert^{2}}-1
\end{equation*} 
is non-increasing, and that
\begin{equation*}
	\frac{V(t)}{t-a} = \frac{1}{t-a}\int_{a}^{t}V^{\prime}(\tau) \mathrm{d}\tau    
\end{equation*}
is also non-increasing. This leads to
\begin{equation*}
	\begin{split}
		\int_{r}^{2r} \frac{V(t)^{2}}{t^{2}} \mathrm{d}t &\le \int_{r}^{2r} \frac{V(t)^{2}}{(t-a)^{2}}\mathrm{d}t 
		\le 2\int_{\frac r2}^{r} \frac{V(t)^{2}}{(t-a)^{2}}\mathrm{d}t
		\le 8\int_{\frac r2}^{r} \frac{V(t)^{2}}{t^{2}} \mathrm{d}t \qquad\text{for } r>4a.
	\end{split}
\end{equation*}
This implies  
\begin{equation*}
\int_{2a}^{2r} \frac{V(t)^{2}}{t^{2}} \mathrm{d}t \le \int_{2a}^{r} \frac{V(t)^{2}}{t^{2}} \mathrm{d}t + 8\int_{\frac r2}^{r} \frac{V(t)^{2}}{t^{2}}\mathrm{d}t 
\le 9\int_{2a}^{r} \frac{V(t)^{2}}{t^{2}} \mathrm{d}t  \qquad\text{for } r>4a.
\end{equation*}
Furthermore, by $V(t)\le Cg(t)$ for $t>2a$ (see  \eqref{eq:Tg_lower_Vt}) and \eqref{eq:Ki_def}, for large $i$, it holds that 
\begin{equation}\label{eq:Vt2_ri}
	\begin{split}
		\int_{a}^{2r_{i}} \frac{V(t)^{2}}{t^{2}}\mathrm{d}t &\le \int_{a}^{2a} \frac{V(t)^{2}}{t^{2}}\mathrm{d}t + 9\int_{2a}^{r_{i}} \frac{V(t)^{2}}{t^{2}} \mathrm{d}t 
		\le C \left(1+\int_{2a}^{r_{i}} \frac{g(t)^{2}}{t^{2}} \mathrm{d}t\right)\\ 
		&\le C\left(1+\int_{2a}^{r_{i}} \frac{K_{i}t^{2\alpha}}{t^{2}}\mathrm{d}t\right) 
		\le C\left(1+\frac{K_{i}}{2\alpha-1}r_{i}^{2\alpha-1}\right).
\end{split}\end{equation}

\textit{Step 3. Upper bound for   $\int_{a}^{r_i}\frac{g(t)g^{\prime}(t)}{t}\mathrm{d}t$.}
For each $i$, using integration by parts gives
\begin{equation*}
\int_{a}^{r_i} \frac{g(t)g^{\prime}(t)}{t}\mathrm{d}t =\frac{g(r_i)^{2}}{2r_i} - \frac{g(a)^{2}}{2a} + \int_{a}^{r_i} \frac{g(t)^{2}}{2t^{2}}\mathrm{d}t \le \frac{g(r_i)^{2}}{2r_i} + \int_{a}^{r_i} \frac{g(t)^{2}}{2t^{2}}\mathrm{d}t.
\end{equation*}
Thus, we obtain from \eqref{eq:Ki_def} that 
\begin{equation}\label{eq:gdgt_ri}
	\int_{a}^{r_{i}} \frac{g(t)g^{\prime}(t)}{t}\mathrm{d}t \le \frac{K_{i}r_{i}^{2\alpha}}{2r_{i}} +  \int_{a}^{r_{i}} \frac{K_it^{2\alpha}}{2t^{2}} \mathrm{d}t \le \frac{K_{i}}{2} \left(1 + \frac{1}{2\alpha-1}\right) r_{i}^{2\alpha-1} = \frac{\alpha K_{i}}{2\alpha-1}r_{i}^{2\alpha-1}.
\end{equation}

\textit{Step 4. Contradiction.} Combining Lemma \ref{lem:r_root_differential_ineq}, \eqref{eq:Vt2_ri} and \eqref{eq:gdgt_ri} yields that, for large $i$,
\begin{equation}\label{eq:r_root_2}
	\begin{split}
		\frac{1}{9}g(r_{i})g^{\prime}(r_{i})\log\frac{r_{i}}{2g(r_{i})} &\le C + 4\int_{a}^{2r_i} \frac{V(t)^{2}}{t^{2}} \mathrm{d}t + \int_{a}^{r_i} \frac{g(t)g^{\prime}(t)}{t} \mathrm{d}t\\
		&\le C\left(1+\frac{K_{i}}{2\alpha-1}r_{i}^{2\alpha-1}\right).
	\end{split}
\end{equation}
In view of \eqref{equation_1_g(r)<r^{1/2}}, the left-hand side of \eqref{eq:r_root_2} is bounded below by
\begin{equation*}
\frac19g(r_{i})g^{\prime}(r_{i})\log\frac{r_{i}}{2g(r_{i})} \ge \frac{\alpha}9 K_{i}r_{i}^{2\alpha-1}\log\frac{r_{i}}{2g(r_{i})}.
\end{equation*}
Substituting this lower bound into the first inequality in \eqref{eq:r_root_2} and dividing both sides by $K_ir_i^{2\alpha-1}$ gives
\begin{equation*}
	\frac{\alpha}{9}\log\frac{r_{i}}{2g(r_{i})} \le C\left(\frac1{K_ir_i^{2\alpha-1}}+\frac1{2\alpha-1}\right) 
 \qquad \text{for large } i. 
\end{equation*}
Since both $r_i$ and $r_i/g(r_i)$ tend to infinity as $i\to\infty$, we arrive at a contradiction. Hence the proof of the proposition is completed.
\end{proof}

\section{Refined growth estimate for the free boundary}\label{sec:refined_growth}

With the help of Proposition \ref{prop:r_root_growth}, we improve the lower bounds from Lemmas \ref{lem:Hg_lower_bound}-\ref{lem:Tg_lower_bound} to refine the growth estimate for the free boundary. These refined estimates result in the upper bound $g(r)\le r^{1/2}(\log r)^{-1/4+\varepsilon}$ at infinity.

\subsection{Refined lower bounds on the decomposed integrals}\label{subsec:refined_lower}

We first improve the lower bounds of the head term $H_{g}(r)$ and the tail term $T_{g}(r)$ obtained in Lemmas \ref{lem:Hg_lower_bound} and \ref{lem:Tg_lower_bound}, respectively.

\begin{lemma}\label{lem:Hg_refine_lower}
Assume that $g$ satisfies \eqref{eq:condition_g_sublinear},  \eqref{eq:concave_fb}, and \eqref{eq:condition_g_unbounded}.   Let $H_g$ be defined in \eqref{eq:Hg_def}. There exists a constant $C>0$ independent of $r$ such that
\begin{equation*}
H_{g}(r) \ge -C - \int_{a}^{\frac r2} \frac{g(t)g^{\prime}(t)}{t} \mathrm{d}t \qquad\text{for }r>2t_0,
\end{equation*}
where $t_0$ is the same as the one in Lemma \ref{lem:Vg_relation}.
\end{lemma}

\begin{proof}
By Lemma \ref{lem:Hg_lower_bound}, it suffices to show 
\begin{equation*}
	\int_{a}^{\infty} \frac{V(t)^{2}}{t^{2}} \mathrm{d}t < \infty,
\end{equation*}
where $V$ is given by \eqref{eq:Vt_def}. 
Fix $\alpha \in(1/2,3/4)$. Using the concavity of $g$ from  \eqref{eq:concave_fb} and Proposition \ref{prop:r_root_growth} yields
\begin{equation}\label{eq:Vt_upper_alpha}\begin{split}
		V(t)-V(2a)&=\int_{2a}^{t}\left(\sqrt{1+\lvert g^{\prime}(\tau)\rvert^2}-1\right)\mathrm d\tau 
		\le \frac12\int_{2a}^{t}\lvert g^{\prime}(\tau)\rvert^2\mathrm d\tau\\
		&\le \frac12\int_{2a}^{t}\frac{(g(\tau)-g(a))^2}{(\tau-a)^2}\mathrm d\tau
		\le C\int_{2a}^{t}\tau^{2\alpha-2}\mathrm d\tau\le Ct^{2\alpha-1} \qquad\text{for }t>2a.
\end{split}\end{equation}
Hence it follows that
\begin{equation*}\begin{split}
		\int_{a}^{\infty} \frac{V(t)^{2}}{t^{2}} \mathrm{d}t \le C\int_{a}^{\infty}\frac{V(2a)^2+t^{4\alpha-2}}{t^2}\mathrm d t <\infty.
\end{split}\end{equation*}
This finishes the proof of the lemma.
\end{proof}

\begin{lemma}\label{lem:Tg_refine_lower}
Assume that $g$ satisfies \eqref{eq:condition_g_sublinear},  \eqref{eq:concave_fb}, and \eqref{eq:condition_g_unbounded}.  Let $T_g$ be defined in \eqref{eq:Tg_def}. For each $\varepsilon > 0$, there exists  $R_{\varepsilon} > 4a$ such that
\begin{equation*}
	T_{g}(r) \ge -\varepsilon
	\qquad \text{for } r > R_{\varepsilon}.
\end{equation*}
\end{lemma}

\begin{proof}
Fix $\alpha \in(1/2,2/3)$. From Lemma \ref{lem:Tg_lower_bound} and Proposition \ref{prop:r_root_growth}, for sufficiently large $r$, we have
\begin{equation*}
	T_{g}(r) \ge -C\frac{g(r)^{3}}{r^{2}} \ge  -Cr^{3\alpha - 2}.
\end{equation*}
Since $3\alpha-2<0$, the desired conclusion follows.
\end{proof}

We now turn to the transition term $J_g(r)$ defined in \eqref{eq:Jg_def}.
We estimate the two parts in $J_g(r)$ separately. 

\begin{lemma}\label{lem:Jg_refine_gvF}
Assume that $g$ satisfies \eqref{eq:condition_g_sublinear},  \eqref{eq:concave_fb}, and \eqref{eq:condition_g_unbounded}. Let $V$ and $F_g$ be as in \eqref{eq:Vt_def} and \eqref{eq:Fg_def}, respectively. For each $\varepsilon\in(0,1)$, there exists $R_{\varepsilon}>4a$ such that
\begin{equation*}
	\frac{1}{2\pi} \int_{0}^{2\pi} \int_{\frac r2}^{2r} g(t)V(t)F_{g}(t,\lambda;r) \mathrm{d}t\mathrm{d}\lambda \ge (1-\varepsilon)V(r) - \varepsilon\qquad  \text{for } r>R_{\varepsilon}.
\end{equation*}
\end{lemma}

\begin{proof}
The proof is divided into three steps.

\textit{Step 1. Splitting of the integral.}  
Note that $g^{\prime}\ge 0$ on $(a,\infty)$ (see   \eqref{eq:condition_dg_sign}) and  $\lvert X\rvert=(t^2+g(t)^2)^{1/2}\ge t$. By Lemma \ref{lem:Fg_lower_bound}, 		\begin{equation}\label{eq:improved_gVF_int1}
	F_{g}(t,\lambda;r) \ge \frac{g(r)(1-\cos\lambda)}{\lvert X-Y\rvert^{3}} + \frac{tg^{\prime}(t)-g(t)}{\lvert X\rvert^{3}} \ge \frac{g(r)(1-\cos\lambda)}{\lvert X-Y\rvert^{3}} - \frac{g(t)}{t^{3}} \qquad\text{for } t,\, r>a.
\end{equation}
Using \eqref{eq:Vt_upper_alpha} and Proposition \ref{prop:r_root_growth} with $\alpha \in (1/2,3/4)$ yields
\begin{equation}\label{eq:improved_gVF_int2}\begin{split}
\frac{1}{2\pi} \int_{0}^{2\pi}\int_{\frac r2}^{2r} g(t)V(t)\frac{g(t)}{t^{3}}\mathrm{d}t\mathrm{d}\lambda
&\le C\int_{\frac r2}^{2r} (V(2a) + Ct^{2\alpha-1})t^{2\alpha-3}\mathrm{d}t \\
&\le C(r^{2\alpha-2}+r^{4\alpha-3})=o(1) 
\qquad\text{as } r\to\infty.
\end{split}\end{equation}
In the remaining steps, we show that for each $\varepsilon\in(0,1)$ there exists $R_{\varepsilon}>4a$ such that
\begin{equation}\label{eq:improved_gVF_int3}
	\frac{1}{2\pi} \int_{0}^{2\pi} \int_{\frac r2}^{2r} g(t)V(t)\frac{g(r)(1-\cos\lambda)}{\lvert X-Y\rvert^{3}}\mathrm{d}t\mathrm{d}\lambda \ge (1-\varepsilon)V(r)\qquad \text{for } r > R_{\varepsilon}.
\end{equation}
Then the desired result follows from  \eqref{eq:improved_gVF_int1}-\eqref{eq:improved_gVF_int3}.

\textit{Step 2. Localization of the integral in \eqref{eq:improved_gVF_int3}.} 
For each fixed $M>0$, since $g(r)=o(r)$ as $r\to\infty$ by \eqref{eq:condition_g_sublinear}, there exists a constant $R_M>4a$ depending on $M$ such that
\begin{equation*}
	\frac{r}{2} < r-Mg(r) < r+Mg(r) < 2r\qquad \text{for } r > R_M.
\end{equation*}
Hence, 
\begin{equation}\label{eq:improved_gVF_int4}
	\int_{\frac r2}^{2r} g(t)V(t)\frac{g(r)(1-\cos\lambda)}{\lvert X-Y\rvert^{3}} \mathrm{d}t \ge \int_{r-Mg(r)}^{r+Mg(r)} g(t)V(t)\frac{g(r)(1-\cos\lambda)}{\lvert X-Y\rvert^{3}} \mathrm{d}t\qquad \text{for } r > R_M.
\end{equation}

It follows from $g^{\prime}\ge0$ on $(a,\infty)$ that 
\begin{equation}\label{eq:gV_increasing}
    g(t) \ge g(r) \quad\text{and}\quad V(t) \ge V(r)
    \qquad\text{for } t\ge r>a.
\end{equation}
On the other hand, for $t\in(r-Mg(r), r)$, by the concavity of $g$ in \eqref{eq:concave_fb} and the fact that $g^{\prime}(r)=o(1)$ as $r\to\infty$ in \eqref{eq:condition_dg_to0}, we have
\begin{equation}\label{eq:improved_gVF_int5}
0 \le g(r) - g(t)\le g^{\prime}(r/2) \lvert r-t\rvert  \le g^{\prime}(r/2) Mg(r)=o(g(r)) \qquad \text{as } r \to \infty.
\end{equation}
Moreover, using
Lemma \ref{lem:Vg_relation} and \eqref{eq:condition_g_sublinear} in the last two of the following inequalities, 
\begin{equation*}
	\begin{split}
		0 \le V(r) - V(t)&= \int_{t}^{r} \left(\sqrt{1 + \lvert g^{\prime}(\tau)\rvert^{2}} - 1\right)\mathrm{d}\tau
		\le\frac12\int_{t}^{r}\lvert g^{\prime}(\tau)\rvert^{2}\mathrm{d}\tau \\
		&\le \frac12\int_{t}^{r} \frac{(g(\tau)-g(a))^{2}}{(\tau-a)^{2}}\mathrm{d}\tau
		\le \frac{g(r)^{2}}{2} \int_{r-Mg(r)}^{r} \frac{1}{(\tau-a)^{2}} \mathrm{d}\tau\\
		&\le 2rV(r)\left(\frac{1}{r-Mg(r)-a}-\frac{1}{r-a}\right)\le \frac{8Mg(r)}{r}V(r)= o(V(r)) \qquad \text{as } r \to \infty.
	\end{split}
\end{equation*}	
Therefore, it follows that
\begin{equation}\label{eq:improved_gVF_int6}
	\int_{r-Mg(r)}^{r+Mg(r)} g(t)V(t)\frac{g(r)(1-\cos\lambda)}{\lvert X-Y\rvert^{3}} \mathrm{d}t\ge (1-o(1))g(r)^{2}V(r)\int_{r-Mg(r)}^{r+Mg(r)} \frac{1-\cos\lambda}{\lvert X-Y\rvert^{3}} \mathrm{d}t
	\qquad \text{as } r \to \infty.
\end{equation}
Furthermore, for $t\in(r-Mg(r),r+Mg(r))$, similarly to  \eqref{eq:improved_gVF_int5}, one has
\[
\lvert g(t)-g(r)\rvert\le g^{\prime}(r/2)\lvert t-r\rvert.
\]
Consequently,
\begin{equation}\label{eq:X-Y_order}
	\begin{split}
		\lvert X-Y\rvert^{2} &= (t-r)^{2} + (g(t)-g(r))^{2} + 2g(t)g(r)(1-\cos\lambda)\\
		&\le \left(1+(g^{\prime}(r/2))^2\right)(t-r)^{2} + 2g(r)^{2}(1-\cos\lambda) + 2g^{\prime}(r/2)\lvert t-r\rvert g(r)(1-\cos\lambda)\\
		&\le (1 +2g^{\prime}(r/2)) \left((t-r)^{2} + 2g(r)^{2}(1-\cos\lambda)\right) \qquad \text{as } r\to\infty,
\end{split}\end{equation}
where $0<g^{\prime}(r/2)<1$ for large $r$ and 
\begin{equation}\label{eq:X-Y_order1}
2\lvert t-r\rvert g(r)(1-\cos\lambda)
\le \lvert t-r\rvert^2+2g(r)^2(1-\cos\lambda)
\end{equation}
have been used in the last step. 
Then with the substitution $t-r=g(r)u$, and noting that $g^{\prime}(r/2)=o(1)$ as $r\to\infty$, it follows that
\begin{equation}\label{eq:improved_gVF_int7}\begin{split}
		\int_{r-Mg(r)}^{r+Mg(r)} \frac{1-\cos\lambda}{\lvert X-Y\rvert^{3}}\mathrm{d}t &\ge  \frac1{(1 +2g^{\prime}(r/2))^{3/2}}\int_{r-Mg(r)}^{r+Mg(r)} \frac{1-\cos\lambda}{\left((t-r)^{2} + 2g(r)^{2}(1-\cos\lambda)\right)^{3/2}}\mathrm{d}t\\
		&\ge \frac{1-o(1)}{g(r)^{2}} \int_{-M}^{M} \frac{1-\cos\lambda}{\big(u^{2} + 2(1-\cos\lambda)\big)^{3/2}} \mathrm{d}u 
		\qquad \text{as } r \to \infty.
\end{split}\end{equation}
Collecting \eqref{eq:improved_gVF_int4}, \eqref{eq:improved_gVF_int6}, and \eqref{eq:improved_gVF_int7}, and then taking average over $\lambda\in(0,2\pi)$ give
\begin{equation}\label{eq:improved_gVF_int8}\begin{split}
&\ \frac{1}{2\pi}\int_{0}^{2\pi}\int_{\frac r2}^{2r} g(t)V(t)\frac{g(r)(1-\cos\lambda)}{\lvert X-Y\rvert^{3}} \mathrm{d}t\mathrm{d}\lambda \\
\ge&\ \frac{1-o(1)}{2\pi}V(r)\int_{0}^{2\pi}\int_{-M}^{M} \frac{1-\cos\lambda}{\big(u^{2} + 2(1-\cos\lambda)\big)^{3/2}} \mathrm{d}u\mathrm{d}\lambda 
	\qquad \text{as } r \to \infty.
\end{split}\end{equation}

\textit{Step 3. Lower bound for the localized integral.} 
For $\lambda\in(0,2\pi)$ and $M>0$, define 
\begin{equation*}
	\theta_{M,\lambda} := \arctan \frac{M}{\sqrt{2(1-\cos\lambda)}}\in\left(0,\frac{\pi}2\right).
\end{equation*}
Applying the substitution $u= \sqrt{2(1-\cos\lambda)}\tan\theta$ and noting that $\theta_{M,\lambda}\to\pi/2$ as $M\to\infty$ yield \begin{equation*}
	\begin{split}
		\int_{-M}^{M} \frac{1-\cos\lambda}{\big(u^{2} + 2(1-\cos\lambda)\big)^{3/2}} \mathrm{d}u &= \int_{0}^{\theta_{M,\lambda}} \frac{1}{(\tan^{2}\theta + 1)^{3/2}} (\tan\theta)^{\prime}\mathrm{d}\theta\\
		&= \int_{0}^{\theta_{M,\lambda}} \cos\theta \mathrm{d}\theta\to 1 
		\qquad\text{as } M\to\infty.
	\end{split}
\end{equation*}
By the dominated convergence theorem, 
\begin{equation}\label{eq:int_M_u}
	\lim_{M \to \infty} \frac{1}{2\pi} \int_{0}^{2\pi}\int_{-M}^{M} \frac{1-\cos\lambda}{\big(u^{2} + 2(1-\cos\lambda)\big)^{3/2}} \mathrm{d}u \mathrm{d}\lambda= 1.
\end{equation}
This, together with \eqref{eq:improved_gVF_int8}, implies that for each $\varepsilon\in(0,1)$, there exist $M_{\varepsilon} > 0$ and $R_{\varepsilon}>4a$ such that		
\begin{equation*}
	\begin{split}
		\frac{1}{2\pi}\int_{0}^{2\pi}\int_{\frac r2}^{2r} g(t)V(t)\frac{g(r)(1-\cos\lambda)}{\lvert X-Y\rvert^{3}} \mathrm{d}t\mathrm{d}\lambda&\ge \frac{1-\varepsilon/2}{2\pi} V(r) \int_{0}^{2\pi}\int_{-M_{\varepsilon}}^{M_{\varepsilon}}\frac{1-\cos\lambda}{\big(u^{2} + 2(1-\cos\lambda)\big)^{3/2}} \mathrm{d}u \mathrm{d}\lambda\\
		&\ge (1-\varepsilon)V(r)\qquad \text{for } r > R_{\varepsilon},
	\end{split}
\end{equation*}
that is, \eqref{eq:improved_gVF_int3} holds. Hence the proof of the lemma is completed.
\end{proof}

\begin{lemma}\label{lem:Jg_refine_gdg}
Assume that $g$ satisfies \eqref{eq:condition_g_sublinear},  \eqref{eq:concave_fb}, and \eqref{eq:condition_g_unbounded}.  For each $\eta \in (0,1/4)$, there exists $R_{\eta} > 4a$ such that
\begin{equation}\label{eq:Jg_refine_gdg}
	\begin{split}
		& \frac{1}{2\pi} \int_{0}^{2\pi} \int_{\frac r2}^{2r} g(t)g^{\prime}(t)\left(\frac{1}{\lvert X-Y\rvert} - \frac{1}{\lvert X\rvert}\right) \mathrm{d}t\mathrm{d}\lambda\\
		\ge &\ 2(1-2\eta)g(r-\eta r)g^{\prime}(r+\eta r)\log\frac{\eta r}{g(r)} - \int_{\frac r2}^{\frac{r+\eta r}2} \frac{g(t)g^{\prime}(t)}{t} \mathrm{d}t
		\qquad \text{for } r > R_{\eta}.
	\end{split}
\end{equation}
\end{lemma}

\begin{proof}
Fix $\eta\in(0,1/4)$. Throughout the remainder of this proof, let $r>R_{\eta}$ for some $R_{\eta}>4a$ to be chosen later. Using $g^{\prime}\ge 0$ on $(a,\infty)$ (see \eqref{eq:condition_dg_sign}) and $\lvert X\rvert\ge t$ gives
\begin{equation*}
	g(t)g^{\prime}(t)\left(\frac{1}{\lvert X-Y\rvert} - \frac{1}{\lvert X\rvert}\right) \ge -\frac{g(t)g^{\prime}(t)}{\lvert X\rvert} \ge - \frac{g(t)g^{\prime}(t)}{t} \qquad\text{for } t>a.
\end{equation*}
Hence one has	
\begin{equation*}
	\frac{1}{2\pi} \int_{0}^{2\pi} \int_{\frac r2}^{\frac{r+\eta r}2} g(t)g^{\prime}(t)\left(\frac{1}{\lvert X-Y\rvert} - \frac{1}{\lvert X\rvert}\right) \mathrm{d}t \mathrm{d}\lambda \ge - \int_{\frac r2}^{\frac{r+\eta r}2} \frac{g(t)g^{\prime}(t)}{t} \mathrm{d}t.
\end{equation*}
Moreover, as in \eqref{eq:X-Y_X_compare}, there exists sufficiently large $R_{\eta}$ such that
\begin{equation*}
	\frac{1}{\lvert X-Y\rvert} - \frac{1}{\lvert X\rvert} > 0 \qquad\text{for } t>\frac{r+\eta r}2>\frac{R_{\eta}+\eta R_{\eta}}2.
\end{equation*}
Then it follows that 
\begin{equation*}
	\frac{1}{2\pi} \int_{0}^{2\pi} \int_{\frac{r+\eta r}2}^{r-\eta r} g(t)g^{\prime}(t) \left(\frac{1}{\lvert X-Y\rvert} - \frac{1}{\lvert X\rvert}\right) \mathrm{d}t\mathrm{d}\lambda \ge 0
\end{equation*}
and		
\begin{equation*}
	\frac{1}{2\pi} \int_{0}^{2\pi} \int_{r+\eta r}^{2r} g(t)g^{\prime}(t) \left(\frac{1}{\lvert X-Y\rvert} - \frac{1}{\lvert X\rvert}\right) \mathrm{d}t\mathrm{d}\lambda \ge 0.
\end{equation*}
Therefore, to prove \eqref{eq:Jg_refine_gdg}, it suffices to show 
\begin{equation}\label{eq:Jg_refine_gdg1}\begin{split}
&\ \frac{1}{2\pi} \int_{0}^{2\pi} \int_{r-\eta r}^{r+\eta r} g(t)g^{\prime}(t)\left(\frac{1}{\lvert X-Y\rvert} - \frac{1}{\lvert X\rvert}\right) \mathrm{d}t \mathrm{d}\lambda\\
\ge&\ 2(1-2\eta)g(r-\eta r)g^{\prime}(r+\eta r) \log\frac{\eta r}{g(r)}
\qquad \text{for } r > R_{\eta}.
\end{split}\end{equation}

To that end, let $t\in(r-\eta r, r+\eta r)$. Note that $t\ge (1-\eta)r$, $\lvert t-r\rvert\le \eta r$, and $g(t)+g(r)=o(r)$ as $r\to\infty$ by \eqref{eq:condition_g_sublinear}. We have
\begin{equation*}
	\begin{split}
		\frac{\lvert X-Y\rvert^{2}}{\lvert X\rvert^{2}} &= \frac{(t-r)^{2} + (g(t) - g(r))^{2} + 2g(t)g(r)(1-\cos\lambda)}{t^{2} + g(t)^{2}}\\
		&\le \frac{(t-r)^{2} + (g(t) + g(r))^{2}}{t^{2}} \le \frac{\eta^{2}+ o(1)}{(1-\eta)^{2}} \qquad \text{as } r \to \infty.
	\end{split}
\end{equation*}
In particular, there exists a sufficiently large $R_{\eta}$ such that for $r>R_\eta$, 
\begin{equation*}
\frac{\lvert X-Y\rvert}{\lvert X\rvert} \le \frac98\cdot\frac{\eta}{1-\eta} \le \frac{3\eta}{2}.
\end{equation*}
Hence one has
\begin{equation*}
\frac{1}{\lvert X-Y\rvert} - \frac{1}{\lvert X\rvert} \ge \left(1-\frac{3\eta}2\right)\frac{1}{\lvert X-Y\rvert}.
\end{equation*}
In view of \eqref{eq:X-Y_order} (which also holds for $t\in(r-\eta r,r+\eta r)\subset(r/2,2r)$), we can choose $R_{\eta}$ larger if necessary such that
\begin{equation*}
	\begin{split}
		\frac{1}{\lvert X-Y\rvert} - \frac{1}{\lvert X\rvert}\ge \frac{1-2\eta}{\sqrt{(t-r)^{2} + 2g(r)^{2}(1-\cos\lambda)}}
		\ge \frac{1-2\eta}{\sqrt{(t-r)^{2} + 4g(r)^{2}}} \qquad \text{for } r > R_{\eta}.
	\end{split}
\end{equation*} 
Recall that $g$ is non-decreasing and $g^{\prime}$ is non-increasing on $(a,\infty)$ (see \eqref{eq:condition_dg_sign} and \eqref{eq:concave_fb}). 
Consequently, 
\begin{equation}\label{eq:Jg_refine_gdg2}
	\begin{split}
		&\frac{1}{2\pi} \int_{0}^{2\pi} \int_{r-\eta r}^{r+\eta r} g(t)g^{\prime}(t)\left(\frac{1}{\lvert X-Y\rvert} - \frac{1}{\lvert X\rvert}\right) \mathrm{d}t \mathrm{d}\lambda\\
		\ge &\ (1-2\eta)g(r-\eta r)g^{\prime}(r+\eta r) \int_{r-\eta r}^{r+\eta r} \frac{1}{\sqrt{(t-r)^{2} + 4g(r)^{2}}} \mathrm{d}t\qquad \text{for } r > R_{\eta}.
	\end{split}
\end{equation}
Furthermore, direct computations give
		\begin{equation}\label{eq:Jg_refine_gdg3}
		\begin{split}
		\int_{r-\eta r}^{r+\eta r} \frac{1}{\sqrt{(t-r)^{2} + 4g(r)^{2}}} \mathrm{d}t 
        &= \log\left(t-r + \sqrt{(t-r)^{2} + 4g(r)^{2}}\right)\bigg|_{r-\eta r}^{r+\eta r}\\
        &= \log\frac{\eta r + \sqrt{(\eta r)^{2} + 4g(r)^{2}}}{-\eta r + \sqrt{(\eta r)^{2} + 4g(r)^{2}}}\\
		&= \log\frac{\left(\eta r + \sqrt{(\eta r)^{2} + 4g(r)^{2}}\right)^{2}}{4g(r)^{2}} \ge 2\log \frac{\eta r}{g(r)}.
		\end{split}
		\end{equation}
Substituting \eqref{eq:Jg_refine_gdg3} into \eqref{eq:Jg_refine_gdg2} yields \eqref{eq:Jg_refine_gdg1}. This finishes the proof of the lemma.
\end{proof}

The refined lower bound for $J_g(r)$ follows directly from Lemmas \ref{lem:Jg_refine_gvF} and \ref{lem:Jg_refine_gdg}.

\begin{corollary}\label{coro:Jg_refine_lower}
Assume that $g$ satisfies \eqref{eq:condition_g_sublinear},  \eqref{eq:concave_fb}, and \eqref{eq:condition_g_unbounded}.  Let $J_g$ be defined in \eqref{eq:Jg_def} and $V$ be as in \eqref{eq:Vt_def}.
For each $\varepsilon\in(0,1)$ and each $\eta \in (0,1/4)$, there exists  $R_{\varepsilon,\eta} > 4a$ such that for $r>R_{\varepsilon,\eta}$, one has
\begin{equation*}
	J_{g}(r) \ge (1-\varepsilon)V(r) - \varepsilon + 2(1-2\eta)g(r-\eta r)g^{\prime}(r+\eta r)\log\frac{\eta r}{g(r)} - \int_{\frac r2}^{\frac{r+\eta r}2} \frac{g(t)g^{\prime}(t)}{t} \mathrm{d}t.
\end{equation*}
\end{corollary}

\subsection{Proof of the refined growth estimate}\label{subsec:refined_growth}

This subsection is devoted to the proof of the refined growth estimate for the free boundary with logarithmic factor. As in Subsection \ref{subsec:initial_growth}, we first derive an integro-differential inequality, which improves Lemma \ref{lem:r_root_differential_ineq}.

\begin{lemma}\label{lem:refined_differential_ineq}
Assume that $g$ satisfies \eqref{eq:condition_g_sublinear},  \eqref{eq:concave_fb}, and \eqref{eq:condition_g_unbounded}. Let $V$ be as in \eqref{eq:Vt_def}. For each $\varepsilon\in(0,1)$, there exist constants $C>0$ independent of $\varepsilon$ and $r$, and $R_{\varepsilon} > 4a$ such that		
\begin{equation}\label{eq:refined_differential_ineq}
	(2-\varepsilon) \left(V(r) + g(r)g^{\prime}(r)\log\frac{r}{g(r)}\right) \le C + \int_{a}^{\frac r2} \frac{g(t)g^{\prime}(t)}{t}\mathrm{d}t
	\qquad \text{for } r > R_{\varepsilon}.
\end{equation}
\end{lemma}
\begin{proof}
The proof is divided into two steps.

{\it Step 1.}
Fix $\varepsilon\in(0,1)$, and let $\eta \in (0,1/4)$. It follows from the identity \eqref{eq:pe_main_NHJT}, Lemmas \ref{lemma_1_potential_expansion}, \ref{lem:Hg_refine_lower} and \ref{lem:Tg_refine_lower}, as well as Corollary \ref{coro:Jg_refine_lower} that there exists $R_{\varepsilon,\eta}>4a$ such that for $r > R_{\varepsilon,\eta}$,  one has
\begin{equation*}
	\begin{split}
		-V(r) + 2V(0)\ge&\ - C - \int_{a}^{\frac r2} \frac{g(t)g^{\prime}(t)}{t} \mathrm{d}t- 2\varepsilon+ (1-\varepsilon)V(r)\\
		&\ + 2(1-2\eta)g(r-\eta r)g^{\prime}(r+\eta r)\log\frac{\eta r}{g(r)} - \int_{\frac r2}^{\frac{r+\eta r}2} \frac{g(t)g^{\prime}(t)}{t} \mathrm{d}t.
	\end{split}
\end{equation*}
This implies
\begin{equation}\label{eq:refined_ineq_1}
	(2-\varepsilon)V(r) + 2(1-2\eta)g(r-\eta r)g^{\prime}(r+\eta r)\log\frac{\eta r}{g(r)} \le C + \int_{a}^{\frac{r+\eta r}2} \frac{g(t)g^{\prime}(t)}{t} \mathrm{d}t.
\end{equation}

{\it Step 2.}
We estimate the left-hand side of \eqref{eq:refined_ineq_1} from below. For sufficiently large $r$, the following estimates hold.  
First, since $V(t)/(t-a)$ is non-increasing as shown in Step 2 of Proposition \ref{prop:r_root_growth}, one has   \begin{equation}\label{eq:refined_ineq_3}
V(r)\ge (r-a)\frac{V(r+\eta r)}{r+\eta r-a} \ge (1-2\eta) V(r+\eta r).
\end{equation}
Second, using the monotonicity and concavity of $g$ (see \eqref{eq:condition_dg_sign} and  \eqref{eq:concave_fb}) yields 
\begin{equation}\label{eq:refined_ineq_6}
	0 \le g(r+\eta r) - g(r-\eta r) \le 2\eta rg^{\prime}(r-\eta r)
	\le 2\eta r\frac{g(r-\eta r)-g(a)}{r-\eta r-a}\le 4\eta g(r-\eta r).
\end{equation}
This gives		
\begin{equation}\label{eq:refined_ineq_5}
	g(r-\eta r) \ge \frac{g(r+\eta r)}{1 +4\eta}.
\end{equation}
Finally, since $\log((r+\eta r)/g(r+\eta r))\to\infty$ as $r\to\infty$ by \eqref{eq:condition_g_sublinear}, it holds that	
\begin{equation}\label{eq:refined_ineq_4}
	\log\frac{\eta r}{g(r)} \ge \log\frac{\eta r}{g(r+\eta r)} = \log\frac{r+\eta r}{g(r+\eta r)} + \log\frac{\eta}{1+\eta}
	\ge \left(1-\frac{\varepsilon}2\right) \log\frac{r+\eta r}{g(r+\eta r)}.
\end{equation}
Combining \eqref{eq:refined_ineq_3}, \eqref{eq:refined_ineq_5}, and \eqref{eq:refined_ineq_4}, with the choice $\eta=\varepsilon/16$ for fixed $\varepsilon\in(0,1)$, we obtain        
\begin{equation}\label{eq:refined_ineq_2}
	\begin{split}
		&(2-\varepsilon)V(r) + 2(1-2\eta)g(r-\eta r)g^{\prime}(r+\eta r)\log\frac{\eta r}{g(r)}\\
		\ge&\ (2-2\varepsilon)\left(V(r+\eta r) +g(r+\eta r)g^{\prime}(r+\eta r)\log\frac{r+\eta r}{g(r+\eta r)}\right) \qquad \text{for } r > R_{\varepsilon,\eta}
	\end{split}
\end{equation}
after enlarging $R_{\varepsilon,\eta}$ if necessary, where $2-2\varepsilon$ is a common lower bound for the coefficients arising from the estimates. 
Substituting \eqref{eq:refined_ineq_2} into \eqref{eq:refined_ineq_1}, and then replacing $r+\eta r$ by $r$ in the resulting inequality, give the desired estimate \eqref{eq:refined_differential_ineq}. 
\end{proof}

By establishing a lower bound for $V(r)$ involving the integral term in \eqref{eq:refined_differential_ineq}, we obtain the following integro-differential inequality, where the coefficients are essential for the refined growth estimate.

\begin{corollary}\label{coro:refined_ineq_sharp}
Assume that $g$ satisfies \eqref{eq:condition_g_sublinear},  \eqref{eq:concave_fb}, and \eqref{eq:condition_g_unbounded}.  For each $\varepsilon\in(0,1)$, there exist constants $C_{\varepsilon} > 0$ independent of $r$ and $R_{\varepsilon} > 4a$ such that
\begin{equation*}
	(4-\varepsilon)g(r)g^{\prime}(r)\log\frac{r}{g(r)} \le C_{\varepsilon} + \int_{a}^{\frac r2} \frac{g(t)g^{\prime}(t)}{t}\mathrm{d}t\qquad \text{for } r > R_{\varepsilon}.
\end{equation*}
\end{corollary}

\begin{proof}
It follows from $g^{\prime}(t)=o(1)$ as $t \to \infty$ (cf. \eqref{eq:condition_dg_to0}) and Taylor's expansion that, for each $\sigma\in(0,1)$, there exists $t_{\sigma} > 2a$ such that
\begin{equation*}
	V^{\prime}(t) = \sqrt{1 + \lvert g^{\prime}(t)\rvert^{2}} - 1 \ge \frac{1-\sigma}{2} \lvert g^{\prime}(t)\rvert^{2}\qquad \text{for } t > t_{\sigma}.
\end{equation*}
In the proof below, we let $r>2t_\sigma$. Then
\begin{equation}\label{eq:refined_ineq_sharp_V}
V(r) = V(t_{\sigma})+\int_{t_{\sigma}}^{r} V^{\prime}(t)\mathrm{d}t
\ge \frac{1-\sigma}{2}\int_{t_{\sigma}}^{r} \lvert g^{\prime}(t)\rvert^{2} \mathrm{d}t\ge \frac{1-\sigma}{2}\int_{t_{\sigma}}^{\frac r2} \lvert g^{\prime}(t)\rvert^{2} \mathrm{d}t.
\end{equation}

We claim that
\begin{equation}\label{eq:refined_ineq_sharp_g}
\int_{a}^{\frac r2} \frac{g(t)g^{\prime}(t)}{t} \mathrm{d}t \le C_{\sigma} + 2\int_{t_{\sigma}}^{\frac r2} \lvert g^{\prime}(t)\rvert^{2}\mathrm{d}t.
\end{equation}
Indeed, 		
\begin{equation}\label{eq:refined_ineq_sharp_g1}
	\int_{a}^{\frac r2} \frac{g(t)g^{\prime}(t)}{t} \mathrm{d}t = \int_{a}^{t_{\sigma}} \frac{g(t)g^{\prime}(t)}{t} \mathrm{d}t + g(t_{\sigma})\int_{t_{\sigma}}^{\frac r2} \frac{g^{\prime}(t)}{t} \mathrm{d}t + \int_{t_{\sigma}}^{\frac r2} \frac{g(t)-g(t_{\sigma})}{t} g^{\prime}(t) \mathrm{d}t.
\end{equation}
It follows from the nonnegativity of $g^{\prime}$ in \eqref{eq:condition_dg_sign}, the concavity of $g$ in \eqref{eq:concave_fb}, and Proposition \ref{prop:r_root_growth} that 
\begin{equation*}
	0\le \frac{g^{\prime}(t)}{t} \le \frac{g(t)-g(a)}{t(t-a)} \le Ct^{\alpha-2} \qquad\text{for } t>2a,
\end{equation*}
where $\alpha\in(1/2,1)$. Hence one has
\begin{equation}\label{eq:refined_ineq_sharp_g2}
\int_{a}^{t_{\sigma}} \frac{g(t)g^{\prime}(t)}{t} \mathrm{d}t + g(t_{\sigma})\int_{t_{\sigma}}^{\frac r2} \frac{g^{\prime}(t)}{t} \mathrm{d}t \le \int_{a}^{t_{\sigma}} \frac{g(t)g^{\prime}(t)}{t} \mathrm{d}t + Cg(t_{\sigma}) \int_{2a}^{\infty} t^{\alpha-2}\mathrm{d}t =: C_{\sigma}.
\end{equation}
On the other hand, using H\"{o}lder's inequality and Hardy's inequality (Lemma \ref{lemma:preliminaries_1}) gives		\begin{equation}\label{eq:refined_ineq_sharp_g3}
	\begin{split}
		\int_{t_{\sigma}}^{\frac r2} \frac{g(t)-g(t_{\sigma})}{t}g^{\prime}(t) \mathrm{d}t 
		&\le \left(\int_{t_{\sigma}}^{\frac r2} \left(\frac{g(t)-g(t_{\sigma})}{t-t_{\sigma}}\right)^2 \mathrm{d}t\right)^{\frac12} \left(\int_{t_{\sigma}}^{\frac r2} \lvert g^{\prime}(t)\rvert^{2} \mathrm{d}t\right)^{\frac12}\\
		&\le \left(4\int_{t_{\sigma}}^{\frac r2}\lvert g^{\prime}(t)\rvert^{2} \mathrm{d}t\right)^{\frac12} \left(\int_{t_{\sigma}}^{\frac r2} \lvert g^{\prime}(t)\rvert^{2} \mathrm{d}t\right)^{\frac12} = 2 \int_{t_{\sigma}}^{\frac r2} \lvert g^{\prime}(t)\rvert^{2} \mathrm{d}t.
	\end{split}
\end{equation}
Combining \eqref{eq:refined_ineq_sharp_g1}, \eqref{eq:refined_ineq_sharp_g2}, and \eqref{eq:refined_ineq_sharp_g3} gives
\eqref{eq:refined_ineq_sharp_g}.

In view of \eqref{eq:refined_ineq_sharp_V} and \eqref{eq:refined_ineq_sharp_g}, one has
\begin{equation*}
V(r)\ge \frac{1-\sigma}{4} \left(\int_{a}^{\frac r2}\frac{g(t)g^{\prime}(t)}{t}\mathrm{d}t - C_{\sigma}\right).
\end{equation*}
Fix $\varepsilon\in(0,1)$. The above lower bound for $V(r)$ together with Lemma \ref{lem:refined_differential_ineq} yields 
\begin{equation*}\begin{split}
&(2-\varepsilon) \left(\frac{1-\sigma}{4} \int_{a}^{\frac r2} \frac{g(t)g^{\prime}(t)}{t}\mathrm{d}t-\frac{1-\sigma}{4}C_{\sigma} + g(r)g^{\prime}(r)\log\frac{r}{g(r)}\right)\\ 
\le&\ C + \int_{a}^{\frac r2} \frac{g(t)g^{\prime}(t)}{t}\mathrm{d}t\qquad \text{for } r > \max\{R_{\varepsilon},\, 2t_{\sigma}\}. 
\end{split}\end{equation*}
Choose $\sigma = \varepsilon^2/2$. Rearranging the preceding inequality, dividing through by the coefficient of the integral term, and absorbing constants into $C_\varepsilon$, we obtain
\begin{equation*}
	(4-4\varepsilon)g(r)g^{\prime}(r)\log\frac{r}{g(r)} \le C_{\varepsilon} + \int_{a}^{\frac r2} \frac{g(t)g^{\prime}(t)}{t}\mathrm{d}t\qquad \text{for } r > R_{\varepsilon}
\end{equation*}
with a possibly larger $R_{\varepsilon}$. Hence the proof of the corollary is finished.
\end{proof}

We now derive the refined growth estimate for $g$.

\begin{proposition}\label{prop:g_refined_growth}
Assume that $g$ satisfies  \eqref{eq:condition_g_sublinear} and  \eqref{eq:concave_fb}. Then
\begin{equation*}
	\limsup_{r \to \infty} \frac{g(r)}{r^{1/2}(\log r)^{\delta}} < \infty\qquad \text{for every } \delta > -\frac{1}{4}.
\end{equation*}
\end{proposition}

\begin{proof}
Similar to Proposition \ref{prop:r_root_growth}, without loss of generality we assume that $g$ satisfies \eqref{eq:condition_g_unbounded}, and then argue by contradiction. Suppose that there exists a $\delta > -1/4$ such that	
\begin{equation}\label{eq:g_refined_growth_1}
\limsup_{r \to \infty} \frac{g(r)}{r^{1/2}(\log r)^{\delta}} = \infty.
\end{equation}
We divide the rest of the proof into four steps.

\textit{Step 1. Set up.} For each $i\in \mathbb{N}$, let $r_{i}\in [a,i]$ be such that	\begin{equation}\label{eq:refined_Ki_def}
\frac{g(r_{i})^{2} }{r_{i}(\log r_{i})^{2\delta}}	 = \max_{r\in [a,i]} \frac{g(r)^{2}}{r(\log r)^{2\delta}}=:K_i.
\end{equation}
Then $K_{i},\, r_i\to\infty$ as $i \to \infty$ by \eqref{eq:g_refined_growth_1}. 
Consider 
\begin{equation*}
	\Phi_{i}(r) := g(r)^{2} - K_{i}r(\log r)^{2\delta}.
\end{equation*}
It follows from \eqref{eq:refined_Ki_def} that $\Phi_{i} \le 0$ on $[a,i]$ and $\Phi_{i}(r_i)=0$. Hence 		\begin{equation}\label{eq:g_refined_growth_2}
	\Phi_{i}^{\prime}(r_{i}) = 2g(r_{i})g^{\prime}(r_{i}) - K_{i}(\log r_{i})^{2\delta} \left(1 + \frac{2\delta}{\log r_{i}}\right) \ge 0.
\end{equation}
where $r_i>a$ for large $i$.

\textit{Step 2. Upper bound for  $\int_{a}^{r_i/2} \frac{g(t)g^{\prime}(t)}{t} \mathrm{d}t$.} 
Integrating by parts gives
\begin{equation*}
	\int_{a}^{\frac{r_i}2} \frac{g(t)g^{\prime}(t)}{t} \mathrm{d}t 
	= \frac{g(r_{i}/2)^{2}}{r_{i}} - \frac{g(a)^{2}}{2a} + \int_{a}^{\frac{r_i}2} \frac{g(t)^{2}}{2t^{2}} \mathrm{d}t.
\end{equation*}
Consequently, using \eqref{eq:refined_Ki_def} we get that for large $i$, 
\begin{equation}\label{eq:g_refined_growth_3}
	\begin{split}
		\int_{a}^{\frac{r_i}2} \frac{g(t)g^{\prime}(t)}{t} \mathrm{d}t 
		&\le \frac{K_i}2\left(\log\frac{r_i}2\right)^{2\delta}+ \int_{a}^{\frac{r_i}2} \frac{K_{i}(\log t)^{2\delta}}{2t} \mathrm{d}t\\
		&\le \frac{K_i}2\left(\log\frac{r_i}2\right)^{2\delta}+\frac{K_i}{2(2\delta+1)}\left(\log\frac{r_i}2\right)^{2\delta+1},
	\end{split}
\end{equation}
where $(\log a)^{2\delta+1}>0$ for $a>1$ has been used in the second inequality.

\textit{Step 3. Lower bound for  $g(r_i)g^{\prime}(r_i)\log\frac{r_i}{g(r_i)}$.}
In view of Proposition \ref{prop:r_root_growth}, for each $\alpha \in (1/2,1)$, there exists an $r_{\alpha}>a$ such that 
$\log g(r)\le \alpha\log r$ for $r > r_{\alpha}$. Hence for large $i$,  one has
\begin{equation*}
	\log\frac{r_{i}}{g(r_{i})} \ge (1-\alpha) \log r_{i}.
\end{equation*}
This, together with \eqref{eq:g_refined_growth_2}, yields 
\begin{equation}\label{eq:g_refined_growth_4}
	g(r_{i})g^{\prime}(r_{i})\log\frac{r_{i}}{g(r_{i})} \ge (1-\alpha)\frac{K_{i}}{2}(\log r_{i})^{2\delta+1} \left(1 + \frac{2\delta}{\log r_{i}}\right).
\end{equation}

\textit{Step 4. Contradiction.} For each $\varepsilon \in (0,1)$ and each $\alpha \in (1/2,1)$, applying \eqref{eq:g_refined_growth_4}, Corollary \ref{coro:refined_ineq_sharp}, and \eqref{eq:g_refined_growth_3} gives that for sufficiently large $i$,
\begin{equation*}
	\begin{split}
		(4-\varepsilon)(1-\alpha)\frac{K_{i}}{2} (\log r_{i})^{2\delta+1}\left(1 + \frac{2\delta}{\log r_{i}}\right) 
		&\le (4-\varepsilon)g(r_{i})g^{\prime}(r_{i})\log\frac{r_{i}}{g(r_{i})}\\
		&\le C_{\varepsilon} + \int_{a}^{\frac {r_i}2} \frac{g(t)g^{\prime}(t)}{t} \mathrm{d}t\\
		&\le C_{\varepsilon} + \frac{K_{i}}{2} \left(\log\frac{r_{i}}2\right)^{2\delta+1}\left(\frac{1}{\log(r_{i}/2)} + \frac{1}{2\delta + 1}\right).
	\end{split}
\end{equation*}
Dividing both sides by $K_i(\log r_i)^{2\delta+1}/2$ and taking the limit $i\to \infty$ yield that for all $\varepsilon\in(0,1)$ and $\alpha\in(1/2,1)$, 
\begin{equation*}
	(4-\varepsilon)(1-\alpha) \le \frac{1}{2\delta + 1}.
\end{equation*}
Letting $\varepsilon \to 0$ and $\alpha \to 1/2$ yields $2 \le 1/(2\delta + 1)$. This means $\delta\le -1/4$ and leads to a contradiction. Hence the proof of the proposition is completed..
\end{proof}

\section{Lower bound for the growth of the free boundary}\label{sec:g_lower_estimate}

This section establishes the lower bound for the growth of the free boundary. As expected (as bounded solutions of the Bernoulli problem without fixed boundary do exist) the estimate from below is considerably more subtle. This will be reflected in the
analysis of the perturbed integral equality derived in the
first subsection.  

\subsection{Derivation of the perturbed integral equality}\label{subsec:sharp_estimate}

In this subsection, we establish expansion formulas for $H_{g}(r)$, $J_g(r)$ and $T_{g}(r)$ as $r\to\infty$,
and we derive the perturbed integral equality Proposition \ref{prop:lower_bound_log1} mentioned in the introduction.

It follows 
from Proposition \ref{prop:g_refined_growth} that 
\begin{equation}\label{eq:sharp_estimate_g}
g(t)\le Ct^{\frac12}(\log t)^{-\frac18}\le Ct^{\frac12} \qquad \text{for } t > a.
\end{equation}
Consequently, using the concavity of $g$ in \eqref{eq:concave_fb} yields
\begin{equation}\label{eq:sharp_estimate_dg}
0\le g^{\prime}(t) \le \frac{g(t)-g(a)}{t-a}\le \frac{2g(t)}{t} \le Ct^{-\frac12} \qquad \text{for } t > 2a.
\end{equation}
Furthermore, the first inequality in \eqref{eq:Tg_lower_Vt} together with \eqref{eq:sharp_estimate_dg} gives   
\begin{equation}\label{eq:sharp_estimate_V}
0\le V(t)\le V(2a)+\frac12\int_{2a}^{t}\lvert g^{\prime}(\tau)\rvert^2\mathrm{d}\tau \le C(1+\log t) 
\qquad \text{for } t >2 a.
\end{equation}

With the above estimates in hand, let us expand $H_g(r)$. 

\begin{lemma}\label{lem:Hg_sharp_estimate}
Assume that $g$ satisfies \eqref{eq:condition_g_sublinear} and \eqref{eq:concave_fb}. Let $H_g$ be defined in \eqref{eq:Hg_def}. 
There exists a fixed constant $C_{H}$ such that
\begin{equation}\label{est_lemma6.1}
H_{g}(r) = C_{H} + o(1) - \frac{1}{2\pi} \int_{0}^{2\pi} \int_{a}^{\frac r2} \frac{g(t)g^{\prime}(t)}{\lvert X\rvert} \mathrm{d}t\mathrm{d}\lambda \qquad\text{as } r \to \infty.
\end{equation}
\end{lemma}

\begin{proof}
It follows from the definitions of $H_g$ in \eqref{eq:Hg_def} and $F_g$ in \eqref{eq:Fg_def} that 
\begin{equation}\label{eq:Hg_rewrite}\begin{split}
H_{g}(r) =&\ \frac{1}{2\pi}\int_{0}^{2\pi} \int_{a}^{\frac r2} g(t)V(t)\frac{tg^{\prime}(t)-g(t)}{\lvert X\rvert^{3}}\mathrm{d}t\mathrm{d}\lambda\\
&\ +\frac{1}{2\pi} \int_{0}^{2\pi}\int_{a}^{\frac r2} g(t)V(t)\frac{g(t)+(r-t)g^{\prime}(t)-g(r)\cos\lambda}{\lvert X-Y\rvert^{3}}\mathrm{d}t\mathrm{d}\lambda\\
&\ +\frac{1}{2\pi} \int_{0}^{2\pi}\int_{a}^{\frac r2} \frac{g(t)g^{\prime}(t)}{\lvert X-Y\rvert} \mathrm{d}t\mathrm{d}\lambda
-\frac{1}{2\pi} \int_{0}^{2\pi} \int_{a}^{\frac r2} \frac{g(t)g^{\prime}(t)}{\lvert X\rvert} \mathrm{d}t\mathrm{d}\lambda,
\end{split}\end{equation}
where $V(t)$ is given by \eqref{eq:Vt_def}.

The estimates \eqref{eq:sharp_estimate_g}-\eqref{eq:sharp_estimate_V} together with $\lvert X\rvert\ge t$ lead to
\begin{equation}\label{eq:Hg_gvX}
\left| g(t)V(t)\frac{tg^{\prime}(t)-g(t)}{\lvert X\rvert^{3}}\right| \le  \frac{C(1+\log t)}{t^{2}}
\qquad \text{for } t > 2a. 
\end{equation}
Note that $(1+\log t)/t^2\in L^{1}(2a,\infty)$. Consequently, it holds that 
\begin{equation}\label{eq:Hg_sharp_estimate_1}
\begin{split}
\frac{1}{2\pi}\int_{0}^{2\pi}\int_{a}^{\frac r2} g(t)V(t)\frac{tg^{\prime}(t)-g(t)}{\lvert X\rvert^{3}}\mathrm{d}t\mathrm{d}\lambda= C_H + o(1) \qquad\text{as } r \to \infty,
\end{split}\end{equation}
where
\begin{equation*}
C_{H} :=\frac{1}{2\pi}\int_{0}^{2\pi}\int_{a}^{\infty} g(t)V(t)\frac{tg^{\prime}(t)-g(t)}{\lvert X\rvert^{3}}\mathrm{d}t\mathrm{d}\lambda.
\end{equation*} 
To estimate the second term of $H_g(r)$ in \eqref{eq:Hg_rewrite}, using \eqref{eq:sharp_estimate_g} and \eqref{eq:sharp_estimate_V} along with $\lvert X-Y\rvert\ge\lvert t-r\rvert\ge r/2$ for $t\le r/2$, we get	\begin{equation}\label{eq:Hg_sharp_estimate_2}
\begin{split}
&\ \left| \frac{1}{2\pi} \int_{0}^{2\pi} \int_{a}^{\frac r2} g(t)V(t)\frac{g(t)+(r-t)g^{\prime}(t)-g(r)\cos\lambda}{\lvert X-Y\rvert^{3}}\mathrm{d}t\mathrm{d}\lambda\right| \\
\le & \ C \int_{a}^{\frac r2} t^{1/2}(1+\log t) \frac{t^{1/2}+rg^{\prime}(t) + r^{1/2}}{r^{3}}\mathrm{d}t
\le \frac{C(1+\log r)}{r}= o(1) 
\qquad\text{as } r \to \infty.
\end{split}
\end{equation}
In addition, by the first inequality in  \eqref{eq:sharp_estimate_g}, one has
	\begin{equation}\label{eq:Hg_sharp_estimate_3}
		\begin{split}
		0 \le \frac{1}{2\pi} \int_{0}^{2\pi} \int_{a}^{\frac r2} \frac{g(t)g^{\prime}(t)}{\lvert X-Y\rvert}\mathrm{d}t\mathrm{d}\lambda 
        &\le C \int_{a}^{\frac r2} \frac{g(t)g^{\prime}(t)}{r}\mathrm{d}t
		= \frac{C}{2r} \left(g\left(\frac r2\right)^{2}-g(a)^{2}\right)\\
        &\le C(\log r)^{-\frac 14}
        = o(1) \qquad\text{as } r \to \infty.
		\end{split}
	\end{equation}
Combining \eqref{eq:Hg_rewrite} with \eqref{eq:Hg_sharp_estimate_1}-\eqref{eq:Hg_sharp_estimate_3} yields the estimate \eqref{est_lemma6.1}.
\end{proof}

Next, we prove the improved estimate for $T_g(r)$.

\begin{lemma}\label{lem:Tg_sharp_estimate}
Assume that $g$ satisfies \eqref{eq:condition_g_sublinear} and \eqref{eq:concave_fb}. Let $T_g$ be defined in \eqref{eq:Tg_def}. Then one has
\begin{equation*}
T_{g}(r) = o(1) \qquad\text{as } r \to \infty.
\end{equation*}
\end{lemma}

\begin{proof}
Let $r>a$. Note that $\lvert X\rvert\ge t$, and $\lvert X-Y\rvert\ge t-r>t/2$ for $t > 2r$. Using the definition of $F_g$ (see \eqref{eq:Fg_def}) together with the estimates for $g$, $g^{\prime}$, and $V$ in \eqref{eq:sharp_estimate_g}-\eqref{eq:sharp_estimate_V} yields
\begin{equation}\label{eq:Tg_sharp_estimate_1}
\begin{split}
\left|g(t)V(t)F_{g}(t,\lambda;r)\right|
&=\left|g(t)V(t) \left(\frac{g(t)+(r-t)g^{\prime}(t) - g(r)\cos\lambda}{\lvert X-Y\rvert^{3}} + \frac{tg^{\prime}(t) - g(t)}{\lvert X\rvert^{3}}\right)\right|\\
&\le Ct^{\frac 12}(1+\log t) \left(\frac{2t^{1/2}+r^{1/2}}{t^{3}} +\frac{2t^{1/2}}{t^{3}}\right)\\
&\le \frac{C(1+\log t)}{t^{2}}  \qquad \text{for } t > 2r.
\end{split}
\end{equation}
Moreover, by \eqref{eq:sharp_estimate_g}, one has
\begin{equation*}
\left|\frac{1}{\lvert X-Y\rvert} - \frac{1}{\lvert X\rvert}\right| \le \frac{\lvert Y\rvert}{\lvert X-Y\rvert \lvert X\rvert} \le \frac{2(r^{2}+g(r)^{2})^{1/2}}{t^2}\le \frac{Cr}{t^{2}}\qquad \text{for } t > 2r.
\end{equation*}
Also, applying \eqref{eq:sharp_estimate_dg} and the first inequality in \eqref{eq:sharp_estimate_g} gives  
\begin{equation}\label{eq:sharp_estimate_gdg}
    g(t)g^{\prime}(t)\le C(\log t)^{-\frac18}
    \qquad\text{for } t>2a.
\end{equation}
Hence, 
\begin{equation}\label{eq:Tg_sharp_estimate_2}
\begin{split}
\left|g(t)g^{\prime}(t)\left(\frac{1}{\lvert X-Y\rvert} - \frac{1}{\lvert X\rvert}\right)\right| 
\le C(\log t)^{-\frac18}\frac{r}{t^{2}} 
\qquad \text{for } t > 2r.
\end{split}
\end{equation}
Then it follows from the definition of $T_g$ in \eqref{eq:Tg_def}, \eqref{eq:Tg_sharp_estimate_1} and \eqref{eq:Tg_sharp_estimate_2} that
\begin{equation}\label{eq:Tg_sharp_estimate_3}
	\begin{split}
		\lvert T_{g}(r)\rvert &= \left|\frac{1}{2\pi} \int_{0}^{2\pi}\int_{2r}^{\infty} g(t)V(t)F_{g}(t,\lambda;r) + g(t)g^{\prime}(t)\left(\frac{1}{\lvert X-Y\rvert} - \frac{1}{\lvert X\rvert}\right) \mathrm{d}t\mathrm{d}\lambda\right|\\
		&\le C\int_{2r}^{\infty} \frac{1 + \log t}{t^{2}} + (\log t)^{-\frac18}\frac{r}{t^{2}} \mathrm{d}t \to 0 \qquad \text{as } r \to \infty.
	\end{split}
\end{equation}
This finishes the proof of the lemma.
\end{proof}

We now establish two lemmas that handle certain terms arising from the expansion of $J_{g}(r)$. In Corollary \ref{coro:Jg_sharp_estimate} we will conclude the expansion formula of $J_{g}(r)$. 
 
\begin{lemma}\label{lem:Jg_sharp_estimate_1}
Assume that $g$ satisfies \eqref{eq:condition_g_sublinear} and \eqref{eq:concave_fb}. Let $V$ be as in \eqref{eq:Vt_def}. Then
\begin{equation}\label{eq:Jg_sharp_estimate_1}
\frac{1}{2\pi} \int_{0}^{2\pi} \int_{\frac r2}^{2r} g(t)V(t)\frac{g(r)(1-\cos\lambda)}{\lvert X-Y\rvert^{3}}\mathrm{d}t\mathrm{d}\lambda = V(r) + o(1) \qquad\text{as } r \to \infty.
\end{equation}
\end{lemma}

\begin{proof}
Set $M(r):= (r/g(r))^{1/3}$. It follows from $g(r)\le C r^{1/2}$ (see \eqref{eq:sharp_estimate_g}) that $M(r)\to\infty$ as $r\to\infty$, and
\begin{equation*}
M(r)g(r) = r^{\frac13}g(r)^{\frac23} \le Cr^{\frac23}< \frac{r}{2}\qquad \text{for large $r$.}
\end{equation*}
We now divide the integral in \eqref{eq:Jg_sharp_estimate_1} into three integrals over the intervals $[r/2,r-M(r)g(r)]$, $[r-M(r)g(r),r+M(r)g(r)]$, and $[r+M(r)g(r),2r]$ and expand those  integrals one by one.

\textit{Step 1. Estimate on the far regions $[r/2,r-M(r)g(r)]$ and $[r+M(r)g(r),2r]$.} Since $g(r)\le Cr^{1/2}$ and $V(r) \le C(1+\log r)$ (see \eqref{eq:sharp_estimate_g} and \eqref{eq:sharp_estimate_V}), we have
\begin{equation*}
\frac{V(r)}{M(r)^{2}}=V(r)\left(\frac{g(r)}r\right)^{\frac23}\le C (1+\log r)r^{-\frac13}= o(1) \qquad\text{as } r \to \infty.
\end{equation*}
Hence, using the monotonicity of $g$ and $V$  (see \eqref{eq:gV_increasing}) together with $\lvert X-Y\rvert\ge |t-r|$ gives
\begin{equation}\label{eq:Jg_sharp_estimate_11}
\begin{split}
0 &\le \frac{1}{2\pi} \int_{0}^{2\pi} \int_{\frac r2}^{r-M(r)g(r)} g(t)V(t)\frac{g(r)(1-\cos\lambda)}{\lvert X-Y\rvert^{3}}\mathrm{d}t\mathrm{d}\lambda\\ 
&\le \int_{\frac r2}^{r-M(r)g(r)} g(r)V(r)\frac{2g(r)}{\lvert t-r\rvert^{3}}\mathrm{d}t
\le 2g(r)^{2}V(r) \int_{M(r)g(r)}^{\infty} \frac{1}{\tau^{3}}\mathrm{d}\tau\\
&= \frac{V(r)}{M(r)^{2}}= o(1) 
\qquad\text{as } r \to \infty.
\end{split}
\end{equation}
Similarly, since $g(2r)\le 3g(r)$ for large $r$ (see \eqref{eq:g2r_r}) and $V(2r)\le C(1+\log(2r))$, one has
\begin{equation}\label{eq:Jg_sharp_estimate_12}
0 \le \frac{1}{2\pi} \int_{0}^{2\pi} \int_{r+M(r)g(r)}^{2r} g(t)V(t)\frac{g(r)(1-\cos\lambda)}{\lvert X-Y\rvert^{3}}\mathrm{d}t\mathrm{d}\lambda \le C\frac{V(2r)}{M(r)^{2}}=o(1) \qquad\text{as } r \to \infty.
\end{equation}
	
\textit{Step 2. Estimate on the middle region $[r-M(r)g(r),r+M(r)g(r)]$.} 
Let $t\in [r-M(r)g(r),r+M(r)g(r)]\subset[r/2,2r]$ with sufficiently large $r$. By the concavity of $g$ in \eqref{eq:concave_fb} and $g(r)\le Cr^{1/2}$ in \eqref{eq:sharp_estimate_g}, we have
\begin{equation}\label{eq:Jg_sharp_estimate_gtr}
\begin{split}
\lvert g(t)-g(r)\rvert 
\le 2M(r)g(r)g^{\prime}\left(\frac r2\right) 
\le CM(r)g(r)\frac{g(r)}{r} = C g(r)\left(\frac{g(r)}{r}\right)^{\frac23}\le Cr^{-\frac13}g(r).
\end{split}
\end{equation}
Moreover, it follows from the expression of $V$ in \eqref{eq:Vt_def} that
\begin{equation}\label{eq:Jg_sharp_estimate_Vtr}
\begin{split}
\lvert V(t)-V(r)\rvert&=\left\lvert\int_{r}^{t}\left(\sqrt{1 + \lvert g^{\prime}(\tau)\rvert^{2}} - 1\right)\mathrm{d}\tau\right\rvert 
\le \frac12\int_{r-M(r)g(r)}^{r+M(r)g(r)}\lvert g^{\prime}(\tau)\rvert^{2}\mathrm{d}\tau\\
&\le CM(r)g(r)\frac{g(r)^{2}}{r^{2}} 
= Cg(r)\left(\frac{g(r)}{r}\right)^{\frac53}\le Cr^{-\frac13}.
\end{split}
\end{equation}
Therefore, it holds that
\begin{equation}\label{eq:Jg_sharp_estimate_13}
g(t)V(t) = \left(1+O(r^{-1/3})\right)\left(V(r) + O(r^{-1/3})\right)g(r) \qquad\text{as } r\to\infty.
\end{equation}
On the other hand, by $g^{\prime}(r/2)\le Cr^{-1/2}$ (see \eqref{eq:sharp_estimate_dg}) and \eqref{eq:Jg_sharp_estimate_gtr}, one has 
	\begin{equation*}
		\begin{split}
		\lvert X-Y\rvert^{2} &= (t-r)^{2} + (g(t)-g(r))^{2} + 2g(t)g(r)(1-\cos\lambda)\\
		&= (t-r)^{2} + O(r^{-1})(t-r)^{2} + 2g(r)^{2}(1-\cos\lambda) \left(1 + O(r^{-1/3})\right)\\
        &=\left(1 + O(r^{-1/3})\right)\left((t-r)^{2} + 2g(r)^{2}(1-\cos\lambda)\right) 
        \qquad\text{as } r\to\infty.
		\end{split}
	\end{equation*}
	This implies 
	\begin{equation}\label{eq:Jg_sharp_estimate_X-Y}
		\frac{1}{\lvert X-Y\rvert^{3}} =\frac{1 + O(r^{-1/3}) }{\big((t-r)^{2}+2g(r)^{2}(1-\cos\lambda)\big)^{3/2}} 
        \qquad\text{as } r\to\infty.
	\end{equation}
Note that the estimates in \eqref{eq:Jg_sharp_estimate_13} and \eqref{eq:Jg_sharp_estimate_X-Y} hold uniformly for $t\in[r-M(r)g(r),r+M(r)g(r)]$ and $\lambda\in[0,2\pi)$ since the implicit constants in the $O(r^{-1/3})$-terms are independent of $t$ and $\lambda$. Therefore, applying \eqref{eq:Jg_sharp_estimate_13} and \eqref{eq:Jg_sharp_estimate_X-Y}, with the substitution $u=(t-r)/g(r)$, we obtain \begin{equation}\label{equation_7_sharp_estimates_for_H_g,J_g_and_T_g}
		\begin{split}
			& \ \frac{1}{2\pi} \int_{0}^{2\pi} \int_{r-M(r)g(r)}^{r+M(r)g(r)} g(t)V(t)\frac{g(r)(1-\cos\lambda)}{\lvert X-Y\rvert^{3}}\mathrm{d}t\mathrm{d}\lambda\\
			= & \  \frac{\left(1+O(r^{-1/3})\right)\left(V(r) + O(r^{-1/3})\right)}{2\pi}\int_{0}^{2\pi}\int_{r-M(r)g(r)}^{r+M(r)g(r)} \frac{g(r)^{2}(1-\cos\lambda)}{\big((t-r)^{2}+2g(r)^{2}(1-\cos\lambda)\big)^{3/2}}\mathrm{d}t\mathrm{d}\lambda\\
			= & \  \frac{\left(1+O(r^{-1/3})\right)\left(V(r) + O(r^{-1/3})\right)}{2\pi}\int_{0}^{2\pi}\int_{-M(r)}^{M(r)} \frac{1-\cos\lambda}{\big(u^{2}+2(1-\cos\lambda)\big)^{3/2}}\mathrm{d}u\mathrm{d}\lambda.
		\end{split}
	\end{equation}
       
        Recall that (see \eqref{eq:int_M_u})
	\begin{equation*}		\frac1{2\pi}\int_{0}^{2\pi}\int_{-\infty}^{\infty} \frac{1-\cos\lambda}{\big(u^{2}+2(1-\cos\lambda)\big)^{3/2}}\mathrm{d}u\mathrm{d}\lambda = 1.
	\end{equation*}
	Moreover, 
        \begin{equation*}
		\begin{split}
		0 &\le \frac1{2\pi}\int_{0}^{2\pi}\int_{\lvert u\rvert > M(r)} \frac{1-\cos\lambda}{\big(u^{2}+2(1-\cos\lambda)\big)^{3/2}}\mathrm{d}u\mathrm{d}\lambda\\
        &\le \int_{\lvert u\rvert > M(r)} \frac{2}{\lvert u\rvert^{3}}\mathrm{d}u = \frac{2}{M(r)^{2}}=2\left(\frac{g(r)}r\right)^{\frac23}\le Cr^{-\frac13}.
		\end{split}
	\end{equation*}
	Then it follows that 
	\begin{equation*}
		\frac{1}{2\pi}\int_{0}^{2\pi}\int_{-M(r)}^{M(r)} \frac{1-\cos\lambda}{\big(u^{2}+2(1-\cos\lambda)\big)^{3/2}}\mathrm{d}u\mathrm{d}\lambda = 1 + O(r^{-1/3}) \qquad\text{as } r\to\infty.
	\end{equation*}
	Substituting the above estimate into  \eqref{equation_7_sharp_estimates_for_H_g,J_g_and_T_g} and using $V(r) \le C(1+\log r)$ from  \eqref{eq:sharp_estimate_V}, we get
	\begin{equation}\label{eq:Jg_sharp_estimate_14}
		\begin{split}
		\frac{1}{2\pi}\int_{0}^{2\pi}\int_{r-M(r)g(r)}^{r+M(r)g(r)} g(t)V(t)\frac{g(r)(1-\cos\lambda)}{\lvert X-Y\rvert^{3}}\mathrm{d}t\mathrm{d}\lambda 
        &=\left(1 + O(r^{-1/3})\right) \left(V(r) + O(r^{-1/3})\right)\\
		&= V(r) + o(1).
		\end{split}
	\end{equation}
	Combining \eqref{eq:Jg_sharp_estimate_11},  \eqref{eq:Jg_sharp_estimate_12} and \eqref{eq:Jg_sharp_estimate_14} completes the proof of  this lemma.
\end{proof}

\begin{lemma}\label{lem:Jg_sharp_estimate_D}
Assume that $g$ satisfies \eqref{eq:condition_g_sublinear} and \eqref{eq:concave_fb}. Let $V$ be as in \eqref{eq:Vt_def}. Then
\begin{equation}\label{eq:Jg_sharp_estimate_D}
\frac{1}{2\pi} \int_{0}^{2\pi} \int_{\frac r2}^{2r} g(t)V(t)\frac{D(t;r)}{\lvert X-Y\rvert^{3}}\mathrm{d}t\mathrm{d}\lambda = o(1) \qquad\text{as } r \to \infty,
\end{equation}
where
\begin{equation}\label{eq:D_def}
D(t;r):= g(t)-g(r) +(r-t)g^{\prime}(t).
\end{equation}
\end{lemma}

\begin{proof}
Let $r>4a$ be sufficiently large such that $g(r)<r/4$. Note that by $g^{\prime\prime}\le0$ on $(a,\infty)$ (see \eqref{eq:concave_fb}), 
\begin{equation*}
D(t;r) = \int_r^t\left(g^{\prime}(\tau)-g^{\prime}(t)\right) \mathrm{d}\tau \ge 0.
\end{equation*}
As in Lemma \ref{lem:Jg_sharp_estimate_1}, we divide the integral in \eqref{eq:Jg_sharp_estimate_D} into three integrals over the three intervals $[r/2,r-g(r)]$, $[r-g(r),r+g(r)]$, and $[r+g(r),2r]$ and expand each of those three integrals one by one. 

\textit{Step 1. Estimate on the outer intervals $[r/2,r-g(r)]$ and $[r+g(r),2r]$.}  
Since $g^{\prime}(t)\le g^{\prime}(r/2)\le Cg(r)/r$ for $t\ge r/2$ (see \eqref{eq:sharp_estimate_dg}), one has   
\begin{equation}\label{eq:Jg_sharp_estimate_D4}
0\le D(t;r) \le C\lvert t-r\rvert \frac{g(r)}{r} \qquad\text{for } t\ge \frac r2.
\end{equation}
Moreover, it follows from $g(r) \le Cr^{1/2}$ and $V(r) \le C(1+\log r)$ (see \eqref{eq:sharp_estimate_g} and \eqref{eq:sharp_estimate_V}) that 
	\begin{equation*}
		\frac{g(r)}{r}V(r) \le Cr^{-\frac 12}(1+\log r) \to 0 \qquad\text{as } r \to \infty.
	\end{equation*}	
Thus using the monotonicity of $g$ and $V$ (see \eqref{eq:gV_increasing}) as well as $\lvert X-Y\rvert\ge\lvert t-r\rvert$ yields
	\begin{equation}\label{eq:Jg_sharp_estimate_D1}
		\begin{split}
		0 &\le \frac{1}{2\pi} \int_{0}^{2\pi} \int_{\frac r2}^{r-g(r)} g(t)V(t) \frac{D(t;r)}{\lvert X-Y\rvert^{3}}\mathrm{d}t\mathrm{d}\lambda \\
        &\le C  \frac{g(r)^{2}}{r}V(r)\int_{\frac r2}^{r-g(r)} \frac{1}{\lvert t-r\rvert^2}\mathrm{d}t
		\le C \frac{g(r)^{2}}{r}V(r) \int_{g(r)}^{\infty} \frac{1}{\tau^{2}}\mathrm{d}\tau\\
        &= C \frac{g(r)}{r}V(r)=o(1) 
        \qquad\text{as } r \to \infty.
		\end{split}
	\end{equation}
	Similarly, with the help of $g(2r)\le 3g(r)$ in  \eqref{eq:g2r_r} and $V(2r) \le C(1+\log (2r))$, we also have
	\begin{equation}\label{eq:Jg_sharp_estimate_D2}
		0 \le \frac{1}{2\pi} \int_{0}^{2\pi} \int_{r+g(r)}^{2r} g(t)V(t) \frac{D(t;r)}{\lvert X-Y\rvert^{3}}\mathrm{d}t\mathrm{d}\lambda= o(1) \qquad\text{as } r \to \infty.
	\end{equation}
	
\textit{Step 2. Estimate on the middle interval $[r-g(r), r+g(r)]$.}  
For $t\in [r-g(r), r+g(r)]$ with sufficiently large $r$ and $\lambda\in(-\pi,\pi)$, it follows from $g(t)\ge g(r/2)\ge g(r)/3$ and the inequality $1-\cos\lambda \ge 2\lambda^{2}/\pi^2$ that
	\begin{equation*}
		\lvert X-Y\rvert^{2} = (t-r)^{2} + (g(t)-g(r))^{2} + 2g(t)g(r)(1-\cos\lambda)\ge C \left((t-r)^{2} + g(r)^{2}\lambda^{2}\right).
	\end{equation*}
    Thus, using the change of variables $u=g(r)\lambda/\lvert t-r\rvert$, one has
        \begin{equation*}\begin{split}
		0 \le \int_{0}^{2\pi} \frac{1}{\lvert X-Y\rvert^{3}}\mathrm{d}\lambda 
        &\le C\int_{-\pi}^{\pi} \frac{1}{\big((t-r)^{2} + g(r)^{2}\lambda^{2}\big)^{3/2}}\mathrm{d}\lambda\\
        &\le \frac{C}{g(r)(t-r)^2}\int_{-\infty}^{\infty} \frac{1}{(1+u^{2})^{3/2}}\mathrm{d}u = \frac{C}{g(r)(t-r)^{2}} \qquad\text{for }t\neq r. 
	\end{split}\end{equation*}	
        This implies that 
	\begin{equation}\label{eq:Jg_sharp_estimate_D5}
		0 \le \frac{1}{2\pi} \int_{0}^{2\pi} \int_{r-g(r)}^{r+g(r)} g(t)V(t)\frac{D(t;r)}{\lvert X-Y\rvert^{3}}\mathrm{d}t\mathrm{d}\lambda \le C \int_{r-g(r)}^{r+g(r)} g(t)V(t) \frac{D(t;r)}{g(r)(t-r)^{2}}\mathrm{d}t.
	\end{equation}
        Furthermore, as in \eqref{eq:Jg_sharp_estimate_Vtr}, one has 
	\begin{equation*}
		\lvert V(t)-V(r)\rvert \le \frac12\int_{r-g(r)}^{r+g(r)}\lvert g^{\prime}(\tau)\rvert^2\mathrm d\tau \le  C \frac{g(r)^{3}}{r^{2}} \qquad \text{for } t\in[r-g(r),r+g(r)]. 
	\end{equation*}
	This combined with \eqref{eq:Jg_sharp_estimate_D5} and $g(r+g(r))\le g(2r)\le 3g(r)$ gives
	\begin{equation}\label{eq:Jg_sharp_estimate_D3}
		0 \le \frac{1}{2\pi} \int_{0}^{2\pi} \int_{r-g(r)}^{r+g(r)} g(t)V(t)\frac{D(t;r)}{\lvert X-Y\rvert^{3}}\mathrm{d}t\mathrm{d}\lambda \le C\left(V(r) + C\frac{g(r)^{3}}{r^{2}}\right) \int_{r-g(r)}^{r+g(r)} \frac{D(t;r)}{(t-r)^{2}}\mathrm{d}t.
	\end{equation}

        To estimate the integral on the right-hand side of \eqref{eq:Jg_sharp_estimate_D3}, we note that by the definition of $D(t;r)$ in \eqref{eq:D_def} and the estimate \eqref{eq:Jg_sharp_estimate_D4}, 
        \begin{equation*}
            \lim_{t \to r} \frac{D(t;r)}{t-r}=0 \quad\text{and}\quad 
            0\le \frac{D(r+g(r);r)}{g(r)}\le C\frac{g(r)}{r}.
        \end{equation*}
        Applying integration by parts and $g^{\prime}(r+g(r))\le g^{\prime}(r)\le Cg(r)/r$ yields 
	\begin{equation*}
		\begin{split}
		0 \le \int_{r}^{r+g(r)} \frac{D(t;r)}{(t-r)^{2}}\mathrm{d}t 
        & = \lim_{t \to r} \frac{D(t;r)}{t-r}-\frac{D(r+g(r);r)}{g(r)} + \int_{r}^{r+g(r)} \frac{\partial_{t}D(t;r)}{t-r}\mathrm{d}t\\
        & \le \int_{r}^{r+g(r)} \frac{(r-t)g^{\prime\prime}(t)}{t-r}\mathrm{d}t
		= g^{\prime}(r) - g^{\prime}(r+g(r))\le C\frac{g(r)}{r}.
		\end{split}
	\end{equation*}
	Similarly, by $g^{\prime}(r-g(r))\le g^{\prime}(r/2)\le Cg(r)/r$, 
	\begin{equation*}
		0\le\int_{r-g(r)}^{r} \frac{D(t;r)}{(t-r)^{2}}\mathrm{d}t \le C\frac{g(r)}{r}.
	\end{equation*}
	As a consequence, from \eqref{eq:Jg_sharp_estimate_D3},
    applying \eqref{eq:sharp_estimate_V} and \eqref{eq:sharp_estimate_g}  in the last of the following inequalities, 
    we obtain
	\begin{equation*}
		\begin{split}
		0 \le \frac{1}{2\pi} \int_{0}^{2\pi} \int_{r-g(r)}^{r+g(r)} g(t)V(t)\frac{D(t;r)}{\lvert X-Y\rvert^{3}}\mathrm{d}t\mathrm{d}\lambda 
        &\le C\left(V(r) +\frac{g(r)^{3}}{r^{2}}\right) \frac{g(r)}{r}\\
		&\le C(1+\log r + r^{-1/2})r^{-1/2} = o(1) \qquad\text{as } r \to \infty.
		\end{split}
	\end{equation*}
        This, together with \eqref{eq:Jg_sharp_estimate_D1} and \eqref{eq:Jg_sharp_estimate_D2}, gives the desired estimate \eqref{eq:Jg_sharp_estimate_D} and finishes the proof of the lemma. 
\end{proof}

Combining Lemmas \ref{lem:Jg_sharp_estimate_1} and \ref{lem:Jg_sharp_estimate_D} we obtain the expansion formula for $J_g(r)$.

\begin{corollary}\label{coro:Jg_sharp_estimate}
Assume that $g$ satisfies \eqref{eq:condition_g_sublinear} and \eqref{eq:concave_fb}. Let $J_g$ be defined in \eqref{eq:Jg_def} and $V$ be as in \eqref{eq:Vt_def}. Then one has
\begin{equation}\label{eq:Jg_sharp_estimate}
J_{g}(r) = V(r) + o(1) + \frac{1}{2\pi} \int_{0}^{2\pi} \int_{\frac r2}^{2r} g(t)g^{\prime}(t)\left(\frac{1}{\lvert X-Y\rvert} - \frac{1}{\lvert X\rvert}\right)\mathrm{d}t\mathrm{d}\lambda \qquad\text{as } r \to \infty.
\end{equation}
\end{corollary}

\begin{proof}
Recall that 
\begin{equation*}
J_{g}(r)= \frac{1}{2\pi} \int_{0}^{2\pi} \int_{\frac r2}^{2r} g(t)V(t)F_{g}(t,\lambda;r) + g(t)g^{\prime}(t)\left( \frac{1}{\lvert X-Y\rvert} - \frac{1}{\lvert X\rvert} \right) \mathrm{d}t\mathrm{d}\lambda,
\end{equation*}
where $F_g$ defined in \eqref{eq:Fg_def} can be rewritten as 
\begin{equation*}
F_{g}(t,\lambda;r) = \frac{g(r)(1-\cos\lambda)}{\lvert X-Y\rvert^{3}} + \frac{g(t)-g(r) +(r-t)g^{\prime}(t)}{\lvert X-Y\rvert^{3}} + \frac{tg^{\prime}(t)-g(t)}{\lvert X\rvert^{3}}.
\end{equation*}
In view of \eqref{eq:Hg_gvX}, 
	\begin{equation*}
		\begin{split}
		\left|\frac{1}{2\pi}\int_{0}^{2\pi}\int_{\frac r2}^{2r} g(t)V(t)\frac{tg^{\prime}(t)-g(t)}{\lvert X\rvert^{3}}\mathrm{d}t\mathrm{d}\lambda\right| 
        &\le C \int_{\frac r2}^{2r} \frac{1+\log t}{t^2}\mathrm{d}t\\ 
        &\le  C \frac{1+\log r}{r} =o(1) \qquad\text{as } r \to \infty.
		\end{split}
	\end{equation*}
        This combined with Lemmas \ref{lem:Jg_sharp_estimate_1} and \ref{lem:Jg_sharp_estimate_D} yields \eqref{eq:Jg_sharp_estimate}.
\end{proof}

In the following lemma we expand the integral on the right-hand side of \eqref{eq:Jg_sharp_estimate}. 

\begin{lemma}\label{lem:Jg_sharp_estimate_3}
Assume that $g$ satisfies \eqref{eq:condition_g_sublinear} and \eqref{eq:concave_fb}. For each $\eta \in (0,1/4)$, we have
\begin{equation*}
		\frac{1}{2\pi} \int_{0}^{2\pi} \int_{\frac r2}^{2r} \frac{g(t)g^{\prime}(t)}{\lvert X-Y\rvert}\mathrm{d}t\mathrm{d}\lambda \le (2+o(1))g(r+\eta r)g^{\prime}(r-\eta r)\log\frac{\eta r}{g(r)} + o(1) \qquad\text{as } r \to \infty,
	\end{equation*}
    and
	\begin{equation*}
		\frac{1}{2\pi} \int_{0}^{2\pi} \int_{\frac r2}^{2r} \frac{g(t)g^{\prime}(t)}{\lvert X-Y\rvert}\mathrm{d}t\mathrm{d}\lambda \ge (2-o(1))g(r-\eta r)g^{\prime}(r+\eta r)\log\frac{\eta r}{g(r)} - o(1) \qquad\text{as } r \to \infty.
	\end{equation*}
\end{lemma}

\begin{proof}
Fix $\eta\in(0,1/4)$ and let $r>4a$. Using $g(t)g^{\prime}(t)\le C(\log t)^{-1/8}$ for $t>2a$ (see \eqref{eq:sharp_estimate_gdg}), combined with $\lvert X-Y\rvert\ge \lvert t-r\rvert\ge \eta r$ for $t\in[r/2,r-\eta r]\cup[r+\eta r,2r]$, leads to 
	\begin{equation*}
		\begin{split}
		0 &\le \frac{1}{2\pi} \int_{0}^{2\pi} \int_{\frac r2}^{r-\eta r} \frac{g(t)g^{\prime}(t)}{\lvert X-Y\rvert}\mathrm{d}t\mathrm{d}\lambda+\frac{1}{2\pi} \int_{0}^{2\pi} \int_{r+\eta r}^{2r} \frac{g(t)g^{\prime}(t)}{\lvert
		 X-Y\rvert}\mathrm{d}t\mathrm{d}\lambda\\
		&\le C \frac{r-\eta r}{\eta r}\left(\log \frac r2\right)^{-\frac18}=o(1) \qquad\text{as } r \to \infty.
		\end{split}
	\end{equation*}
         Therefore, it suffices to show 
          \begin{equation}\label{eq:Jg_sharp_estimate_31}
		\frac{1}{2\pi}\int_{0}^{2\pi} \int_{r-\eta r}^{r+\eta r} \frac{g(t)g^{\prime}(t)}{\lvert X-Y\rvert}\mathrm{d}t\mathrm{d}\lambda \le (2+o(1))g(r+\eta r)g^{\prime}(r-\eta r)\log\frac{\eta r}{g(r)} \qquad\text{as } r \to \infty,
	\end{equation}
    and
       \begin{equation}\label{eq:Jg_sharp_estimate_32}
		\frac{1}{2\pi}\int_{0}^{2\pi} \int_{r-\eta r}^{r+\eta r} \frac{g(t)g^{\prime}(t)}{\lvert X-Y\rvert}\mathrm{d}t\mathrm{d}\lambda \ge (2-o(1))g(r-\eta r)g^{\prime}(r+\eta r)\log\frac{\eta r}{g(r)} \qquad\text{as } r \to \infty.
	\end{equation}
	
For $t\in[r-\eta r,r+\eta r]\subset(r/2,2r)$, with the help of $\lvert g(t)-g(r)\rvert\le g^{\prime}(r/2)\lvert t-r\rvert$ and \eqref{eq:X-Y_order1}, we get
\begin{equation*}
	\begin{split}
		\lvert X-Y\rvert^{2} &= (t-r)^{2} + (g(t)-g(r))^{2} + 2g(t)g(r)(1-\cos\lambda)\\
		&\ge (t-r)^{2} + 2g(r)^{2}(1-\cos\lambda) - 2g^{\prime}(r/2)\lvert t-r\rvert g(r)(1-\cos\lambda)\\
		&\ge (1-g^{\prime}(r/2)) \left((t-r)^{2} + 2g(r)^{2}(1-\cos\lambda)\right) \qquad \text{as } r\to\infty.
\end{split}\end{equation*}
This together with \eqref{eq:X-Y_order} yields
	\begin{equation*}
		\lvert X-Y\rvert^{2} = (1+o(1)) \big((t-r)^{2} + 2g(r)^{2}(1-\cos\lambda)\big) \qquad\text{as } r \to \infty,
	\end{equation*}
        where the $o(1)$-term is uniform with respect to  $t\in[r-\eta r,r+\eta r]$ and $\lambda\in[0,2\pi)$. 
	Hence
	\begin{equation}\label{eq:Jg_sharp_estimate_33}
		\begin{split}
		\int_{r-\eta r}^{r+\eta r} \frac{1}{\lvert X-Y\rvert}\mathrm{d}t 
        =(1+o(1))\int_{r-\eta r}^{r+\eta r} \frac{1}{\sqrt{(t-r)^{2}+2g(r)^{2}(1-\cos\lambda)}}\mathrm{d}t
        \qquad\text{as } r \to \infty.
		\end{split}
	\end{equation}
        A similar calculation as in  \eqref{eq:Jg_refine_gdg3} gives that for $\lambda\in(0,2\pi)$,
        \begin{equation*}
		\begin{split}
		\int_{r-\eta r}^{r+\eta r} \frac{1}{\sqrt{(t-r)^{2} + 2g(r)^{2}(1-\cos\lambda)}} \mathrm{d}t 
		&=\log\frac{\left(\eta r + \sqrt{(\eta r)^{2} + 2g(r)^{2}(1-\cos\lambda)}\right)^{2}}{2g(r)^{2}(1-\cos\lambda)}.
		\end{split}
		\end{equation*}
        Applying $g(r)=o(\eta r)$ as $r\to \infty$ to the square root on the right-hand side yields
\begin{equation*}
\sqrt{(\eta r)^2 + 2g(r)^2(1-\cos\lambda)}
= (1+o(1))\eta r \qquad\text{as }r\to\infty,
\end{equation*}
with uniform convergence for $t\in[r-\eta r,r+\eta r]$ and $\lambda\in[0,2\pi)$. 
Then it follows that for $\lambda\in(0,2\pi)$,
        \begin{equation}\label{eq:Jg_sharp_estimate_34}
		\begin{split}
		&\ \int_{r-\eta r}^{r+\eta r} \frac{1}{\sqrt{(t-r)^{2} + 2g(r)^{2}(1-\cos\lambda)}} \mathrm{d}t\\
        =&\ \log\frac{\left((2+o(1))\eta r\right)^2}{2g(r)^2(1-\cos\lambda)}
        =2\log\frac{\eta r}{g(r)} + \log(2+o(1)) - \log (1-\cos\lambda) \qquad\text{as } r \to \infty.
		\end{split}
		\end{equation}
Note that $\lambda \mapsto \log(1-\cos\lambda) \in L^{1}(0,2\pi)$. Thus, after integrating in $\lambda$, the terms $\log(2+o(1))$ and 
$\log(1-\cos\lambda)$ in \eqref{eq:Jg_sharp_estimate_34} contribute only $O(1)$. Therefore, from \eqref{eq:Jg_sharp_estimate_33} and \eqref{eq:Jg_sharp_estimate_34}, we obtain 
	\begin{equation}\label{eq:Jg_sharp_estimate_35}
		\begin{split}
		\frac{1}{2\pi}\int_{0}^{2\pi} \int_{r-\eta r}^{r+\eta r} \frac{1}{\lvert X-Y\rvert}\mathrm{d}t\mathrm{d}\lambda &= (2+o(1)) \log\frac{\eta r}{g(r)} +O(1)
		= (2+o(1))\log\frac{\eta r}{g(r)} \qquad\text{as } r \to \infty.
		\end{split}
	\end{equation}
    Since $g$ is non-decreasing while $g^{\prime}$ is non-increasing on $(a,\infty)$ (by \eqref{eq:condition_dg_sign} and \eqref{eq:concave_fb}), we have
\begin{equation*}
g(r-\eta r)g^{\prime}(r+\eta r) \le g(t)g^{\prime}(t) \le g(r+\eta r)g^{\prime}(r-\eta r) \qquad \text{for } t\in[r-\eta r,r+\eta r].
\end{equation*}
This, together with \eqref{eq:Jg_sharp_estimate_35}, yields \eqref{eq:Jg_sharp_estimate_31} and \eqref{eq:Jg_sharp_estimate_32}. Hence the proof of the lemma is completed.
\end{proof} 
Moreover, we simplify the head term up to higher order as follows.
\begin{lemma}\label{lem:Jg_sharp_estimate_4}
Assume that $g$ satisfies \eqref{eq:condition_g_sublinear} and \eqref{eq:concave_fb}. Then there exists a fixed constant $C_{B}$ such that
	\begin{equation}\label{est_lemma6.7}
		\frac{1}{2\pi}\int_{0}^{2\pi}\int_{a}^{2r} \frac{g(t)g^{\prime}(t)}{\lvert X\rvert}\mathrm{d}t\mathrm{d}\lambda = C_{B} + o(1)+\int_{a}^{2r} \frac{g(t)g^{\prime}(t)}{t}\mathrm{d}t  \qquad\text{as } r \to \infty.
	\end{equation}
\end{lemma}

\begin{proof}
	Note that $\lvert X\rvert=(t^2+g(t)^2)^{1/2}$. It follows that
        \begin{equation*}
		\begin{split}
		 \left|\frac{1}{\lvert X\rvert}-\frac{1}{t}\right| 
        = \frac{1}{t} \left|\left(1+\frac{g(t)^{2}}{t^{2}}\right)^{-1/2}-1\right| \le C \frac{g(t)^{2}}{t^{3}} \qquad\text{for }t>a.
		\end{split}
	\end{equation*}
        Thus by $g^{\prime}(t)\le 2g(t)/t$ for $t>2a$ and $g(t)\le Ct^{1/2}$ (see \eqref{eq:sharp_estimate_dg} and \eqref{eq:sharp_estimate_g}), we have  
	\begin{equation*}
		\left|g(t)g^{\prime}(t) \left(\frac{1}{\lvert X\rvert}-\frac{1}{t}\right)\right| \le C\frac{g(t)^4}{t^4} \le \frac{C}{t^{2}} \qquad\text{for }t>2a.
	\end{equation*}
        Note that $C/t^{2}\in L^{1}(2a,\infty)$. Therefore 
	\begin{equation*}
		\begin{split}
		&\ \frac{1}{2\pi}\int_{0}^{2\pi}\int_{a}^{2r} \frac{g(t)g^{\prime}(t)}{\lvert X\rvert}\mathrm{d}t\mathrm{d}\lambda - \int_{a}^{2r} \frac{g(t)g^{\prime}(t)}{t}\mathrm{d}t\\
        =&\ \int_{a}^{2r} g(t)g^{\prime}(t)\left(\frac{1}{\lvert X\rvert}-\frac{1}{t}\right)\mathrm{d}t
        =\int_{a}^{\infty} g(t)g^{\prime}(t)\left(\frac{1}{\lvert X\rvert}-\frac{1}{t}\right)\mathrm{d}t+o(1) \qquad\text{as } r \to \infty.
		\end{split}
	\end{equation*}
    Denote 
    \[
    C_B = \int_{a}^{\infty} g(t)g^{\prime}(t)\left(\frac{1}{\lvert X\rvert}-\frac{1}{t}\right)\mathrm{d}t.
    \]
    Then one immediately has \eqref{est_lemma6.7} and finishes the proof of the lemma.
\end{proof}
We are now in a position to prove the perturbed integral equality.

\begin{proposition}[The perturbed integral equality]\label{prop:lower_bound_log1}
Assume that $g$ satisfies \eqref{eq:condition_g_sublinear} and \eqref{eq:concave_fb}. Let $V$ be as in \eqref{eq:Vt_def}. There exists a fixed constant $C$ such that
	\begin{equation*}
		V(r) + (1+o(1))g(r)g^{\prime}(r)\log\frac{r}{g(r)} - \frac{1}{2} \int_{a}^{r} \frac{g(t)g^{\prime}(t)}{t}\mathrm{d}t = C+o(1) \qquad\text{as } r \to \infty.
	\end{equation*}
\end{proposition}

\begin{proof}
It follows from the identity \eqref{eq:pe_main_NHJT}, Lemmas \ref{lemma_1_potential_expansion},  \ref{lem:Hg_sharp_estimate} and \ref{lem:Tg_sharp_estimate}, as well as Corollary \ref{coro:Jg_sharp_estimate} that
    \begin{equation*}\begin{split}
		-V(r) + 2V(0)
		&= V(r) + \frac{1}{2\pi} \int_{0}^{2\pi}\int_{\frac r2}^{2r} \frac{g(t)g^{\prime}(t)}{\lvert X-Y\rvert}\mathrm{d}t\mathrm{d}\lambda - \frac{1}{2\pi}\int_{0}^{2\pi}\int_{a}^{2r} \frac{g(t)g^{\prime}(t)}{\lvert X\rvert}\mathrm{d}t\mathrm{d}\lambda\\
        &\quad\ +  C_{N} + C_{H} + o(1) \qquad\text{as } r \to \infty.
	 \end{split}\end{equation*} 
        Replacing the second integral by Lemma \ref{lem:Jg_sharp_estimate_4}, and consolidating constants, we obtain
	\begin{equation*}
		2V(r) + \frac{1}{2\pi}\int_{0}^{2\pi}\int_{\frac r2}^{2r} \frac{g(t)g^{\prime}(t)}{\lvert X-Y\rvert}\mathrm{d}t\mathrm{d}\lambda - \int_{a}^{2r} \frac{g(t)g^{\prime}(t)}{t}\mathrm{d}t = C + o(1) \qquad\text{as } r \to \infty.
	\end{equation*}
         Substituting the estimates from Lemma \ref{lem:Jg_sharp_estimate_3} into the left-hand side yields the two-sided bounds
	\begin{equation}\label{eq:lower_bound_log11}
		\begin{split}
		V(r) + (1+o(1))g(r+\eta r)g^{\prime}(r-\eta r)\log\frac{\eta r}{g(r)} - \frac{1}{2}\int_{a}^{2r} \frac{g(t)g^{\prime}(t)}{t}\mathrm{d}t \ge C - o(1),\\
		V(r) + (1-o(1))g(r-\eta r)g^{\prime}(r+\eta r)\log\frac{\eta r}{g(r)} - \frac{1}{2}\int_{a}^{2r} \frac{g(t)g^{\prime}(t)}{t}\mathrm{d}t \le C + o(1), 
		\end{split}
	\end{equation}
	as $r\to\infty$, where $\eta\in(0,1/4)$ and $C$ is a fixed constant independent of $\eta$.
	
Since $g(t) \le Ct^{1/2}(\log t)^{-1/8}$ and $g^{\prime}(t)\le Ct^{-1/2}$ for $t>2a$ (see \eqref{eq:sharp_estimate_g} and \eqref{eq:sharp_estimate_dg}), we have
	\begin{equation*}\begin{split}
		\lvert V(r\pm\eta r) - V(r)\rvert 
        &\le \left|\frac12\int_{r\pm\eta r}^{r}\lvert g^{\prime}(t)\rvert^{2} \mathrm{d}t\right| 
        \le \frac12g^{\prime}\left(\frac r2\right)g(2r)\le C(\log r)^{-\frac18}=o(1) \qquad\text{as }r\to\infty,
	\end{split}\end{equation*}
	and
	\begin{equation*}
		\left|\int_{r\pm\eta r}^{2r} \frac{g(t)g^{\prime}(t)}{t}\mathrm{d}t\right| \le C(\log r)^{-\frac18}\log 4= o(1) \qquad\text{as }r\to\infty.
	\end{equation*}
	Combining the above estimates with \eqref{eq:lower_bound_log11} gives
	\begin{equation}\label{eq:lower_bound_log12}
		\begin{split}
		V(r-\eta r) + (1+o(1))g(r+\eta r)g^{\prime}(r-\eta r)\log\frac{\eta r}{g(r)} - \frac{1}{2}\int_{a}^{r-\eta r} \frac{g(t)g^{\prime}(t)}{t}\mathrm{d}t \ge C - o(1),\\
		V(r+\eta r) + (1-o(1))g(r-\eta r)g^{\prime}(r+\eta r)\log\frac{\eta r}{g(r)} - \frac{1}{2}\int_{a}^{r+\eta r} \frac{g(t)g^{\prime}(t)}{t}\mathrm{d}t \le C + o(1),
		\end{split}
	\end{equation}
        as $r\to\infty$.
	
        In view of \eqref{eq:refined_ineq_6} and \eqref{eq:refined_ineq_5}, 
	\begin{equation}\label{eq:lower_bound_log13}
		g(r+\eta r) \ge g(r-\eta r) \ge \frac{g(r+\eta r)}{1+4\eta}
	\end{equation} 
        for sufficiently large $r$. Moreover, by the monotonicity of $g$ in \eqref{eq:condition_dg_sign},   
	\begin{equation}\label{eq:lower_bound_log14}
		\begin{split}
		&\log\frac{\eta r}{g(r)} \le \log\frac{\eta r}{g(r-\eta r)}=\log\frac{r-\eta r}{g(r-\eta r)} + \log\frac{\eta}{1-\eta},\\
		&\log\frac{\eta r}{g(r)} 
        \ge\log\frac{\eta r}{g(r+\eta r)}
        = \log\frac{r+\eta r}{g(r+\eta r)} + \log\frac{\eta}{1+\eta}.
		\end{split}
	\end{equation}  
	Let $\tilde{r} := r-\eta r$ and $\tilde{R} := r+\eta r$. From \eqref{eq:lower_bound_log12}-\eqref{eq:lower_bound_log14}, we obtain
	\begin{equation*}
		\begin{split}
		&V(\tilde{r}) + (1+o(1))(1+4\eta)g(\tilde{r})g^{\prime}(\tilde{r}) \left(\log\frac{\tilde{r}}{g(\tilde{r})} + \log\frac{\eta}{1-\eta}\right) - \frac{1}{2}\int_{a}^{\tilde{r}} \frac{g(t)g^{\prime}(t)}{t}\mathrm{d}t \ge C - o(1),\\
		&V(\tilde{R}) + (1-o(1))\frac{g(\tilde{R})}{1+4\eta}g^{\prime}(\tilde{R}) \left(\log\frac{\tilde{R}}{g(\tilde{R})} + \log\frac{\eta}{1+\eta}\right) - \frac{1}{2}\int_{a}^{\tilde{R}} \frac{g(t)g^{\prime}(t)}{t}\mathrm{d}t \le C + o(1),
		\end{split}
	\end{equation*}
      as $r\to\infty$. 
	Note that $\log(r/g(r))\to \infty$ as $r \to \infty$. For each $\eta\in (0,1/4)$, there exists sufficiently large  $R_{\eta}$ such that
	\begin{equation*}
		\begin{split}
		&V(\tilde{r}) + (1+8\eta)g(\tilde{r})g^{\prime}(\tilde{r})\log\frac{\tilde{r}}{g(\tilde{r})} - \frac{1}{2}\int_{a}^{\tilde{r}} \frac{g(t)g^{\prime}(t)}{t}\mathrm{d}t \ge C - \eta \qquad \text{for } \tilde{r} \ge (1-\eta)R_{\eta},\\
		&V(\tilde{R}) + (1-8\eta)g(\tilde{R})g^{\prime}(\tilde{R})\log\frac{\tilde{R}}{g(\tilde{R})} - \frac{1}{2}\int_{a}^{\tilde{R}} \frac{g(t)g^{\prime}(t)}{t}\mathrm{d}t \le C + \eta \qquad \text{for } \tilde{R} \ge (1+\eta)R_{\eta}.
		\end{split}
	\end{equation*}
	The desired estimate follows from the above bounds.
\end{proof}

\subsection{Positive lower bound for $\mathbf{g(r)g^{\prime}(r)\log(r/g(r))}$}\label{subsec:logbound}

The goal of this subsection is to show that, assuming the concavity of the fixed boundary $N$, a positive lower bound for $g(r)g^{\prime}(r)\log(r/g(r))$ holds, which plays a crucial role in establishing the desired growth estimate for the free boundary.

The following observation will be used in the subsequent argument.

\begin{lemma}\label{lem:Nconcave_liminf1}
Let $F_g$ be defined in \eqref{eq:Fg_def}. Then
\begin{equation*}
\frac{1}{2\pi} \int_{0}^{2\pi}\int_{1}^{\infty}g(t)F_{g}(t,\lambda;r)\mathrm{d}t\mathrm{d}\lambda = 1
\qquad \text{for } r > a.
\end{equation*}
\end{lemma}

\begin{proof}
Since $\phi + 1$ is also a velocity potential, the representation formula \eqref{eq:pe_main} holds with $V=\phi-\phi_{\infty}$
replaced by $V+1=\phi+1-\phi_{\infty}$ (recall that $\phi_\infty$ is the blow-down limit defined in \eqref{eq:condition_blowdown}). Hence, taking the difference of both representation formulas,
\begin{equation*}
	\begin{split}
		1 =&\ \big(-(V(r)+1) + 2(V(0)+1)\big) - \big(-V(r)+2V(0)\big)\\
		=&\ \frac{1}{2\pi} \int_{0}^{2\pi} \int_{1}^{\infty} g(t)\big(V(t)+1\big)F_{g}(t,\lambda;r) + g(t)g^{\prime}(t)\left(\frac{1}{\lvert X-Y\rvert} - \frac{1}{\lvert X\rvert}\right)\mathrm{d}t\mathrm{d}\lambda\\
		&\ - \frac{1}{2\pi} \int_{0}^{2\pi} \int_{1}^{\infty} g(t)V(t)F_{g}(t,\lambda;r) + g(t)g^{\prime}(t)\left(\frac{1}{\lvert X-Y\rvert} - \frac{1}{\lvert X\rvert}\right)\mathrm{d}t\mathrm{d}\lambda\\
		=&\ \frac{1}{2\pi} \int_{0}^{2\pi}\int_{1}^{\infty} g(t)F_{g}(t,\lambda;r)\mathrm{d}t\mathrm{d}\lambda.
	\end{split}
\end{equation*}
This finishes the proof. 
\end{proof}

We fix a free boundary point $Y_{0}:=(r_{0},g(r_{0}),0)$, where $r_0>a$ is given by the following lemma. 

\begin{lemma}\label{lem:Nconcave_liminf_r0}
Assume that $g$ satisfies \eqref{eq:condition_g_sublinear} and \eqref{eq:concave_fb}. There exists $r_0>a$ such that
\begin{equation}\label{eq:Nconcave_liminf_r0}
\lvert X-Y_{0}\rvert\le t-1 \qquad \text{for } t > r_{0}.
\end{equation}
\end{lemma}

\begin{proof}
Let $t>r_0>a$. Note that $g(r_0)\le g(t)$ by \eqref{eq:condition_dg_sign}. Direct calculations give
\begin{equation*}
\lvert X-Y_{0}\rvert^{2}= (t-r_{0})^{2} + (g(t)-g(r_{0}))^{2} + 2g(t)g(r_{0})(1-\cos\lambda)\le (t-r_{0})^{2}+4g(t)^{2}
\end{equation*}
and 
\begin{equation*}
(t-1)^2=(t-r_0)^2-(r_0^2-1)+2t(r_{0}-1).
\end{equation*} 
Hence
\begin{equation}\label{eq:Nconcave_liminf_r01}
\begin{split}
\lvert X-Y_{0}\rvert^{2}-(t-1)^{2} 
&\le 4g(t)^{2}+(r_0^2-1)-2t(r_{0}-1)\\
&= 2t(r_{0}-1)\left(\frac{2g(t)^{2}}{t(r_{0}-1)}+\frac{r_{0}+1}{2t}-1 \right).
\end{split}
\end{equation}
By $g(t)\le Ct^{1/2}$ in \eqref{eq:sharp_estimate_g} and $t>r_0$, we have  
\begin{equation*}
\frac{2g(t)^{2}}{t(r_{0}-1)} \le \frac{C}{r_{0}-1}\to 0
\quad\text{and}\quad
\frac{r_{0}+1}{2t}<\frac{r_{0}+1}{2r_{0}} = \frac{1}{2}+ \frac{1}{2r_{0}}\to\frac12 
\qquad\text{as } r_0\to\infty.
\end{equation*}
Therefore, the right-hand side of \eqref{eq:Nconcave_liminf_r01} is negative for sufficiently large $r_{0}$. 
This implies \eqref{eq:Nconcave_liminf_r0}. 
\end{proof}

From the identity \eqref{eq:pe_main}, we have 
\begin{equation}\label{eq:Vdiff_initial}
\begin{split}
    -V(r)+V(r_{0})
     &=  \big(-V(r)+2V(0)\big) - \big(-V(r_{0})+2V(0)\big)\\
     &= \frac{1}{2\pi} \int_{0}^{2\pi}\int_{1}^{\infty} g(t)V(t)\big(F_{g}(t,\lambda;r) - F_{g}(t,\lambda;r_{0})\big)\mathrm{d}t\mathrm{d}\lambda\\
     &\quad\ +\frac{1}{2\pi} \int_{0}^{2\pi}\int_{1}^{\infty} g(t)g^{\prime}(t)\left(\frac{1}{\lvert X-Y\rvert} - \frac{1}{\lvert X-Y_{0}\rvert}\right)\mathrm{d}t\mathrm{d}\lambda.
\end{split}
\end{equation}
Since Lemma \ref{lem:Nconcave_liminf1} gives
\begin{equation*}
\frac{1}{2\pi} \int_{0}^{2\pi}\int_{1}^{\infty} g(t)V(r_{0})\big(F_{g}(t,\lambda;r) - F_{g}(t,\lambda;r_{0})\big)\mathrm{d}t\mathrm{d}\lambda=0,   
\end{equation*}
the function $V(t)$ in the first integrand of \eqref{eq:Vdiff_initial} can be replaced by $V(t)-V(r_0)$. That is, 
\begin{equation}\label{eq:Vdiff_transformed}
\begin{split}
    -V(r)+V(r_{0})
    &=\frac{1}{2\pi}\int_{0}^{2\pi}\int_{1}^{\infty} g(t)\big(V(t)-V(r_{0})\big)\big(F_{g}(t,\lambda;r) - F_{g}(t,\lambda;r_{0})\big)\mathrm{d}t\mathrm{d}\lambda\\
    &\quad\ +\frac{1}{2\pi}\int_{0}^{2\pi}\int_{1}^{\infty} g(t)g^{\prime}(t)\left(\frac{1}{\lvert X-Y\rvert} - \frac{1}{\lvert X-Y_{0}\rvert}\right)\mathrm{d}t\mathrm{d}\lambda.
\end{split}
\end{equation}
We decompose the integrals in \eqref{eq:Vdiff_transformed} according to $t\in(1,r_0)$ and $t\in(r_0,\infty)$, and define
\begin{align}
&\begin{gathered}
    I_{0}(r;r_{0}):= \frac{1}{2\pi} \int_{0}^{2\pi}\int_{1}^{r_{0}} g(t)\big(V(t)-V(r_{0})\big)\big(F_{g}(t,\lambda;r) - F_{g}(t,\lambda;r_{0})\big)\mathrm{d}t\mathrm{d}\lambda\\
    \qquad + \frac{1}{2\pi} \int_{0}^{2\pi}\int_{1}^{r_{0}}  g(t)g^{\prime}(t)\left(\frac{1}{\lvert X-Y\rvert} - \frac{1}{\lvert X-Y_{0}\rvert}\right)\mathrm{d}t\mathrm{d}\lambda,
    \end{gathered} \label{eq:I0_def}\\
    &I_{1}(r;r_{0}): = \frac{1}{2\pi} \int_{0}^{2\pi} \int_{r_{0}}^{\infty} g(t)\big(V(t)-V(r_{0})\big)\big(F_{g}(t,\lambda;r) - F_{g}(t,\lambda;r_{0})\big)\mathrm{d}t\mathrm{d}\lambda,\label{eq:I1_def}\\
    &I_{2}(r;r_{0}): = \frac{1}{2\pi} \int_{0}^{2\pi}\int_{r_{0}}^{\infty} g(t)g^{\prime}(t)\left(\frac{1}{\lvert X-Y\rvert} - \frac{1}{\lvert X-Y_{0}\rvert}\right)\mathrm{d}t\mathrm{d}\lambda.\label{eq:I2_def}
\end{align}
With these quantities, \eqref{eq:Vdiff_transformed} is rewritten as
\begin{equation}\label{eq:Vdiff_final}
-V(r)+V(r_{0})
=I_{0}(r;r_{0}) + I_{1}(r;r_{0}) + I_{2}(r;r_{0}).
\end{equation}
We proceed to estimate the three integrals on the right-hand side separately. 

The estimate for $I_0(r;r_{0})$ is as follows.

\begin{lemma}\label{lem:Nconcave_liminf3}
Assume that $g$ satisfies \eqref{eq:condition_g_sublinear}, \eqref{eq:concave_N}, and \eqref{eq:concave_fb}. Let $I_0$ be defined in \eqref{eq:I0_def}. For each $\varepsilon \in (0,1)$, there exist constants $C_0=C_0(r_0) > 0$ independent of $\varepsilon$ and $R_{\varepsilon} > 2r_0$ such that
\begin{equation}\label{eq:Nconcave_liminf_I0}
I_{0}(r;r_{0}) \le -C_0 + \varepsilon\qquad \text{for } r > R_{\varepsilon}.
\end{equation}
\end{lemma}

\begin{proof}
The proof consists of two steps.

\textit{Step 1. Proof of the estimate \eqref{eq:Nconcave_liminf_I0} with an unsigned constant.} By the definition of $I_0$ in \eqref{eq:I0_def}, 
\begin{equation}\label{eq:I0_expansion}\begin{split}
    I_{0}(r;r_{0})&= \frac{1}{2\pi} \int_{0}^{2\pi}\int_{1}^{r_{0}} g(t)\big(V(t)-V(r_{0})\big)\big(F_{g}(t,\lambda;r) - F_{g}(t,\lambda;r_{0})\big)\mathrm{d}t\mathrm{d}\lambda\\
    &\quad\ + \frac{1}{2\pi} \int_{0}^{2\pi}\int_{1}^{r_{0}}  \frac{g(t)g^{\prime}(t)}{\lvert X-Y\rvert} \mathrm{d}t\mathrm{d}\lambda-\frac{1}{2\pi} \int_{0}^{2\pi}\int_{1}^{r_{0}}\frac{g(t)g^{\prime}(t)}{\lvert X-Y_{0}\rvert}\mathrm{d}t\mathrm{d}\lambda.
\end{split}
\end{equation}

It follows from the definition of $F_g$ in \eqref{eq:Fg_def} that
\begin{equation}\label{eq:Nconcave_liminf_I02}
F_{g}(t,\lambda;r) - F_{g}(t,\lambda;r_{0}) = \frac{g(t)+(r-t)g^{\prime}(t)-g(r)\cos\lambda}{\lvert X-Y\rvert^{3}} - \frac{g(t)+(r_{0}-t)g^{\prime}(t)-g(r_{0})\cos\lambda}{\lvert X-Y_{0}\rvert^{3}}.
\end{equation}
For $t\in (1,r_{0})$ and $r>2r_0$, we have
\begin{equation}\label{eq:Nconcave_liminf_I01}
\lvert X-Y\rvert=\left((t-r)^2+(g(t)-g(r))^2+2g(t)g(r)(1-\cos\lambda)\right)^{1/2}\ge \frac r2.
\end{equation}
Hence, 
\begin{equation*}
\left\lvert\frac{1}{2\pi} \int_{0}^{2\pi}\int_{1}^{r_{0}} g(t)\big(V(t)-V(r_{0})\big)\frac{g(t)+(r-t)g^{\prime}(t)-g(r)\cos\lambda}{\lvert X-Y\rvert^{3}}\mathrm{d}t\mathrm{d}\lambda\right\rvert\le \frac{C}{r^2} \qquad \text{as } r \to \infty,
\end{equation*}
where the constant $C$ depends on $r_0$. 
Therefore, in view of \eqref{eq:Nconcave_liminf_I02}, for fixed $r_0$ and each $\varepsilon \in (0,1)$, there exists $R_{\varepsilon}>2r_0$ such that
\begin{equation*}
	\begin{split}
		&\ \frac{1}{2\pi} \int_{0}^{2\pi} \int_{1}^{r_{0}} g(t)\big(V(t)-V(r_{0})\big) \big(F_{g}(t,\lambda;r) - F_{g}(t,\lambda;r_{0})\big)\mathrm{d}t\mathrm{d}\lambda\\
		\le&\ \varepsilon- \frac{1}{2\pi} \int_{0}^{2\pi} \int_{1}^{r_{0}} g(t)\big(V(t)-V(r_{0})\big) \frac{g(t)+(r_0-t)g^{\prime}(t)-g(r_{0})\cos\lambda}{\lvert X-Y_{0}\rvert^{3}}\mathrm{d}t\mathrm{d}\lambda
        \qquad \text{for } r > R_{\varepsilon}.
	\end{split}
\end{equation*}
Also, by \eqref{eq:Nconcave_liminf_I01}, we have  
\begin{equation*}
\begin{split}
\frac{1}{2\pi} \int_{0}^{2\pi} \int_{1}^{r_{0}}\frac{ g(t)g^{\prime}(t)}{\lvert X-Y\rvert} \mathrm{d}t\mathrm{d}\lambda \le C\frac{g^2(r_0)}{r}
\le \varepsilon
\qquad \text{for } r > R_{\varepsilon}
\end{split}
\end{equation*}
with a possibly larger $R_\varepsilon$. 
Plugging the above two estimates into \eqref{eq:I0_expansion} yields \eqref{eq:Nconcave_liminf_I0} with $C_0$ defined by
\begin{equation}\label{eq:Nconcave_liminf_C0}
C_0 := \frac{1}{2\pi} \int_{0}^{2\pi} \int_{1}^{r_{0}} g(t)\big(V(t)-V(r_{0})\big) \frac{g(t)+(r_0-t)g^{\prime}(t)-g(r_{0})\cos\lambda}{\lvert X-Y_{0}\rvert^{3}} + \frac{g(t)g^{\prime}(t)}{\lvert X-Y_{0}\rvert}\mathrm{d}t\mathrm{d}\lambda.
\end{equation}

\textit{Step 2. Positivity of the constant $C_0$.} 
Note that $g(t)+(r-t)g^{\prime}(t)\ge g(r)$ for $t,\, r>1$ by the concavity of $g$ on $(1,\infty)$ (see \eqref{eq:concave_N} and \eqref{eq:concave_fb}). Thus
\begin{equation}\label{eq:g_lambda}
g(t)+(r_0-t)g^{\prime}(t)-g(r_{0})\cos\lambda \ge g(r_{0})(1-\cos\lambda) \ge 0 \qquad\text{for } t>1.
\end{equation}
Moreover, from the expression of $V$ in \eqref{eq:Vt_expression}, we have
\begin{equation*}\begin{split}
V(t)-V(r_{0}) 
&= -\int^{r_{0}}_{t}\left(\partial_{s}\phi(\tau)\sqrt{1 + \lvert g^{\prime}(\tau)\rvert^{2}} - 1\right) \mathrm{d}\tau\\ 
&\ge \int_{t}^{r_{0}} \left(1 - \sqrt{1 + \lvert g^{\prime}(\tau)\rvert^{2}}\right)\mathrm{d}\tau =: \tilde{V}(t;r_{0}) \qquad\text{for } t\in(1,r_0),
\end{split}\end{equation*}
where $\lvert \nabla \phi \rvert \le 1$ on $N$ (see \eqref{eq:condition_speed}) and $\lvert \nabla \phi \rvert=1$ on $\Gamma$ (see \eqref{eq:cavity_pde}) have been used in the inequality.  
Thus $C_0$ defined in \eqref{eq:Nconcave_liminf_C0} has the lower bound 
\begin{equation*}
C_0 \ge \frac{1}{2\pi} \int_{0}^{2\pi}\int_{1}^{r_{0}} \frac{g(t)}{\lvert X-Y_{0}\rvert^{3}} \left(\tilde{V}(t;r_{0}) \big(g(t)+(r_0-t)g^{\prime}(t)-g(r_{0})\cos\lambda\big) + g^{\prime}(t)\lvert X-Y_{0}\rvert^{2}\right)\mathrm{d}t\mathrm{d}\lambda.
\end{equation*}

Since $g\not\equiv0$ on $[1,\infty)$ (see \eqref{assumption_g_not_zero}), there exists a largest $\tau_0\in(1,r_0]$ such that $g^{\prime}>0$ on $(1,\tau_0)$. If $\tau_0<r_0$ and $g^{\prime}(\tau_0)=0$, then by concavity $g^{\prime}\equiv0$ on $[\tau_0,\infty)$, and hence $ \tilde{V}(t;r_{0})=0$ for all $t\in(\tau_0,r_0]$. 
Thus it suffices to show that 
\begin{equation}\label{eq:C0_lower_bound}
\frac{\tilde{V}(t;r_{0})}{g^{\prime}(t)} \big(g(t)+(r_0-t)g^{\prime}(t)-g(r_{0})\cos\lambda\big) + \lvert X-Y_{0}\rvert^{2} > 0 \qquad \text{for } t \in (1,\tau_{0}).
\end{equation}

The concavity of $g$ on $(1,\infty)$ implies that $g^{\prime}(\tau)\le g^{\prime}(t)$ for $\tau\ge t>1$, hence ---using the fact that $z\mapsto -z/(1+\sqrt{1+z^2})$ is non-increasing--- we get
\begin{equation*}
	\begin{split}
		\frac{\tilde{V}(t;r_{0})}{g^{\prime}(t)} 
        &\geq \frac{1}{g^{\prime}(t)} \int_{t}^{r_{0}} \frac{-\lvert g^{\prime}(\tau)\rvert^{2}}{1 + \sqrt{1 + \lvert g^{\prime}(\tau)\rvert^{2}}}\mathrm{d}\tau 
        \ge \frac{1}{g^{\prime}(t)} \int_{t}^{r_{0}} \frac{-g^{\prime}(t)g^{\prime}(\tau)}{1 + \sqrt{1 + \lvert g^{\prime}(t)\rvert^{2}}}\mathrm{d}\tau\\
		&= \frac{g(t)-g(r_{0})}{1+\sqrt{1 + \lvert g^{\prime}(t)\rvert^{2}}} \qquad \text{for } t \in (1,\tau_{0}).
	\end{split}
\end{equation*}
Therefore, 
\begin{equation*}
	\begin{split}
		& \frac{\tilde{V}(t;r_{0})}{g^{\prime}(t)} \big( g(t)+(r_0-t)g^{\prime}(t)-g(r_{0})\cos\lambda \big) + \lvert X-Y_{0}\rvert^{2}\\
		\ge &\ \frac{g(t)-g(r_{0})}{1+\sqrt{1+\lvert g^{\prime}(t)\rvert^{2}}} \big( g(t)-g(r_{0}) + (r_{0}-t)g^{\prime}(t) + g(r_{0})(1-\cos\lambda) \big)\\
		&\ + (t-r_{0})^{2} + (g(t)-g(r_{0}))^{2} + 2g(t)g(r_{0})(1-\cos\lambda)\\
		= &\ \frac{(t-r_{0})^{2}}{1+\sqrt{1+\lvert g^{\prime}(t)\rvert^{2}}} \left( \frac{(g(t)-g(r_{0}))^{2}}{(t-r_{0})^{2}} - \frac{g(t)-g(r_{0})}{t-r_{0}}g^{\prime}(t) + \frac{g(t)-g(r_{0})}{t-r_{0}} \frac{g(r_{0})}{t-r_{0}}(1-\cos\lambda) \right)\\
		&\ + (t-r_{0})^{2} \left( 1+ \frac{(g(t)-g(r_{0}))^{2}}{(t-r_{0})^{2}} + \frac{2g(t)g(r_{0})}{(t-r_{0})^{2}}(1-\cos\lambda) \right) \qquad \text{for } t \in (1,\tau_{0}).
	\end{split}
\end{equation*}

Applying now for $t \in (1,\tau_{0})$ Lemma \ref{lem:Nconcave_liminf2} to the right-hand side of the preceding formula with
\begin{equation*}
p=g^{\prime}(t),\quad x=\frac{g(t)-g(r_{0})}{t-r_{0}},\quad 
y=-\frac{g(t)}{t-r_{0}}, \quad w=1-\cos\lambda,
\end{equation*}
we obtain \eqref{eq:C0_lower_bound}.
\end{proof}

We now estimate $I_1(r;r_{0})$.

\begin{lemma}\label{lem:Nconcave_liminf4}
Assume that $g$ satisfies \eqref{eq:condition_g_sublinear} and \eqref{eq:concave_fb}. Let $I_1$ be defined in \eqref{eq:I1_def} and $V$ be as in \eqref{eq:Vt_def}. For each $\varepsilon \in (0,1)$, there exists $R_{\varepsilon}>2r_0$ such that
\begin{equation}\label{eq:Nconcave_liminf_I1}
I_{1}(r;r_{0}) \le V(r)-V(r_{0}) + \varepsilon\qquad \text{for } r > R_{\varepsilon}.
\end{equation}
\end{lemma}

\begin{proof}
From the expression of $V$ in \eqref{eq:Vt_def}, we have
\begin{equation}\label{eq:V_diff_positive}
V(t)-V(r_{0}) = \int_{r_{0}}^{t} \left(\sqrt{1+\lvert g^{\prime}(\tau)\rvert^{2}}-1\right)\mathrm{d}\tau\ge0 \qquad\text{for } t>r_0.
\end{equation}
Moreover, as in \eqref{eq:g_lambda}, by the concavity of $g$ on $(a,\infty)$,  
\begin{equation}\label{eq:concav6}
g(t)+(r-t)g^{\prime}(t)-g(r)\cos\lambda\ge g(r)(1-\cos\lambda)\ge 0 \qquad\text{for } t,\, r>a.
\end{equation}
Hence, in view of the expression of $F_{g}(t,\lambda;r) - F_{g}(t,\lambda;r_{0})$ in \eqref{eq:Nconcave_liminf_I02}, 
$I_1$ defined in \eqref{eq:I1_def} satisfies ---applying \eqref{eq:concav6} with $r=r_0$---
\begin{equation*}\begin{split}
I_{1}(r;r_{0}) &= \frac{1}{2\pi} \int_{0}^{2\pi} \int_{r_{0}}^{\infty} g(t)\big(V(t)-V(r_{0})\big)\big(F_{g}(t,\lambda;r) - F_{g}(t,\lambda;r_{0})\big)\mathrm{d}t\mathrm{d}\lambda\\
&\le \frac{1}{2\pi} \int_{0}^{2\pi} \int_{r_{0}}^{\infty} g(t)\big(V(t)-V(r_{0})\big)\frac{g(t)+(r-t)g^{\prime}(t)-g(r)\cos\lambda}{\lvert X-Y\rvert^{3}}\mathrm{d}t\mathrm{d}\lambda.
\end{split}\end{equation*}

From \eqref{eq:V_diff_positive}, \eqref{eq:concav6} and \eqref{eq:Hg_sharp_estimate_2} we infer that
\begin{equation*}
\frac{1}{2\pi} \int_{0}^{2\pi} \int_{r_{0}}^{\frac r2} g(t)\big(V(t)-V(r_{0})\big)\frac{g(t)+(r-t)g^{\prime}(t)-g(r)\cos\lambda}{\lvert X-Y\rvert^{3}}\mathrm{d}t\mathrm{d}\lambda =o(1) \qquad\text{as } r \to \infty.
\end{equation*}
The same arguments as in the proofs of \eqref{eq:Tg_sharp_estimate_1} and \eqref{eq:Tg_sharp_estimate_3} give
\begin{equation*}
\frac{1}{2\pi} \int_{0}^{2\pi} \int_{2r}^{\infty} g(t)\big(V(t)-V(r_{0})\big)\frac{g(t)+(r-t)g^{\prime}(t)-g(r)\cos\lambda}{\lvert X-Y\rvert^{3}}\mathrm{d}t\mathrm{d}\lambda= o(1) \qquad\text{as } r \to \infty.
\end{equation*}
Additionally, Lemma \ref{lem:Jg_sharp_estimate_1} together with Lemma \ref{lem:Jg_sharp_estimate_D}, with $V(t)$ replaced by $V(t)-V(r_0)$ in both, yields
\begin{equation*}\begin{split}
&\ \frac{1}{2\pi} \int_{0}^{2\pi} \int_{\frac r2}^{2r} g(t)\big(V(t)-V(r_{0})\big) \frac{g(t)+(r-t)g^{\prime}(t)-g(r)\cos\lambda}{\lvert X-Y\rvert^{3}}\mathrm{d}t\mathrm{d}\lambda\\ 
=&\ V(r)-V(r_{0}) + o(1) \qquad\text{as } r \to \infty.
\end{split}\end{equation*}
Combining the above estimates yields \eqref{eq:Nconcave_liminf_I1}.
\end{proof}

It remains to estimate $I_2(r;r_{0})$.
 
\begin{lemma}\label{lem:Nconcave_liminf5}
Assume that $g$ satisfies \eqref{eq:condition_g_sublinear} and \eqref{eq:concave_fb}. Let $I_2$ be defined in \eqref{eq:I2_def}. For each $\varepsilon \in (0,1)$ and each  $\eta \in (0,1/4)$, there exists $R_{\varepsilon,\eta}>2r_0$ such that
\begin{equation*}
I_{2}(r;r_{0}) \le (2+\varepsilon) g(r+\eta r)g^{\prime}(r-\eta r)\log\frac{\eta r}{g(r)} - \frac{1}{2\pi}\int_{0}^{2\pi}\int_{r_{0}}^{\frac r2} \frac{g(t)g^{\prime}(t)}{\lvert X-Y_{0}\rvert}\mathrm{d}t\mathrm{d}\lambda + \varepsilon \qquad \text{for } r > R_{\varepsilon,\eta}.
\end{equation*}
\end{lemma}

\begin{proof}
Fix $\varepsilon\in(0,1)$ and $\eta\in(0,1/4)$. Let $r>2r_0$. As in Lemma \ref{lem:Nconcave_liminf4}, we estimate
$I_{2}(r;r_{0})$ by splitting the integration interval $(r_0,\infty)$ into three subintervals $(r_0,r/2)$, $(r/2,2r)$, and $(2r,\infty)$.

By virtue of  \eqref{eq:Hg_sharp_estimate_3}, there exists $R_{\varepsilon}>2r_0$ such that 
\begin{equation}\label{eq:Nconcave_liminf_I21}
\begin{split}
&\ \frac{1}{2\pi} \int_{0}^{2\pi}\int_{r_{0}}^{\frac r2} g(t)g^{\prime}(t)\left(\frac{1}{\lvert X-Y\rvert} - \frac{1}{\lvert X-Y_{0}\rvert}\right)\mathrm{d}t\mathrm{d}\lambda \\
\le&\ \varepsilon- \frac{1}{2\pi}\int_{0}^{2\pi}\int_{r_{0}}^{\frac r2} \frac{g(t)g^{\prime}(t)}{\lvert X-Y_{0}\rvert}\mathrm{d}t\mathrm{d}\lambda \qquad \text{for } r > R_{\varepsilon}.
\end{split}\end{equation}
If $r$ is sufficiently large,
for $t>2r$,  one has 
\begin{equation*}
\lvert X-Y\rvert\lvert X-Y_0\rvert\ge \lvert t-r\rvert\lvert t-r_0\rvert\ge \frac {t^2}4,
\end{equation*}
and
\begin{equation*}
\lvert Y-Y_0\rvert=((r-r_0)^2+(g(r)-g(r_0))^2)^{1/2}\le Cr. 
\end{equation*}
Therefore, it holds that 
\begin{equation*}
\left|\frac{1}{\lvert X-Y\rvert} - \frac{1}{\lvert X-Y_{0}\rvert}\right| \le \frac{\lvert Y-Y_{0}\rvert}{\lvert X-Y\rvert \lvert X-Y_{0}\rvert} \le \frac{Cr}{t^{2}}.
\end{equation*}
This together with $g(t)g^{\prime}(t)\le C(\log t)^{-1/8}$ for $t>2a$ in \eqref{eq:sharp_estimate_gdg} gives
\begin{equation}\label{eq:Nconcave_liminf_I22}
\begin{split}
&\ \frac{1}{2\pi} \int_{0}^{2\pi}\int_{2r}^{\infty} g(t)g^{\prime}(t)\left(\frac{1}{\lvert X-Y\rvert} - \frac{1}{\lvert X-Y_{0}\rvert}\right)\mathrm{d}t\mathrm{d}\lambda\\ 
\le&\ C\int_{2r}^{\infty} (\log t)^{-\frac18}\frac r{t^{2}}\mathrm{d}t
\le C(\log r)^{-\frac18} 
<\varepsilon \qquad \text{for } r>R_{\varepsilon},
\end{split}
\end{equation}
with a possibly larger $R_{\varepsilon}$. 
Finally, by Lemma \ref{lem:Jg_sharp_estimate_3}, there exists $R_{\varepsilon,\eta}>R_{\varepsilon}$ such that one has
\begin{equation}\label{eq:Nconcave_liminf_I23}
\begin{split}
&\ \frac{1}{2\pi} \int_{0}^{2\pi}\int_{\frac r2}^{2r} g(t)g^{\prime}(t)\left(\frac{1}{\lvert X-Y\rvert} - \frac{1}{\lvert X-Y_{0}\rvert}\right)\mathrm{d}t\mathrm{d}\lambda \\
\le&\  \frac{1}{2\pi}\int_{0}^{2\pi}\int_{\frac r2}^{2r} \frac{g(t)g^{\prime}(t)}{\lvert X-Y\rvert}\mathrm{d}t\mathrm{d}\lambda
\le (2+\varepsilon)g(r+\eta r)g^{\prime}(r-\eta r)\log\frac{\eta r}{g(r)} + \varepsilon \ \quad \text{for } r > R_{\varepsilon,\eta}.
\end{split}
\end{equation}
Combining \eqref{eq:Nconcave_liminf_I21}, \eqref{eq:Nconcave_liminf_I22}, \eqref{eq:Nconcave_liminf_I23}, and the definition of $I_2(r;r_0)$ in \eqref{eq:I2_def} yields the desired estimate.
\end{proof}

We shall need the following estimate prior to the proof of the lower bound for $g(r)g^{\prime}(r)\log(r/g(r))$.

\begin{lemma}\label{lem:Nconcave_liminf6}
Assume that $g$ satisfies \eqref{eq:condition_g_sublinear}, \eqref{eq:concave_N}, and \eqref{eq:concave_fb}.  Let $V$ be as in \eqref{eq:Vt_def}. For each $\varepsilon \in (0,1)$, there exists $R_{\varepsilon}>2r_0$ such that
\begin{equation*}
2\big(V(r)-V(r_{0})\big)\le \int_{r_{0}}^{\frac r2} \frac{g(t)g^{\prime}(t)}{t-1}\mathrm{d}t+\varepsilon \qquad\text{for }r>R_{\varepsilon}.
\end{equation*}
\end{lemma}
\begin{proof}
By the expression of $V$ in \eqref{eq:Vt_def}, one has
\begin{equation*}\begin{split}
2\big(V(r)-V(r_{0})\big) &=2\int_{r_{0}}^{r} \left(\sqrt{1+\lvert g^{\prime}(t)\rvert^{2}}-1\right)\mathrm{d}t\\
&\le \int_{r_{0}}^{r} \lvert g^{\prime}(t)\rvert^{2}\mathrm{d}t = \int_{r_{0}}^{\frac r2} \lvert g^{\prime}(t)\rvert^{2}\mathrm{d}t + \int_{\frac r2}^{r} \lvert g^{\prime}(t)\rvert^{2}\mathrm{d}t \qquad\text{for } r>2r_0.
\end{split}\end{equation*}
The concavity of $g$ on $(1,\infty)$ (see \eqref{eq:concave_N} and \eqref{eq:concave_fb}) implies $g^{\prime}(t)\le g(t)/(t-1)$ for $t>1$. Hence
\begin{equation*}
\int_{r_{0}}^{\frac r2} \lvert g^{\prime}(t)\rvert^{2}\mathrm{d}t \le \int_{r_{0}}^{\frac r2} \frac{g(t)g^{\prime}(t)}{t-1}\mathrm{d}t \qquad\text{for } r>2r_0.
\end{equation*}
Moreover, using $g(t) \le Ct^{1/2}(\log t)^{-1/8}$ and $g^{\prime}(t)\le Ct^{-1/2}$ for $t>2a$ (see \eqref{eq:sharp_estimate_g} and \eqref{eq:sharp_estimate_dg}) yields
\begin{equation*}\begin{split}
\int_{\frac r2}^{r} \lvert g^{\prime}(t)\rvert^{2}\mathrm{d}t &\le g^{\prime}(r/2)\int_{\frac r2}^rg^{\prime}(t)\mathrm dt\le 
g^{\prime}(r/2)g(r)\le C(\log r)^{-1/8}=o(1)
\qquad\text{as } r \to \infty.
\end{split}\end{equation*}
Combining the above estimates finishes the proof of the lemma.
\end{proof}

We are now ready to prove the lower bound for $g(r)g^{\prime}(r)\log(r/g(r))$. 

\begin{proposition}\label{prop:concave_log}
Assume that $g$ satisfies \eqref{eq:condition_g_sublinear}, \eqref{eq:concave_N}, and \eqref{eq:concave_fb}. Then
\begin{equation*}
\liminf_{r \to \infty} g(r)g^{\prime}(r)\log\frac{r}{g(r)} > 0.
\end{equation*}
\end{proposition}

\begin{proof}
For each $\varepsilon\in(0,1)$ and $\eta\in(0,1/4)$, from  \eqref{eq:Vdiff_final} and Lemmas \ref{lem:Nconcave_liminf3}-\ref{lem:Nconcave_liminf5}, we obtain
\begin{equation*}
	\begin{split}
		-V(r)+V(r_{0})
		\le&\ -C_0 + V(r)-V(r_{0}) + (2+\varepsilon)g(r+\eta r)g^{\prime}(r-\eta r)\log\frac{\eta r}{g(r)}\\
		&\ - \frac{1}{2\pi}\int_{0}^{2\pi}\int_{r_{0}}^{\frac r2} \frac{g(t)g^{\prime}(t)}{\lvert X-Y_{0}\rvert}\mathrm{d}t\mathrm{d}\lambda + \varepsilon \qquad \text{for } r > R_{\varepsilon,\eta},
	\end{split}
\end{equation*}
where the constants $C_0>0$ and $R_{\varepsilon,\eta}>2r_0$. Equivalently, 
\begin{equation*}\begin{split}
&\ (2+\varepsilon)g(r+\eta r)g^{\prime}(r-\eta r)\log\frac{\eta r}{g(r)} \\
\ge&\ C_0 - 2\big(V(r)-V(r_{0})\big) + \frac{1}{2\pi}\int_{0}^{2\pi} \int_{r_{0}}^{\frac r2} \frac{g(t)g^{\prime}(t)}{\lvert X-Y_{0}\rvert}\mathrm{d}t\mathrm{d}\lambda- \varepsilon 
\qquad \text{for } r > R_{\varepsilon,\eta}.
\end{split}\end{equation*}
In view of $\lvert X-Y_0\lvert\le t-1$ for $t>r_0$ (see Lemma \ref{lem:Nconcave_liminf_r0}) and Lemma \ref{lem:Nconcave_liminf6}, we further deduce
\begin{equation}\label{eq:Nconcave_liminf61}\begin{split}
(2+\varepsilon)g(r+\eta r)g^{\prime}(r-\eta r) \log\frac{\eta r}{g(r)} 
&\ge C_0 - 2\big(V(r)-V(r_{0})\big) + \int_{r_{0}}^{\frac r2} \frac{g(t)g^{\prime}(t)}{t-1}\mathrm{d}t- \varepsilon\\
&\ge C_0-2\varepsilon
\qquad \text{for } r > R_{\varepsilon,\eta},
\end{split}\end{equation}
where $R_{\varepsilon,\eta}$ is chosen larger if necessary.

Let $\eta = 1/10$ and $\tilde{r}:=r-\eta r$. As shown in \eqref{eq:refined_ineq_6} and \eqref{eq:refined_ineq_5}, by the concavity of $g$, 
\begin{equation*}
g(r+\eta r) \le Cg(\tilde{r})
\end{equation*}
for some $C>0$. On the other hand, using the monotonicity of $g$ (see \eqref{eq:condition_dg_sign}) gives 
\begin{equation*}
\log\frac{\eta r}{g(r)} 
\le \log \frac{\eta r}{g(\tilde{r})} = \log \frac{\tilde{r}}{g(\tilde{r})} + \log\frac{\eta}{1-\eta} = (1+o(1))\log\frac{\tilde{r}}{g(\tilde{r})} \qquad\text{as } r \to \infty.
\end{equation*}
Plugging the upper bounds of $g(r+\eta r)$ and $\log(\eta r/g(r))$ into \eqref{eq:Nconcave_liminf61} leads to
\begin{equation*}
C(2+\varepsilon)(1+o(1))g(\tilde{r})g^{\prime}(\tilde{r})\log \frac{\tilde{r}}{g(\tilde{r})} \ge C_0-2\varepsilon \qquad \text{for } r > R_{\varepsilon,1/10}.
\end{equation*}
Letting $r \to \infty$ and then $\varepsilon \to 0$ yields
\begin{equation*}
	\liminf_{\tilde{r}\to\infty} g(\tilde{r})g^{\prime}(\tilde{r})\log\frac{\tilde{
    r}}{g(\tilde{r})} \ge \frac{C_0}{2C} > 0.
\end{equation*}
Hence the proof of the proposition is completed.
\end{proof}

\subsection{A Tauberian step and proof of the lower growth bound as well as the main theorems}\label{subsec:taub}

In this subsection, we show that the positive lower bound on $g(r)g^{\prime}(r)\log(r/g(r))$ in Proposition \ref{prop:concave_log} gives the desired lower estimate for $g(r)$. 
The following Tauberian-type growth proposition, based on phase-space analysis, bridges the refined asymptotic equality from Proposition \ref{prop:lower_bound_log1}, with error improved to $o(g(r)g^{\prime}(r)\log(r/g(r)))$, to the desired lower bound for $g$.

\begin{proposition}[A Tauberian-type growth result]\label{prop:Tauberian}
Assume that $g$ satisfies \eqref{eq:condition_g_sublinear} and \eqref{eq:concave_fb}. Let $V$ be as in \eqref{eq:Vt_def}. If there exists a fixed constant $\tilde C$ (not necessarily to be positive) such that 
\begin{equation}\label{eq:Tauberian_condition}
V(r) + g(r)g^{\prime}(r)\log\frac{r}{g(r)} - \frac{1}{2}\int_{a}^{r} \frac{g(t)g^{\prime}(t)}{t}\mathrm{d}t 
= \tilde C+o\left(g(r)g^{\prime}(r)\log\frac{r}{g(r)}\right) \qquad\text{as } r \to \infty,
\end{equation}
then 
\begin{equation*}
\liminf_{r \to \infty} \frac{g(r)}{r^{1/2}(\log r)^{\delta}} > 0 \qquad \text{for every } \delta < -\frac{1}{4}.
\end{equation*}
\end{proposition}

\begin{proof}
First, we denote
\begin{equation}\label{eq:Tauberian_U_def}
U(r):= \frac{1}{2}\int_{a}^{r}\frac{g(t)g^{\prime}(t)}{t}\mathrm{d}t - V(r) + \tilde C, \qquad x:=\log r,
\end{equation}
and let
\begin{equation}\label{eq:Tauberian_notation}\begin{split}
&p(x):=\frac{e^xU(e^x)}{g(e^x)^2\log(e^x/g(e^x))}=\frac{rU(r)}{g(r)^2\log(r/g(r))},\\
&w(x):=2\log g(e^x)-x=2\log g(r)-\log r.  
\end{split}\end{equation}
The proof proceeds in four steps.

\textit{Step 1. Estimate for $p^{\prime}(x)/p(x)$.} Since
\begin{equation*}
\log p(x) = \log\frac{e^{x}U(e^{x})}{g(e^{x})^{2}\log(e^x/g(e^x))} = x + \log U(e^{x}) - 2\log g(e^{x}) - \log\log(e^x/g(e^x)),
\end{equation*}
differentiating with respect to $x$ gives
\begin{equation}\label{eq:Tauberian_pdp_expression}
\begin{split}
\frac{p^{\prime}(x)}{p(x)}=\left(\log p(x)\right)^{\prime}&=1+\frac{e^xU^{\prime}(e^x)}{U(e^x)}-\frac{2e^xg^{\prime}(e^x)}{g(e^x)}-\frac{g(e^x)-e^xg^{\prime}(e^x)}{g(e^x)\log(e^x/g(e^x))}\\
&=1+\frac{rU^{\prime}(r)}{U(r)}-\frac{2rg^{\prime}(r)}{g(r)}-\frac{g(r)-rg^{\prime}(r)}{g(r)\log(r/g(r))}.
\end{split}    
\end{equation}

We claim that
\begin{equation}\label{eq:Tauberian_pdp}
\frac{p^{\prime}(x)}{p(x)} = 1  + o(1) - (2+o(1))p(x)
\qquad\text{as } x \to \infty.
\end{equation}
To prove this, we estimate the last three terms on the right-hand side of \eqref{eq:Tauberian_pdp_expression}. 
It follows from the definition of $U$ and the expression of $V$ (see \eqref{eq:Tauberian_U_def} and \eqref{eq:Vt_def}) that
\begin{equation*}
U^{\prime}(r)=\frac{g(r)g^{\prime}(r)}{2r}-\left(\sqrt{1+\lvert g^{\prime}(r)\rvert^2}-1\right)=\frac{g(r)g^{\prime}(r)}{2r}-\frac{\lvert g^{\prime}(r)\rvert^2}{\sqrt{1+\lvert g^{\prime}(r)\rvert^2}+1} \qquad\text{for } r>a.
\end{equation*}
Moreover, in view of \eqref{eq:Tauberian_condition}, one has
\begin{equation}\label{eq:Tauberian_U_estimate}
U(r) =(1+o(1))g(r)g^{\prime}(r)\log\frac{r}{g(r)}
\qquad\text{as } r \to \infty.
\end{equation}
Hence, using $g(r)\le Cr^{1/2}$ and $0\le g^{\prime}(r)\le 2g(r)/r$ for $r>2a$ (see \eqref{eq:sharp_estimate_g} and \eqref{eq:sharp_estimate_dg}) yields
\begin{equation}\label{eq:Tauberian_pdp1}
\frac{rU^{\prime}(r)}{U(r)}
=\frac{1+o(1)}{\log(r/g(r))}\left(\frac{1}{2}-\frac{rg^{\prime}(r)}{g(r)}\cdot\frac{1}{\sqrt{1+\lvert g^{\prime}(r)\rvert^2}+1}\right)\to0  \qquad\text{as } r \to \infty.
\end{equation}
Also, 
\begin{equation}\label{eq:Tauberian_pdp3}
    \frac{g(r)-rg^{\prime}(r)}{g(r)\log(r/g(r))}\to0 
    \qquad\text{as } r \to \infty.
\end{equation}
On the other hand, by the estimate for $U(r)$ in \eqref{eq:Tauberian_U_estimate} and the definition of $p(x)$ in \eqref{eq:Tauberian_notation}, 
\begin{equation}\label{eq:Tauberian_pdp2}
\frac{2rg^{\prime}(r)}{g(r)}=\frac{2g(r)g^{\prime}(r)\log(r/g(r))}{U(r)}\cdot\frac{rU(r)}{g(r)^2\log(r/g(r))}=2(1+o(1))p(x) \qquad\text{as } r \to \infty.
\end{equation}
Then \eqref{eq:Tauberian_pdp} follows from \eqref{eq:Tauberian_pdp_expression} and \eqref{eq:Tauberian_pdp1}-\eqref{eq:Tauberian_pdp2}. 

\textit{Step 2. Asymptotics of $p(x)$.} 
Our goal of this step is to show
\begin{equation}\label{eq:Tauberian_p_lim}
\lim_{x \to \infty} p(x) = \frac{1}{2}.
\end{equation}
This is equivalent to prove that for each $\varepsilon \in (0,1/10)$, there exists an $x_{\varepsilon}>0$ such that
\begin{equation}\label{eq:Tauberian_p_lim1}
\frac{1}{2} - \varepsilon \le p(x) \le \frac{1}{2} + \varepsilon \qquad \text{for all } x > x_{\varepsilon}.
\end{equation}

Note that for large $x$, $U(e^x)>0$ by \eqref{eq:Tauberian_U_estimate} and hence $p(x)>0$ by its definition. In view of \eqref{eq:Tauberian_pdp}, for each $\varepsilon \in (0,1/10)$, there exists an $x_{\varepsilon}>0$ such that
\begin{equation*}
1 - \varepsilon^{2} - (2+\varepsilon^{2})p(x) \le \frac{p^{\prime}(x)}{p(x)} \le 1 + \varepsilon^{2} - (2-\varepsilon^{2})p(x) \qquad \text{for all } x > x_{\varepsilon}.
\end{equation*}
We proceed the analysis through the following two cases. 

\textit{Case (i). $p(x) > 1/2 + \varepsilon$ for some $x > x_{\varepsilon}$.} 
For such $x$, 
\begin{equation*}
\frac{p'(x)}{p(x)}\le 1 + \varepsilon^{2} - (2-\varepsilon^{2})\left(\frac12 + \varepsilon\right)
= -2\varepsilon + \frac32\varepsilon^{2} + \varepsilon^{3} < -\varepsilon.
\end{equation*}
Thus, integrating from $x$ to $y$ as long as the integration variable satisfies $p(t) \ge 1/2 + \varepsilon$ on $(x,y)$ gives $p(y) \le p(x) e^{-\varepsilon(y-x)}$. In particular, $p$ decays exponentially until it falls below the level $1/2+\varepsilon$. Therefore, by increasing $x_{\varepsilon}$ if necessary, we obtain $p(x) \le 1/2 + \varepsilon$ for all $x > x_{\varepsilon}$.

\textit{Case (ii). $p(x) < 1/2 - \varepsilon$ for some $x > x_{\varepsilon}$.} 
For such $x$,
\begin{equation*}
\frac{p'(x)}{p(x)}\ge 1 - \varepsilon^{2} - (2+\varepsilon^{2})\left(\frac12 - \varepsilon\right)
= 2\varepsilon - \frac{3\varepsilon^2}{2} + \varepsilon^3 > \varepsilon.
\end{equation*}
Hence, $p(y)\ge p(x) e^{\varepsilon(y-x)}$ on the maximal interval containing $x$ on which $p(t) \leq 1/2-\varepsilon$. 
This implies $p$ grows exponentially until it reaches the level $1/2-\varepsilon$. Increasing $x_{\varepsilon}$ if necessary, we conclude that $p(x)\ge 1/2-\varepsilon$ for all $x > x_{\varepsilon}$. 

Combining the analysis for the both cases yields \eqref{eq:Tauberian_p_lim1}.

\textit{Step 3. Lower bound for $(\log p(x)+w(x))^{\prime}$.} 
For $w(x)$ defined in \eqref{eq:Tauberian_notation}, the straightforward calculations give 
\begin{equation}\label{eq:Tauberian_dw}
w^{\prime}(x) = \frac{2e^{x}g^{\prime}(e^x)}{g(e^{x})} - 1 =\frac{2rg^{\prime}(r)}{g(r)} - 1.
\end{equation} 
Note that by \eqref{eq:Tauberian_pdp2} and \eqref{eq:Tauberian_p_lim},
\begin{equation*}
\frac{rg^{\prime}(r)}{g(r)}=\frac{1+o(1)}2 \qquad\text{as } r\to\infty;
\end{equation*}
moreover, $g^{\prime}(r)\to 0$ as $r\to\infty$ (see \eqref{eq:condition_dg_to0}). 
It follows from \eqref{eq:Tauberian_pdp_expression}, \eqref{eq:Tauberian_dw}, and \eqref{eq:Tauberian_pdp1} that 
\begin{equation*}\begin{split}
\left(\log p(x) + w(x)\right)^{\prime}
&= \frac{rU^{\prime}(r)}{U(r)}-\frac{g(r)-rg^{\prime}(r)}{g(r)\log(r/g(r))}\\
&=\frac{1+o(1)}{4\log(r/g(r))}-\frac{1+o(1)}{2\log(r/g(r))}=-\frac{1+o(1)}{4(x-\log g(e^x))} \qquad\text{as } x\to\infty.
\end{split}\end{equation*}
Since $g(r)\le Cr^{1/2}$ (see \eqref{eq:sharp_estimate_g}), we have 
\[
\log g(e^x)\le \frac{1+o(1)}2x\qquad 
\text{as }  x\to\infty.
\]
Therefore, it holds that 
\begin{equation*}
\left(\log p(x) + w(x)\right)^{\prime}\ge -\frac{1+o(1)}{2x}
\qquad\text{as } x\to\infty.
\end{equation*}

\textit{Step 4. Lower bound for the growth of $g$.} 
By Step 3, for each $\varepsilon > 0$, there exists an $x_{\varepsilon}>0$ such that
\begin{equation*}
\left(\log p(x) + w(x)\right)^{\prime} \ge - \frac{1+\varepsilon}{2x} \qquad \text{for } x > x_{\varepsilon}.
\end{equation*}
Integrating over $[x_\varepsilon,x]$ for $x>x_\varepsilon$ yields
\begin{equation*}
\log p(x) + w(x) \ge \log p(x_{\varepsilon}) + w(x_{\varepsilon}) - \frac{1+\varepsilon}{2} \log\frac{x}{x_{\varepsilon}} = - \frac{1+\varepsilon}{2}\log x - C_{\varepsilon} \qquad \text{for } x > x_{\varepsilon}.
\end{equation*}
Since $p(x)\to1/2$ as $x\to\infty$ (see Step 2), choosing a larger $x_{\varepsilon}$ if necessary, one has $\log p(x) < 0$ for $x > x_{\varepsilon}$. Hence
\begin{equation*}
w(x) \ge -\frac{1+\varepsilon}{2}\log x - C_{\varepsilon} \qquad \text{for } x > x_{\varepsilon}.
\end{equation*}
Recall that $w(x)=2\log g(r)-\log r$. It follows that
\begin{equation*}
\begin{split}
\frac{g(r)^{2}}{r} = e^{w(x)} \ge e^{-\frac{1+\varepsilon}{2}\log x - C_{\varepsilon}} = e^{-C_{\varepsilon}} x^{-\frac{1+\varepsilon}{2}} = e^{-C_{\varepsilon}}(\log r)^{-\frac{1+\varepsilon}{2}} \qquad \text{for } r > R_{\varepsilon} := e^{x_{\varepsilon}}.
\end{split}
\end{equation*}
That is,
\begin{equation*}
g(r) \ge e^{-C_{\varepsilon}}r^{\frac 12}(\log r)^{-\frac{1+\varepsilon}{4}} \qquad \text{for } r > R_{\varepsilon}.
\end{equation*} 
Hence the proof of the proposition is completed. 
\end{proof}

Combining the previous proposition with Proposition \ref{prop:lower_bound_log1}, we obtain the following.

\begin{proposition}\label{prop:g_growth_lower}
Assume that $g$ satisfies \eqref{eq:condition_g_sublinear} and \eqref{eq:concave_fb}. If 
\begin{equation}\label{eq:liminf_positive}
\liminf_{r \to \infty} g(r)g^{\prime}(r)\log\frac{r}{g(r)} > 0,
\end{equation}
then 
\begin{equation}\label{eq:lower_bound_log}
\liminf_{r \to \infty} \frac{g(r)}{r^{1/2}(\log r)^{\delta}} > 0 \qquad \text{for every } \delta < -\frac{1}{4}.
\end{equation}
\end{proposition}

\begin{proof}
It follows from Proposition \ref{prop:lower_bound_log1} and \eqref{eq:liminf_positive} that 
\begin{equation*}\begin{split}
&\ V(r) + (1+o(1))g(r)g^{\prime}(r)\log\frac{r}{g(r)} - \frac{1}{2} \int_{a}^{r} \frac{g(t)g^{\prime}(t)}{t}\mathrm{d}t\\
=&\ C +o\left(g(r)g^{\prime}(r)\log\frac{r}{g(r)}\right) \qquad\text{as } r \to \infty.
\end{split}\end{equation*}
This equality is equivalent to the condition \eqref{eq:Tauberian_condition} in Proposition \ref{prop:Tauberian}. Consequently,  \eqref{eq:lower_bound_log} holds.
\end{proof}

\begin{proof}[Proof of Theorems \ref{thm:results_1} and \ref{thm:existence}]
Lemma \ref{lem:concavity} and Propositions \ref{prop:g_refined_growth}, \ref{prop:concave_log}, and \ref{prop:g_growth_lower} imply Theorem \ref{thm:results_1}. Together with the existence theory of \cite{CF_1982,GLS_1952} for axially symmetric cavity flows with sublinear growth at infinity, Theorem \ref{thm:results_1}, which applies to conical fixed boundaries (being both convex and concave), yields Theorem \ref{thm:existence}.
\end{proof}

\medskip\textbf{Acknowledgement.} Y. Li is partially supported by NSFC Grant No. 12526508 and Zhejiang Provincial Natural Science Foundation of China under Grant No. LQN26A010009.  The research of Xie is partially supported by NSFC grants 12571238 and 12426203.

\end{document}